\documentclass[11pt]{amsart}
\usepackage{amssymb,amsmath,amsthm,mathtools}
\usepackage{a4wide}
\usepackage{graphicx}
\usepackage[utf8]{inputenc} % Serve per i caratteri speciali
\usepackage{verbatim}
\usepackage{microtype}
\usepackage{hyperref}
\usepackage{amsfonts} 
\usepackage{latexsym}
\usepackage[font=small,format=hang,labelfont={sf,bf}]{caption}
\usepackage{epsfig}
\usepackage{subfig}
\usepackage{url}
\usepackage{varioref}
\usepackage{bm}
\mathtoolsset{showonlyrefs} %non mostra la roba non citata
\usepackage{tikz}
\usepgflibrary{patterns} % LATEX and plain TEX and pure pgf
\usepgflibrary[patterns] % ConTEXt and pure pgf
\usetikzlibrary{patterns} % LATEX and plain TEX when using Tik Z
\usetikzlibrary[patterns] 
\usepackage{amsfonts} 
\usepackage{geometry}%fa i mergini del pdf
\usepackage{color}
\usepackage{xcolor}
\usepackage{mathrsfs}

\usepackage{latexsym}
\usepackage{bbm}
\usepackage{dsfont}
\numberwithin{equation}{section}
\theoremstyle{plain}
\begingroup
\newtheorem{theorem}{Theorem}[section]
\newtheorem{lemma}[theorem]{Lemma}
\newtheorem{proposition}[theorem]{Proposition}

\newtheorem{definition}[theorem]{Definition}
\newtheorem{remark}[theorem]{Remark}

\endgroup

\definecolor{ddorange}{rgb}{1,0.5,0}
\definecolor{ddcyan}{rgb}{0,0.2,1.0}

\renewcommand{\div}{\mathrm{div}}
\newcommand{\pa}{\partial}
\newcommand{\eps}{\varepsilon}
\newcommand{\numberset}{\mathbb}
\newcommand{\R}{\numberset{R}}

\newcommand{\N}{\numberset{N}}
\newcommand{\Ha}{\mathcal{H}}
\newcommand{\medint}{-\kern -,40cm\int}
\newcommand{\medintinrigo}{-\kern -,34cm\int}
\newcommand{\res}
{\mathop{\hbox{\vrule height 7pt width .5pt depth 0pt
			\vrule height .5pt width 5pt depth 0pt}}\nolimits}
\newcommand{\weakly}{\rightharpoonup}
\newcommand{\weakstar}{\stackrel{*}{\weakly}}   
\title[The anisotropic surface diffusion with elasticity in the plane ]{On variational scheme modeling the anisotropic surface diffusion with elasticity in the plane}

\author[A. Kubin]
{A. kubin}
\address[Andrea Kubin]{
	Jyv\"askyl\"an Yliopisto, Matematiikan ja Tilastotieteen Laitos, Jyv\"askyl\"a, Finland
}
\email[A. Kubin]{andrea.a.kubin@jyu.fi}

\begin{document}

	\begin{abstract}
		In this paper, we prove the existence of classical solutions for the anisotropic surface diffusion with elasticity in the plane using a minimizing movements scheme, provided that the initial set is sufficiently regular. This scheme is inspired by the one introduced by Cahn-Taylor \cite{CaTa} to modeling the surface diffusion. Moreover, we prove that this scheme converges to the global solution of the equation.
	\end{abstract}
	
	\maketitle
	\tableofcontents
	\section{Introduction} 
	In this paper, we investigate the existence of solutions to the anisotropic surface diffusion equation with elasticity in the plane, employing the minimizing movements scheme. 
	
	We provide a brief overview of the physical and mathematical motivation for this equation. In recent years, there has been growing interest in the physics literature in energy functionals that involve a competition between surface interface energy and elastic energy. This interest is driven by the study of interface morphologies influenced by such energies. From a mathematical perspective, the problem is formulated as the analysis of local or global minimizers of a free energy functional, given by the sum of elastic energy and surface energy (typically modeled via isotropic or anisotropic perimeter terms). The static version of this problem has been extensively investigated in both the physical and numerical literature. In the mathematical literature, several works address this topic: \cite{BGZ2015, Bonacini2013,BonCham2002, FoFuLeMo2007, FM2012, GZ2014} present results on existence, regularity, and stability for variational models describing equilibrium configurations in two dimensions, while \cite{Bonacini2015, CS2007} provide results in three dimensions.
	As previously mentioned, our focus is on the dynamic and evolutionary counterpart of such energy models. Before introducing the differential equation we are studying, we recall the Einstein–Nernst equation, as our equation represents a special case of it. This equation describes the evolution of an interface driven by surface mass transport under the influence of a chemical potential $\mu$. In particular, the surface flux of atoms is proportional to the tangential gradient of the chemical potential, and the divergence of this flux corresponds to the rate at which material is either removed from or deposited onto the interface. Throughout the evolution, the volume is conserved, as bulk mass transport can be neglected due to its occurring on a much faster timescale (see \cite{Mullins1963}).
	Thus, the evolution law is
	\begin{equation}\label{EQEINNER}
		V_t= \Delta_\tau \mu_t \text{ on } \pa E_t
	\end{equation}
	where  $V_t$  is the normal velocity, $\Delta_\tau$  is the Laplace–Beltrami operator on $\pa E_t$, and $\mu_t$ is the chemical potential.  The chemical potential $\mu$ is defined as the first variation of the free-energy functional.
	The prototypical free-energy functional we consider is given by:
	\begin{equation}\label{Fintroeng}
		J(F)= \int_{\pa F} \varphi(\nu_F)\, d\mathcal{H}^1+ \frac{1}{2} \int_{\Omega \setminus F} Q(E(u_F))\, dx,
	\end{equation}
	where $\Omega \subset \R^2$ denotes the planar region in which the phenomena of interest occur (e.g., the region occupied by the elastic body), and $F \subset \Omega$ (e.g., represents the void that has formed within the elastic body). 
	As previously mentioned, the minimizers of the functional $F \rightarrow J(F)$ under the volume constraint $\vert F \vert=m$ can be used to describe the equilibrium shapes of voids in elastically stressed solids; see \cite{SMV2004}. We now clarify the various terms appearing in equation \eqref{Fintroeng}. The function $u_F$ represents the elastic equilibrium in $\Omega \setminus F$ subject to the boundary condition $u_{F}= w_0$ on $\pa \Omega$, i.e.,
	\begin{equation}\label{eqelasticoliberointro}
		u_{F} \in \arg \!\min \left\{ \int_{\Omega \setminus F } Q(E(u))\, dx \colon u \in H^1(\Omega \setminus F,\R^2) , \, u |_{\pa \Omega}=w_0    \right\}.
	\end{equation}
	The function $Q$ is the quadratic form defined by $Q(A):= \frac{1}{2} \mathbb{C}A: A$ for all $2 \times 2$-symmetric matrices $A$, where $\mathbb{C}$ is the elasticity tensor. The quantity $E(u_{F})$ denotes the symmetric part of the gradient $\nabla u_{F}$, given by $ E(u_{F})= \frac{\nabla u_{F}+ (\nabla u_{F})^t}{2}$. Finally $\varphi(\nu_F)$ is the anisotropic surface energy density evaluated at the outer unit normal $\nu_F$ to $F$. The anisotropy considered in this work is regular and strictly convex; i.e., $\varphi$ is one-homogeneous, $\varphi \in C^{\infty}(\R^2 \setminus \{0 \})$ and  \begin{equation}\label{unifellintro}
		\exists J>0 \, \colon D^2 \varphi (\nu) \xi \cdot \xi \geq J \vert \xi \vert^2 \quad \forall \nu \in \mathcal{S}^1,\,\xi \in \R^2 \text{ such that } \nu \bot \xi.
	\end{equation}
	The existence and regularity of minimizers of the energy functional $F \mapsto J(F)$ under a volume constraint on $F$ have been studied in various works; see \cite{CJP2013, FFLMi2011} for the two-dimensional case. A relaxation result valid in all dimensions, concerning a variant of the energy \eqref{Fintroeng}, is provided in \cite{BCS2007}.

	The equation studied in this work is derived from the Einstein–Nernst equation (see \eqref{EQEINNER}), under the assumption that $\mu_t$ corresponds to the first variation of the free energy \eqref{Fintroeng}. As a result, we obtain the following system:
	\begin{equation}\label{MAINEQsol}
		\left\{
		\begin{aligned}
			& V_t= \Delta_\tau \big(   \kappa^\varphi_{E_t}-Q (E(u_{E_t}))\big), \text{ on } \pa E_t\\
			& E_0 \text{ initial datum, } \\
			&u_{E_t} \in \arg \!\min \left\{ \int_{\Omega \setminus E_t } Q(E(u))\, dx \colon u \in H^1(\Omega \setminus E_t,\R^2) , \, u |_{\pa \Omega}=w_0    \right\},
		\end{aligned}
		\right.
	\end{equation}      
	where $\kappa_{E_t}^\varphi$ denotes the anisotropic curvature of $\pa E_t$.  The existence of classical solutions to equation \eqref{MAINEQsol} and the asymptotic stability of strictly stable stationary sets are studied in \cite{FJM2018}. In \cite{FJM2020}, the authors investigate the existence and asymptotic stability of solutions in three dimensions for the isotropic surface diffusion equation with elasticity.

	The equation \eqref{MAINEQsol} can be viewed as a nonlocal perturbation of the surface diffusion equation, where the nonlocality arises from the elasticity term. In dimension $n$, the surface diffusion equation takes the form
	\begin{equation}\label{eq:surf-diff}
		\left\{
		\begin{aligned}
			& V_t = \Delta_\tau H_{E_t} \text{ on } \pa E_t, \\
			& E_0 \text{ initial datum, }
		\end{aligned}
		\right.
	\end{equation}
	where $H_{E_t}(x)$ denotes the mean curvature of the hypersurface $\pa E_t$ at the point $x$. The short-time existence of classical solutions to \eqref{eq:surf-diff} was first established in the planar case in \cite{BDR, EG, GigaIto}, and later extended to all dimensions by Escher, Mayer, and Simonett \cite{EMS}, for initial sets with $C^{2,\alpha}$-regularity. Remarkably, the result in \cite{EMS} also applies to immersed surfaces, and the authors prove both global existence and exponential convergence for initial sets sufficiently close to a sphere. In the flat torus $\mathbb{T}^n$ , similar long-time existence and convergence results near stable critical sets have been obtained: for $n=3$ in \cite{AFJM}, and for $n\geq4$ in \cite{DGDKK} and \cite{DianaFuscoMantegazza}. The equation \eqref{eq:surf-diff} can be interpreted as the $H^{-1}$-gradient flow of the area functional; see \cite{TaCa}. This naturally leads to the question of whether a variational approach based on minimizing movements can be used to model the flow. In 1994, Cahn and Taylor \cite{CaTa} proposed such a scheme to describe surface diffusion. The proposed scheme is as follows: given any initial bounded set of finite perimeter $E_0 \subset \R^n$ and a small time step $h>0$, one defines $E_0^h= E_0$ and then constructs  $E_{hk}^h$ for $k = 1,2,\dots$ inductively as a minimizer of the functional
	\begin{equation} \label{eq:def-scheme-intro}
		P(F)+\frac{d_{H^{-1}}(F;E_{h(k-1)}^h)^2}{2h},
	\end{equation}
	where  
	\begin{equation} \label{eq:def-distance-intro}
		d_{H^{-1}}(F;E):=\sup_{\|\nabla_{\pa E} f\|_{L^2(\pa E)}\leq1}\int_{\R^n}   f( \pi_{\pa E}(x))(\chi_F(x)-\chi_E(x))\,dx.
	\end{equation}
	Above, $P(F)$ denotes the De-Giorgi perimeter of the set $F$, $\nabla_{\pa E}$ denotes the tangential gradient, $\chi_E$ the characteristic function of $E$ and $\pi_{\pa E}$ the projection on the boundary $\pa E$. Only recently, however, has it been rigorously shown in \cite{CFJKsd} that this scheme indeed models surface diffusion. In particular \cite{CFJKsd} proves that the scheme produces classical solutions and converges to the classical solution of \eqref{EQEINNER} throughout the full interval of existence in dimension $3$.
	
	Our work focuses on the implementation of the minimizing movements scheme employed in \cite{CFJKsd} for the case of equation \eqref{MAINEQsol}. In the literature, minimizing movement-type schemes have previously been used to model the $H^{-1}$ gradient flow of a variant of the energy \eqref{Fintroeng}, although these schemes differ from the one in \cite{CFJKsd}. Specifically, the variant energy considered includes a curvature regularization term added to the original energy \eqref{Fintroeng}. It is worth noting that such variants have been extensively studied in physical and mathematical literature; see, for example, \cite{AG1989,BHSV2007,DCGPG1992,GJ2002,H1951,RRV2006,SMV2004}.
	For instance, the authors of \cite{FFLM2012} study the $H^{-1}$  gradient flow of the functional
	\begin{equation}
		F \rightarrow  \int_{\pa F} \varphi(\nu_F)\, d\mathcal{H}^1 +\frac{1}{2} \int_{\Omega \setminus F} Q(E(u_F))\, dx + \frac{\varepsilon}{2}\int_{\pa F}\kappa_{F}^2 \, d \mathcal{H}^1,
	\end{equation}
	where $\varepsilon>0$ and $\kappa_{F}$ denotes the curvature. Their analysis focuses on periodic graph models describing the evolution of epitaxially strained elastic films in two dimensions.
	The corresponding flow is governed by the area-preserving evolution equation
	\begin{equation}
		\left\{
		\begin{aligned}
			&V_t= \Delta_\tau \big(   \kappa^\varphi_{E_t}-Q (E(u_{E_t}))-\varepsilon(\Delta_\tau \kappa_{E_t}+\frac{1}{2}\kappa^3_{E_t})\big) \text{ on  } \pa E_t,\\
			& E_0 \text{ initial datum,}
		\end{aligned}
		\right.
	\end{equation}
	They prove a local existence result even when $\varphi$ does not satisfy condition \eqref{unifellintro}. As previously mentioned, their approach is based on the minimizing movements scheme, which differs from that of \cite{CFJKsd}. It is well defined only when the sets have boundaries that can be represented as the graph of a function, unlike the method in \cite{CFJKsd}, which applies to general sets of finite perimeter. Moreover, their approach crucially depends on the curvature regularization term. In fact, all estimates derived in their work are $\varepsilon$-dependent and degenerate as $\varepsilon \rightarrow 0^+$, even when $\varphi$ satisfies \eqref{unifellintro}.
	A similar analysis was carried out in the three-dimensional setting in \cite{FFLM2015}.
	
	The main results of this work are the proof of the existence of a solution to equation \eqref{MAINEQsol}, and the prove of the consistency of the minimizing movements scheme. We briefly outline the strategy of the proof. The first step is to introduce a constrained elastic equilibrium by modifying the original problem \eqref{eqelasticoliberointro}. To this end, we fix two constants $K_{el}>0$ and $h>0$, where $h$ plays the role of a time discretization parameter, as in formula \eqref{eq:def-scheme-intro}.
	Given a set $F \subset \Omega$,  we consider the following constrained minimization problem: 
	\begin{equation}\label{minelvincintro}
		\min \left\{ \int_{\Omega \setminus F} Q(E(u))\, dx \colon \| u \|_{C^{3,\frac{1}{4}}(\Omega)} \leq K_{el} ,\, \| \nabla^4 u \|_{C^{0,\frac{1}{4}}(\Omega)} \leq \frac{K_{el}}{h^{\frac{1}{4}}},\, u |_{\pa \Omega}= \omega_0   \right\}.
	\end{equation}
	We denote by $ u_{F}^{K_{el},h}$ a minimizer of this problem. Accordingly, the constrained elastic energy is defined as
	\begin{equation}\label{enelastvincintro}
		\mathcal{E}(E(u_F^{K_{el},h})):= \int_{\Omega \setminus F} Q(E(u_F^{K_{el},h}))\, dx .
	\end{equation}
	As a second step, we implemented the minimizing movement algorithm as described in \cite{CFJKsd}.
	Let $E_0 \Subset \Omega$ be an open, connected set of class $C^5$ , which serves as the initial datum. We fix a small parameter $\beta >0$, set  $E_0^{h,\beta}= E_0$ and  define $E_{hk}^{h,\beta}$ as a minimizer of the following incremental minimization problem:
	\begin{equation}  \label{eq:def-scheme-intro-2}
		\inf \Big{\{} \int_{\pa F} \varphi(\nu_F)\, d\mathcal{H}^1+\mathcal{E}(E(u_F^{K_{el},h}))+\frac{d_{H^{-1}}(F;E_{h(k-1)}^{h,\beta})^2}{2h} :  F \Delta E_{h(k-1)}^{h,\beta} \subset \mathcal{I}_\beta(\pa  E_{h(k-1)}^{h,\beta}) \Big{\}} ,
	\end{equation}
	where $\mathcal{I}_\beta(\Gamma)$  denotes the tubular neighborhood of a set $\Gamma \subset \R^2$ (see \eqref{intornotubolare} for its definition). Due to the constraint condition $ F \Delta E_{hk}^{h,\beta} \subset \mathcal{I}_\beta(\pa  E_{h(k-1)}^{h,\beta})$, the existence of a minimizer for the above problem follows readily from the direct methods of the Calculus of Variations.
	Using quantitative geometric estimates we show that any minimizer $ E_{k}^{\beta,h}$ of  \eqref{eq:def-scheme-intro-2} satisfies 
	\begin{equation}  \label{eq:intro-3}
		E_{k}^{h,\beta} \Delta E_{h(k-1)}^{h,\beta} \subset \mathcal{I}_{\frac{\beta}{2}}(\pa  E_{h(k-1)}^{h,\beta}),
	\end{equation}
	when $h$ is sufficiently small, provided $\pa E_{h(k-1)}^{h,\beta}$ is sufficiently regular. This shows that the additional constraint in \eqref{eq:def-scheme-intro-2} it is not touched. We define  $ E^{h,\beta}_t= E^{h,\beta}_k$ for $ t \in [kh,(k+1)h)$ for all $ k \in \N$. The family $\{E_t^{h,\beta}\}_{t \geq 0}$ is called a constrained discrete flat flow with initial datum $E_0$ and time step $h$ (see definition \ref{12092023def1}).  
	
	Our main result is the short time regularity and the consistency of the minimizing movement
	scheme defined above.
	\begin{theorem}\label{teoremaintro}
		There exist constants $K_{el},T,\beta_0,\sigma_1$ with the following property: for every $ \beta< \beta_0$ there exists $h_0>0$ such that the family $\{E_t^{h,\beta}\}_{t \in [0,T]}$ satisfies 
		\begin{equation}
			\pa E^{h,\beta}_t= \{x+ f^{h,\beta}(t,x)\nu_{E_0}(x)\colon x \in \pa E_0    \}, \, \| f^{h,\beta} \|_{H^4(\pa E_0)} \leq C_0,\, \| f^{h,\beta} \|_{L^\infty(\pa E_0)} \leq \sigma_1,
		\end{equation}
		for all $t \in [0,T]$ and $ 0 < h \leq h_0$. 
		The function $f^{h,\beta}$ converge in 
		$L^\infty([0,T], H^4(\pa E_0))$ to a function $f^{\beta}$, such that the family $ (E^\beta(t))_{t \in [0,T]}$ have the properties
		\begin{equation}
			\pa E^\beta_t= \{  x + f^\beta(t,x) \nu_{E_0}(x) : x \in \pa E_0 \},
		\end{equation}
		and $ (E^\beta_t)_{t \in [0,T]}$ is a solution to \eqref{MAINEQsol} with initial datum $E_0$ on the interval $[0,T]$.
	\end{theorem}
	We briefly outline the strategy used to prove the main theorem. The key ingredients for establishing the main result—whose proof is presented in the final section of the paper (see Section \ref{sezionefinale})—are the preliminary estimates (see Section \ref{stimaprelininari}) and an iteration argument (see Section \ref{iterazionesezione}).
	
	The main goal of the preliminary estimates is to show that the minimizers of each incremental problem \eqref{eq:def-scheme-intro-2} satisfy suitable regularity estimates. By regularity estimates, we mean that the minimizer $E^{h,\beta}_{hk}$ of problem \eqref{eq:def-scheme-intro-2} satisfies the following properties:
	\begin{itemize}
		\item  the boundary of $E^{h,\beta}_{hk}$ can be written as a normal graph over the previous step, i.e.,
		\begin{equation}\label{itemizeintro1}
			\pa E_{hk}^{h,\beta}= \{  x+ \psi_k(x)\nu_{E_{h(k-1)}^{h,\beta}}(x): x \in \pa E_{hk}^{h,\beta}\}, \text{ with } \psi_k \in C^1(\pa E_{h(k-1)}^{h,\beta}),  
		\end{equation}
		\item the boundary of $E^{h,\beta}_{hk}$ does not intersect the constraint, i.e., \begin{equation}\label{itemizeintro2}
			\pa E_{hk}^{h,\beta} \Subset \mathcal{I}_{\beta}(\pa E_{h(k-1)}^{h,\beta}),
		\end{equation}
		\item the function $\psi_k$ satisfies the bounds 
		\begin{equation}\label{itemizeintro3}
			\| \psi_k \|_{L^2(\pa E_{hk}^{h,\beta})} \leq C h,\,  \| \psi_k \|_{H^4(\pa E_{hk}^{h,\beta})} \leq C , \, 
			\| \kappa_{ E_{hk}^{h,\beta}} \|_{H^3(\pa  E_{hk}^{h,\beta} )} \leq C h^{-\frac{1}{4}}, 
		\end{equation}
	\end{itemize}
	where the constant $C$ depends only on the $H^2$-norm of the curvature of $E_{h(k-1)}^{h,\beta}$. We explain here how to obtain the estimates for the case $k=1$, since the subsequent steps will be proved by induction and iteration. The idea is to show that the minimizer $E_h^{h,\beta}$ is a $\Lambda$-minimizer of the $\varphi$-perimeter, for some constant $\Lambda$ independent  of $h$, but depending only on the $H^2$-norm of the curvature of $E_0$ (see Lemma \ref{lalalambdamin}). This allows us to apply a variant of the $\varepsilon$-regularity  theorem for 
	$\Lambda$-minimizer of the $\varphi$-perimeter, namely Lemma \ref{propdadim}, from which we deduce the existence of a function $\psi_1 : \pa E_0 \rightarrow \R$ such that
	\begin{equation}
		\pa  E_h^{h,\beta}= \{ x+ \psi_1(x)\nu_{E_0}(x)\colon x \in \pa E_0  \}, \text{ with } \psi \in C^1(\pa E_0).
	\end{equation}
	To carry out all of this, we need to show that the discrete velocity in $H^{-1}$ is bounded, namely
	\begin{equation}
		\frac{d_{H^{-1}}(E_{h}^{h,\beta}, E_0)}{h}\leq C,
	\end{equation}
	where the constant $C$ depends only on the $H^2$-norm of the curvature of $E_0$. This inequality follows from the minimality of $E_{h}^{h,\beta}$ and the regularity of $E_0$; see Lemmas \ref{lego123og} and \ref{lalalambdamin}. Furthermore, using Lemma \ref{lego123og}, we obtain that $ \pa E_{h}^{h,\beta} \Subset \mathcal{I}_{\beta}(\pa E_0) $, so the constraint $ \pa \mathcal{I}_{\beta}(\pa E_0)$ is never touched. Thanks to these results, we can compute the first variation of the energy
	\begin{equation}
		F \rightarrow \int_{\pa F} \varphi(\nu_F)\, d\mathcal{H}^1+\mathcal{E}(E(u_F^{K_{el},h}))+\frac{d_{H^{-1}}(F;E_0)^2}{2h}
	\end{equation}
	at the minimizer $E_h^{h,\beta}$. This leads to a differential equation for the unknown function $\psi_1$; see equation \eqref{LEEULEROCOMB}. Using the Euler–Lagrange equation, we also obtain another estimate for the discrete velocity, this time in $L^2$:
	\begin{equation}
		\frac{\| \psi_1 \|_{L^2(\pa E_0)}}{h} \leq C,
	\end{equation}
	where the constant $C$ depends only on the $H^2$-norm of the curvature of $E_0$. Moreover, we obtain a bound in the $H^4$-norm, namely $\| \psi_1 \|_{H^4(\pa E_0)} \leq C$, while the curvature satisfies $ \| \kappa_{E_h^{h,\beta}} \|_{H^3(\pa E_0)} \leq C h^{-\frac{1}{4}}$, where the constant $C$ depends only on the $H^2$-norm of the curvature of $E_0$. All of this is proved in Theorem \ref{MainThm375}.
	
	In Section \ref{iterazionesezione}, the goal is to establish a connection between the steps $k-1,k,k+1$ in such a way that the validity of formulas \eqref{itemizeintro1}, \eqref{itemizeintro2}, and \eqref{itemizeintro3} can be ensured for every admissible $k$. The main idea is to relate the Euler–Lagrange equation satisfied by the set $E_{h(k+1)}^{h,\beta}$ with the one satisfied by the set $E_{hk}^{h,\beta}$. To achieve this, we use the expansion of the $\varphi$-curvature given in formula \eqref{L'ESPANSOIONE}. Indeed, the Euler–Lagrange equation for $E_{h(k+1)}^{h,\beta}$ involves the $\varphi$-curvature of the set $E_{hk}^{h,\beta}$, which also appears in the Euler–Lagrange equation satisfied by $E_{hk}^{h,\beta}$ itself. Therefore, by substituting the latter equation into the former, we derive the desired iteration—see Lemma \ref{spaghettiska} and Proposition \ref{propdiiterazione}.
	
	In Section \ref{sezionefinale}, we provide the proof of Theorem \ref{teoremaintro}. This proof follows from Theorems \ref{thmausilmain} and \ref{locexisitsol}. In the first of these, we show that for any fixed $K_{el}>0$, there exists a time $T>0$ such that the family $\{E_t^{h,\beta}\}_{t \in [0,T]}$ satisfies 
	\begin{equation}
		\pa E^{h,\beta}_t= \{x+ f^{h,\beta}(t,x)\nu_{E_0}(x)\colon x \in \pa E_0    \}, \, \| f^{h,\beta} \|_{H^4(\pa E_0)} \leq C_0,\, \| f^{h,\beta} \|_{L^\infty(\pa E_0)} \leq \sigma_1,
	\end{equation}
	for all $t \in [0,T]$. The function
	$f^{h,\beta}$ converge as $h \rightarrow 0^+$ in 
	$L^\infty([0,T], H^4(\pa E_0))$ to a function $f^{\beta}$, with 
	$ f^\beta\in \mathrm{Lip}([0,T],L^2(\pa E_0))$, such that the family $ (E^\beta(t))_{t \in [0,T]}$ satisfies
	\begin{equation}
		\pa E^\beta_t= \{  x + f^\beta(t,x) \nu_{E_0}(x) : x \in \pa E_0 \},
	\end{equation}
	and
	\begin{equation}\label{HOlfinaleINTRO}
		\| f^\beta(t,\cdot) \|_{C^{3,\frac{1}{4}}(\pa E_0)} \leq C t^\frac{1}{21} 
	\end{equation}
	where $C=C(K_{el})$. Formula \eqref{HOlfinaleINTRO} will be sufficient to prove the existence of classical solutions. Indeed, by fixing a sufficiently large $K_{el}$, one can show that there exists a small time $T>0$ such that the constraint $K_{el}$  is not active for the minimizer of the constrained elasticity problem associated with $E^\beta_t$, for every $t \in [0,T]$. Therefore, thanks to the regularity of $\pa E_t^\beta$, this minimizer coincides with the one for the unconstrained elasticity problem \eqref{eqelasticoliberointro}. This will allow us to prove that the family
	$\{ E^\beta_t\}_{t \in [0,T]}$ satisfies the equation \eqref{MAINEQsol}.
	
	In the final section, namely Section \ref{CONVGLOBSOL}, we prove that the minimizing movements scheme converges to the solution of problem \eqref{MAINEQsol} throughout the entire interval of existence.
	
	The paper is organized as follows.
	In Section \ref{NOtazioniLAvoro}, we introduce the notation used throughout the paper, along with some useful formulas, the functional spaces involved, and interpolation inequalities.
	In Section \ref{sezionesettingpro}, we define the function $d_{H^{-1}}(F,E)$, discuss some of its properties including the computation of its first variation, introduce both the free and constrained elasticity problems, and finally present the minimizing movement scheme used in the analysis.
	In Section \ref{stimaprelininari}, we prove the $\Lambda$-minimality property for the minimizer of the incremental problem and establish a regularity estimate for the heightfunction.
	Section \ref{iterazionesezione} is devoted to the proof of the iteration argument.
	In Section \ref{sezionefinale}, we prove Theorem \ref{teoremaintro}.  Finally, in Section \ref{CONVGLOBSOL}, we prove the convergence to the global solution of equation \eqref{MAINEQsol}.

	\section{Notation of the paper and useful formulas}\label{NOtazioniLAvoro}
	In this paper, we work in the $2$-dimensional Euclidean space $\R^2$. We denote with $ \{e_1, e_2\}$ the canonical basis of $\R^2$, by $\vert \cdot \vert $ the Euclidean norm, and by $\cdot$ the inner product in $\R^2$. Let $r>0$ we set $B_r(x)= \{ y \in \R^2 : \vert x- y \vert < r \}$ when $x=0$, we simply write $B_r:=B_r(0)$. For every $A\subset \R^2$ we denote by ${\rm cl}( A)$ ($ {\rm int} (A)$) its topological closure (respectively its topological interior) with respect to the Euclidean topology. Given $A \subset \R^2$ and $x \in \R^2$, we denote by $ \mathrm{dist}(x,A)$ the distance between $x$ and $A$. The Lebesgue measure of a Borel set $A\subset \R^2$ is denoted by $\vert A \vert$. We denote by $\mathcal{H}^1$ the $1$-dimensional Hausdorff measure and by $\mathrm{dist}_{\mathcal{H}}$ the Hausdorff distance between sets. 
	In what follows we denote with $\varphi$ a regular strictly convex norm; i.e., $\varphi \in C^{\infty}(\R^2 \setminus \{0 \})$ and  \begin{equation}\label{unifell}
		\exists J>0 \, \colon D^2 \varphi (\nu) \xi \cdot \xi \geq J \vert \xi \vert^2 \quad \forall \nu \in \mathcal{S}^1,\,\xi \in \R^2 \text{ such that } \nu \bot \xi.
	\end{equation}
	We denote by $m_\varphi,\, M_\varphi$ the constants
	\begin{equation}\label{mMvarphi}
		m_\varphi:= \min_{\vert \nu \vert=1} \varphi(\nu),\quad M_{\varphi}:= \max_{\vert \nu \vert=1} \varphi(\nu).
	\end{equation} 
	The dual norm $\varphi^0$ is defined as $\varphi^0(\xi)= \sup_{\eta \in \R^2 \setminus \{0\}} \frac{\xi \cdot \eta}{\varphi(\eta)}$. Given $a,b \in \R^2$, we denote by $a \otimes b : \R^2 \rightarrow \R^2$ the linear map defined as $ a\otimes b (x):= (x \cdot b)a$. We denote by $\R^{2 \times 2}$ the space of the $2 \times 2$ matrices. Given $A \subset \R^2$ we denote by $A^c= \R^2 \setminus A$. Given $P,C \in \R^{2 \times 2}$, we set $P : C= \sum_{i,j=1}^2 p_{ij}c_{ij}$. The standard gradient in $\R^2$ is denoted by $\nabla$ and the Laplace operator in $\R^2$ is denoted by $\Delta_{\R^2}$. Throughout the paper, we write $C(*,\cdots,*)$ to indicate a generic positive constant that depends only on $*,\cdots,*$ and that may change from line to line.
	%%%%%%%%%%%%%%%%%%%%%%%%%%%%%
	%%%%%%%%%%%%%%%%%%%%%%%%%%%%%
	%%%%%%%%%%%%%%%%%%%%%%%%%%%%%
	%%%%%%%%%%%%%%%%%%%%%%%%%%%%%
	\subsection{Regular sets and useful formulas}
	Let $E \subset \R^2$ be a bounded open set of class $C^{2}$. The derivative of a function $f$ or of a vector field $X$ along $\pa E$ is denoted by $\pa_{\tau} f$ and $ \pa_{\tau} X$, respectively. In cases of ambiguity, we use $ \pa_{\pa E} f$ and $ \pa_{\pa E} X$. The Laplace–Beltrami operator on $\pa E$ is denoted by $\pa_{\tau}^2 $ (or $ \Delta_{\tau}$) and the tangential divergence on $\pa E$ is denoted by
	$\div_{\tau}$. If necessary for clarity, we also write these as $\pa_{\pa E}^2$ or ($ \Delta_{\pa E}$) and $\div_{\pa E}$.  We recall that the second fundamental form $B_E:\pa E\mapsto\R^{2\times2}$ and the curvature $\kappa_E:\pa E\to\R$ are given by
	$$
	\kappa_E=\div_\tau\nu_E \qquad B_E=\kappa_E\,\tau_E\otimes\tau_E,
	$$
	where $ \nu_E : \pa E \rightarrow \R^2$ is the outer normal vector field on $\pa E$ and $\tau_E:\pa E\to\R^2$ is the tangent vector field on $\pa E$, obtained by rotating $\nu_E$ by $\frac{\pi}{2}$ clockwise. We denote the tangential gradient on $\pa E$ by $\nabla_{\tau}$ (or $ \nabla_{\pa E} $), so that $\nabla_{\tau} f= \pa_{\tau} f \tau_E= \nabla f- (\nabla f \cdot \nu_E) \nu_E$ for a function $f$. Let $A \subset \R^2$ and given $\delta>0$, define the tubular neighborhood
	\begin{equation}\label{intornotubolare}
		\mathcal{I}_{\delta}(A):=\{x\in\R^2:\,\text{dist}(x, A)<\delta\}.  
	\end{equation}
	We define the signed distance function to $\pa E$ by
	\begin{equation}\label{distsd}
		d_E(x):=
		\left\{
		\begin{aligned}
			& \mathrm{dist}(x, \pa E) & \text{ for } x \in \R^2 \setminus E \, , \\
			& -\mathrm{dist}(x,\pa E) & \text{ for }  x \in E \, . 
		\end{aligned}
		\right.
	\end{equation}
	Let $E \subset \R^2$ be a open and bounded set of class $C^2$. We define
	\begin{equation}
		\sigma_E:= \frac{1}{2\| \kappa_E \|_{L^\infty(\pa E)}},
	\end{equation}
	It is known (see \cite[Chapter 14.6]{GilbargTrudinger1977}) that $d_E \in C^2(\mathcal{I}_{\sigma_E}(\pa E))$. The projection onto $\pa E$ is define for all $x \in \R^2$ where exists $ \nabla d_E (x)$, and is denoted by $ \pi_{\pa E}(x)$. For all $ x\in I_{\sigma_{E}}(\pa E)$ the projection satisfies
	$$ x=\pi_{\pa E}(x)+ d_E(x) \nabla d_{E}(x).$$  As shown in \cite[formula (2.31)]{JN}, for all $ x \in  \mathcal{I}_{\sigma_{E}}(\partial E)$, it holds that
	\begin{equation}\label{formunoquattro}
		\begin{split}
			\nabla \pi_{\pa E}&(x)\\
			&=I -\nu_{E} \circ \pi_{\partial E}(x) \otimes \nu_{ E} \circ \pi_{\pa E}(x)-d_E(x) (B_{E} \circ \pi_{\partial E}(x) )(I+ d_E(x) B_E \circ \pi_{\partial E}(x))^{-1}.
		\end{split}
	\end{equation} 
	\begin{definition}\label{ngrafico}
		Let $E \subset \R^2$ be an open bounded set of class $C^2$ and let $0<\sigma \leq  \sigma_E$. Let $F \subset \R^2$ be another open bounded set. We say that $\pa F$ is a normal graph over $\pa E$ if there exists a function $\psi: \pa E \rightarrow [-\sigma,\sigma]$, called the height function, such that 
		\begin{equation}
			\pa F= \{ x + \psi(x)\nu_{E}(x) \colon x \in \pa E\} \text{ and } E \Delta F \subset {\rm cl} \big( I_\sigma(\pa E) \big).
		\end{equation}
	\end{definition}
	Let $E, F$ as in the above definition with $\pa F = \{ x +\psi(x)\nu_E(x)\colon x \in \pa E\}$, and let $f  \in C^1(\pa E)$. Then, for all $y \in \pa F$,
	\begin{equation}\label{derivatafpi}
		\nabla_{\pa F}(f \circ \pi_{\pa E})(y)= \nabla_{\pa E} f (\pi_{\pa E}(y)) \nabla_{\pa F} \pi_{\pa E}(y), 
	\end{equation}
	where
	\begin{equation}\label{asdasd12??}
		\nabla_{\pa F} \pi_{\pa E}(y)= \nabla \pi_{\pa E}(y)-\nabla \pi_{\pa E}(y) \nu_F(y) \otimes \nu_F(y).
	\end{equation}
	If $\psi \in C^1(\pa E)$,  then the following formulas hold (see \cite[formulas (2.5), (2.6), (2.7)]{FJM2018}). For $x\in\pa E$:
	\begin{equation}\label{tangnuF}
		\tau_F(x+ \psi(x)\nu_F(x))= \frac{ (1+ \psi(x)\kappa_E(x)) \tau_E(x)+ \pa_{\tau} \psi (x) \nu_E(x)  }{\sqrt{ (1+\psi(x)\kappa_E(x))^2+ \vert \nabla_{\tau } \psi (x) \vert^2   }} 
	\end{equation}
	and
	\begin{equation}\label{nu_Fexp}
		\nu_F(x+ \psi(x) \nu_E(x))=\frac{-\nabla_{\tau }\psi(x)+(1+ \psi(x)\kappa_{E}(x))\nu_E(x)}{\sqrt{ (1+ \psi(x)\kappa_{E}(x))^2+ \vert \nabla_{\tau } \psi(x) \vert^2}}.
	\end{equation}
	If $\psi \in C^2(\pa E)$, then the curvature expands as (see \cite[formulas (2.7)]{FJM2018}, \cite[Lemma 2.5]{CFJKsd}):
	\begin{equation}\label{Cdacit}
		\kappa_F(x+ \psi(x)\nu_E(x))= - \Delta_{\tau} \psi(x)+ \kappa_E(x)+ R_0(x), \, \, x \in \pa E
	\end{equation}
	where the error term $R_0$ is given by
	\begin{equation}\label{Rdacit}
		R_0 = a_0(\psi, \pa_{\tau} \psi, \kappa_E)+ a_1(\psi \kappa_E, \pa_{\tau} \psi) \Delta_{\tau} \psi+ a_2(\psi \kappa_E,\pa_\tau \psi) \pa_{\tau}(\psi \kappa_E)
	\end{equation}
	with $a_0,a_1,a_2$ smooth functions satisfying $a_0(0,0,\cdot)= a_1(0,0)=a_2(0,0)=0$.
	\begin{lemma}\label{cappelloliberoL}
		Let $E \subset \R^2$ be a open and bounded set of class $C^2$ and let $F \subset \R^2$ be open and bounded set of class $C^1$ such that $\pa F $ is a normal graph over $\pa E$ given by $ \pa F= \{ x+ \psi(x)\nu_E(x)\colon x \in \pa E\}$. Let $g \in C^1(\pa F)$, and define $\hat{g} : \pa E \rightarrow \R$ by $ \hat{g}(x)= g(x+ \psi(x)\nu_E(x))$.
		Then, \begin{equation}\label{cappelloliberoF}
			\int_{\pa F} \vert \nabla_{\pa F} g \vert^2 \, d \mathcal{H}^1= \int_{\pa E} \frac{\vert \nabla_{\pa E} \hat{g} \vert^2}{\sqrt{ (1+ \psi \kappa_E)^2+ \vert \nabla_{\pa E} \psi \vert^2}}\, d \mathcal{H}^1.
		\end{equation}
	\end{lemma}
	\begin{proof}
		Let $x \in \pa E$ and let $y \in \pa F$ such that $ \pi_{\pa E}(y)=x$. Hence $d_E(y)= \psi(x)$. Without loss of generality, we may assume that $\tau_E(x)=(1,0)$ and $\nu_E(x)=(0,1)$, and we write $ \nu_F(y)= (\nu_1,\nu_2)$.
		Using formula
		\eqref{formunoquattro} we get
		\begin{equation}
			\begin{split}
				\nabla \pi_{\pa E}(y)=&  \begin{bmatrix}
					1 & 0 \\
					0 & 1 
				\end{bmatrix} - 
				\begin{bmatrix}
					0 & 0 \\
					0 & 1 
				\end{bmatrix} -
				d_E(y) \kappa_E(x) \begin{bmatrix}
					1 & 0 \\
					0 & 0 
				\end{bmatrix} \begin{bmatrix}
					1+ d_E(y) \kappa_E(x)& 0 \\
					0 & 0 
				\end{bmatrix}^{-1} \\
				=& \begin{bmatrix}
					1 & 0 \\
					0 & 0 
				\end{bmatrix} - \frac{\psi(x)\kappa_E(x)}{1+ \psi(x)\kappa_E(x)} \begin{bmatrix}
					1 & 0 \\
					0 & 0 
				\end{bmatrix} = \begin{bmatrix}
					\frac{1}{1+\psi(x)\kappa_E(x)} & 0 \\
					0 & 0 
				\end{bmatrix}. 
			\end{split}
		\end{equation}
		Therefore, by \eqref{asdasd12??}, we have
		\begin{equation}
			\nabla_{\pa F} \pi_{\pa E}(y)= \begin{bmatrix}
				\frac{1-\nu_1 \nu_1}{1+ \psi(x)\kappa_E(x)}& - \frac{\nu_1 \nu_2}{1+ \psi(x)\kappa_E(x)} \\
				0 & 0 
			\end{bmatrix}.
		\end{equation}
		Using the formula above, we deduce that   
		\begin{equation}
			\vert \tau_E(x) \nabla_{\pa F} \pi_{\pa E} (y) \vert^2= \frac{\vert \nu_E(x) \cdot \nu_F(y)\vert^2}{(1+ d_E(y) \kappa_E(x))^2} \quad \forall (x,y) \in \pa E \times \pa F: \, x= \pi_{\pa E}(y).
		\end{equation}
		Hence, from the previous expression, we obtain \begin{equation}\label{forsesiusa}
			\begin{split}
				\int_{\pa F} \vert \nabla_{\pa F} g \vert^2 d \mathcal{H}^1& = \int_{\pa F} \vert \nabla_{\pa F} (\hat{g} \circ \pi_{\pa E}) \vert^2 d \mathcal{H}^1\\
				&= \int_{\pa F} \vert \nabla_{\pa E } \hat{g}(\pi_{\pa E}(y)) \nabla_{\pa F} \pi_{\pa E}(y) \vert^2 d \mathcal{H}^1_y\\
				&= \int_{\pa F} \frac{ \vert\nabla_{\pa E } \hat{g}(\pi_{\pa E}(x)) \vert^2 \vert \nu_{E}(\pi_{\pa E}(x)) \cdot \nu_F(x)\vert^2}{\vert 1+ d_E(x) \kappa_E(\pi_{\pa E}(x))\vert^2} d\mathcal{H}^1_x.
			\end{split}
		\end{equation}
		Using formula \eqref{nu_Fexp}, we have that for all $y \in \pa F$
		\begin{equation}\label{normali}
			\nu_E(\pi_{\pa E}(y))\cdot\nu_F(y)=\frac{1+ d_E(y)\kappa_{E}(\pi_{\pa E}(y))}{\sqrt{ (1+ \psi(\pi_{\pa E}(y)))\kappa_{E}(\pi_{\pa E}(y)))^2+ \vert (\nabla_{\pa E} \psi)(\pi_{\pa E}(y)) \vert^2}}.
		\end{equation}
		Let us define $\Psi :\pa E \rightarrow \pa F$ as $\Psi(x):= x + \psi(x) \nu_{E}(x)$. Recalling that the tangential Jacobian of $\Psi$ is $$ J_{\tau} \Psi(x) = \sqrt{ (1+ \psi(x)\kappa_{E}(x))^2+ \vert (\nabla_{\pa E} \psi)(x) \vert^2}, $$ we deduce formula \eqref{cappelloliberoF}, from  \eqref{forsesiusa} and \eqref{normali}, indeed
		\begin{equation}\label{forse'giusta}
			\begin{split}
				\int_{\pa F} \vert \nabla_{\pa F} g \vert^2 d \mathcal{H}^1&= \int_{\pa F} \frac{ \vert\nabla_{\pa E } \hat{g}(\pi_{\pa E}(y)) \vert^2 \vert \nu_{E}(\pi_{\pa E}(y)) \cdot \nu_F(y)\vert^2}{\vert 1+ d_E(y) \kappa_E(\pi_{\pa E}(y))\vert^2} d\mathcal{H}^1_y  \\
				&= \int_{\Psi(\pa E) } \frac{\vert\nabla_{\pa E } \hat{g}(\pi_{\pa E}(y)) \vert^2}{(1+ \psi(\pi_{\pa E}(y)))\kappa_{E}(\pi_{\pa E}(y)))^2+ \vert (\nabla_{\pa E} \psi)(\pi_{\pa E}(y)) \vert^2}d \mathcal{H}^1_y\\
				& = \int_{\pa E} \frac{\vert\nabla_{\pa E } \hat{g}(x) \vert^2}{\sqrt{(1+ \psi(x))\kappa_{E}(x)^2+ \vert (\nabla_{\pa E} \psi)(x) \vert^2}} d\mathcal{H}^1_x.
			\end{split}
		\end{equation}
	\end{proof}
	\subsection{Spaces of functions}
	In what follows, we denote by $\Omega\subset \R^2$ an open and bounded set of class $C^5$. Let $E_0 \Subset \Omega$ be open and connected set of class $C^5$ such that $\vert E_0 \vert=1$.
	We denote by $\sigma_0$ a constant such that
	\begin{equation}\label{Laconstantesigmazerofin}
		\sigma_0 < \min\{\sigma_{E_0}, \mathrm{dist}_{\mathcal{H}}(\pa E_0, \pa \Omega)\}.  
	\end{equation}
	Given $1\leq k\leq 5$, $ \alpha \in [0,1]$, and $K>0$ we define
	\begin{align}\label{22082024pom2}
		\mathfrak{C}^{k,\alpha}_{K,\sigma_0}(E_0):=
		\big\{E\subset\R^2: \,\, E \Delta E_0 & \subset {\rm cl} \big( \mathcal{I}_{\sigma_0}(\pa E_0) \big),\,\pa E=\{y+\varphi_E(y)\nu_{E_0}(y)\!:\,y\in\pa E_0\},\cr
		& \|\varphi_E\|_{L^\infty(\pa E_0)}\leq \sigma_0,\,\|\varphi_E\|_{C^{k,\alpha}(\pa E_0)}\leq K\big\}.
	\end{align}
	For every $k \in \{1,\cdots,5\}$, we define the set $\mathfrak{H}^{k,\alpha}_{K,\sigma_0}(E_0)$ in the same way of $ \mathfrak{C}_{K,\sigma_0}^{k,\alpha}(E_0)$ by replacing $\|\varphi_E\|_{C^{k,\alpha}(\pa E_0)}$ with  $\|\varphi_E\|_{H^k(\pa E_0)}$. Let $\{ E_n \}_{n \in \N}$ and $ E$ be such that $ E_n \in \mathfrak{C}_{K,\sigma_0}^{k,\alpha}(E_0)$ (respectively $\mathfrak{H}^{k,\alpha}_K(E_0)$) for all $n \in \N$. We say that $E_n \rightarrow E $ in $\mathfrak{C}_{K,\sigma_0}^{k,\alpha}(E_0)$ (respectively in $\mathfrak{H}_{K,\sigma_0}^{k,\alpha}(E_0)$) if $\varphi_{E_n}$ is uniformly bounded by $\sigma_0$ in $L^{\infty}(\pa E_0)$ and it is a Cauchy sequence in $C^{k,\alpha}(\pa E_0)$ (respectively $H^{k}(\pa E_0)$).\\
	Let $ F \subset \Omega$ be an open set. Given $ k \in \N$, $ \alpha \in [0,1]$, and $M>0$ we define
	\begin{equation}
		\mathfrak{C}^{k,\alpha}_M(F,\R^2):= \left\{ f \in C^{k,\alpha}(F,\, \R^2)\colon \| f \|_{C^{k,\alpha}(F)} \leq M \right\}.
	\end{equation}

	\subsection{Sets of finite $\varphi$-perimeter and anisotropic curvature} Let $E \subset \R^2$ be a Borel set. We define the De Giorgi $\varphi$-perimeter of $E$ as 
	\begin{equation}
		P_\varphi(E)= \sup \left\{ \int_{E} \div X dx \colon\, X \in C^1_c(\R^2,\R^2),\, \sup_{x \in \R^2} \varphi^0(X) \leq 1   \right\}.
	\end{equation}
	When $\varphi(\cdot)= \vert \cdot \vert$ (the Euclidean norm), we write $P(E)$ instead of $P_{\vert \,\, \vert }(\cdot)$. We say that a Borel set $E \subset \R^2$ has finite perimeter if $P(E) < + \infty$. Given the assumptions we made on the function $\varphi$, it is easy to verify that $$P_{\varphi}(E)< + \infty \iff P(E)< +\infty. $$  
	For every set $E  \subset \R^2$ with finite perimeter, the set $\partial^* E \subset \R^2$ identifies the reduced boundary of $E$ and the Borel measurable map $\nu_{E}: \partial^* E \rightarrow \mathbb{R}^2$ the measure theoretic outer normal vector field (see, for instance, \cite[Definition 3.53]{AmbFuscPall} for the definitions of these objects). By De Giorgi's structure theorem (see, for instance, \cite[Theorem 3.59]{AmbFuscPall}, \cite[Theorem 15.19]{Maggibook}), for every set  $E \subset \R^2$ of finite perimeter, we have that $$P(E)=\mathcal{H}^{1}(\partial^*E).$$ Let $E \subset \R^2$ be a set of finite perimeter. A straightforward computation gives
	\begin{equation}
		P_\varphi(E):= \int_{\pa^* E} \varphi(\nu_E) \, d \mathcal{H}^1.
	\end{equation}
	Now we recall the well-known first variation formula for the anisotropic perimeter. Let $E \subset \R^2$ of class $C^2$. For any vector field $X \in C^1_c(\R^2,\R^2)$, let $(\Phi (t,\cdot))_{t \in (-\varepsilon,\varepsilon)}$ be the unique solution of the Cauchy problem
	\begin{equation}
		\begin{cases}
			\frac{\pa }{\pa t} \Phi (t,x)= X\circ \Phi(t,x) \quad \forall x \in \R^2\\
			\Phi(0,x)=x \quad \forall x \in \R^2. 
		\end{cases}
	\end{equation}
	Then we have 
	\begin{equation}\label{FeulPANI}
		\frac{d}{d t} \bigg|_{t=0} \int_{\pa \Phi(t,E)} \varphi (\nu_{\Phi(t,E)})\, d \mathcal{H}^1 = \int_{\pa E} \kappa_E^{\varphi} X \cdot \nu_E \, d \mathcal{H}^1
	\end{equation}
	where the anisotropic curvature $\kappa^{\varphi}_E$ of $ \pa E$ is given by $$\kappa_E^{\varphi}:= \div_{\pa E} (\nabla \varphi (\nu_E))$$ and can also be written as
	\begin{equation}\label{ancurvE}
		\begin{split}
			\kappa_E^{\varphi}&=\div_{\pa E} (\nabla \varphi (\nu_E))= \nabla (\nabla \varphi (\nu_E)) \tau_E \cdot \tau_E - \nabla (\nabla \varphi (\nu_E)) \nu_E \cdot \nu_E \\
			&=( \nabla^2 \varphi (\nu_E) \tau_E \cdot \tau_E )\kappa_E:= g(\nu_E)\kappa_E
		\end{split}
	\end{equation}
	where \begin{equation}\label{lafunzghaprop}
		g \in C^{\infty}(\R^2 \setminus \{0\}), \,  C_g=\min_{\vert \nu \vert =1} g(\nu)>0.
	\end{equation} 
	We recall an anisotropic version of the Gauss–Bonnet theorem for curves (see \cite{KimKwon2025} for a proof).
	\begin{lemma}\label{lGaussB}
		Let $\varphi \in C^\infty(\R^n \setminus \{0\})$ be a regular, strictly convex norm. There exists a constant $C_\varphi>0$, depending only on $\varphi$,  such that for all open, bounded sets $E \subset \R^2 $ of class $C^2$, the following holds:
		\begin{equation}
			\int_{\pa E} \kappa_E^\varphi(x) \varphi (\nu_E(x)) \, d \mathcal{H}^1_x= C_\varphi.
		\end{equation}
	\end{lemma}
	
	\subsection{Interpolation inequality} We recall the interpolation inequalities involving Sobolev norms on embedded surfaces. We use the result from \cite[Proposition 6.5]{Mantegazza2002} (see also \cite[Proposition 4.3]{DianaFuscoMantegazza}).
	\begin{proposition}\label{PROPINTER}
		Let $ E \in \mathfrak{C}^m_{K,\sigma_0}(E_0)$ for some $m \geq 2$. Then for integers $0\leq k < l \leq m$ and numbers $p \in [1,\infty)$,  $q,r \in [1,\infty]$   there is $\theta \in [k/l,1]$ such that for every function $f$ of class $C^l$ on $\pa E$ it holds
		\begin{equation}
			\|  \pa_{\tau}^k f \|_{L^p(\pa E)} \leq C \| f\|_{W^{l,q}(\pa E)}^\theta \| f\|_{L^{r}(\pa E )}^{1-\theta}
		\end{equation}
		for a constant $C=C(k,l,p,q,r,\theta,C_0)$,  provided that the following  condition is satisfied
		\begin{equation}
			\frac{1}{p} = k + \theta \left( \frac{1}{q} - l \right) + \frac{1}{r}(1 - \theta).
		\end{equation}
		Moreover, if $f: \pa E \rightarrow \R$ is a smooth function with $\int_{\pa E} f\,d\mathcal{H}^{1}=0$ the above inequality can be written as 
		\begin{equation}
			\|\pa_{\tau}^k f\|_{L^p(\pa E)} \leq C \|  \pa_{\tau}^l f\|_{L^{q}(\pa E)}^\theta \| f\|_{L^{r}(\pa E)}^{1-\theta}.
		\end{equation}
	\end{proposition}
	If $E \in \mathfrak{C}^{k,\alpha}_{K,\sigma_0}(E_0)$ for some $1 \leq k \leq 5$ and $ \alpha \in [0,1]$, then the classical interpolation inequality in H\"older norms holds, i.e.,
	for $ 0 < \beta < \alpha \leq 1 $ and $0 \leq l \leq m \leq k$ it holds
	\begin{equation}\label{interHOLDER}
		\| f \|_{C^{l,\beta}(\pa E)} \leq C \| f \|_{C^{m,\alpha}(\pa E)}^{\theta} \|   f \|_{C^0(\pa E)}^{1-\theta} \quad \theta= \frac{l+ \beta}{m+ \alpha},
	\end{equation}
	where $C$ depend on $K,l,m,\alpha,\beta$. This result follows from the Euclidean case; see, for example \cite[Example 1.9]{Lunardi2018}. The interpolation inequality in Proposition \ref{PROPINTER} implies the following useful estimate. The
	proof is standard, and we refer to \cite[Proposition 2.3]{JN}. Note that the argument is similar to that used in the Euclidean case; see \cite[Proposition 3.7]{Taylor}. We denote the sum of the components of an index vector $\alpha \in \N^{l}$ by
	$$  \vert \alpha \vert= \alpha_1+ \cdots + \alpha_l.$$
	\begin{lemma}\label{LemmLeibinz}
		Let $ E \in \mathfrak{C}^m_{K,\sigma_0}(E_0)$ for some $m \geq 2$ and let $f_1, \cdots, f_l $ be function of class $C^m$. Then for an index vector $ \alpha \in \N^{l}$ with norm $ \vert \alpha \vert \leq k \leq m$ it holds
		\begin{equation}
			\| \vert \pa_{\tau}^{\alpha_1} f_1 \vert \cdots \vert \pa_{\tau}^{\alpha_l} f_l \vert \|_{L^2(\pa E)} 
			\leq  C(K) \sum_{ \sigma \in S_l} \| f_{\sigma(1)}   \|_{L^\infty(\pa E)} \cdots \| f_{\sigma(l-1)}   \|_{L^\infty(\pa E)} \| f_{\sigma(l)}  \|_{H^k(\pa E)}
		\end{equation}
		where $S_l$ is the group of  permutation of $l$ object. In particular, 
		\begin{equation}
			\| \pa_{\tau}^k (f_1 f_2) \|_{L^2(\pa E)} \leq C(K) \big[ 
			\|f_1 \|_{L^\infty(\pa E)} \| f_2 \|_{H^k(\pa E)}+ \| f_2 \|_{L^\infty(\pa E)} \| f_1 \|_{H^k(\pa E)}\big].
		\end{equation}
	\end{lemma}
	
	\section{Setting of the problem}\label{sezionesettingpro}
	\subsection{Pseudo-pseudo-$H^{-1}$ metric} In this subsection, we recall the definition and some basic properties of the pseudo-pseudo-$H^{-1}$ distance introduced in \cite{CaTa} to model surface diffusion.
	\begin{definition}[Pseudo-pseudo-$H^{-1}$ metric]
		Let $ E \subset \R^2$ be a set of finite perimeter, and let $ F \subset \R^2$ be a measurable set. We define the function $d_{H^{-1}}(F,E)$ as
		\begin{equation}\label{d_{H^{-1}}}
			d_{H^{-1}}(F,E):= \sup_{ \| \nabla f \|_{L^2(\pa^* E)} \leq 1} \int_{\R^2} f \circ \pi_{\pa^* E}(x) (\chi_F(x)-\chi_E(x))\, dx.
		\end{equation}
	\end{definition}
	\begin{remark}
		Let $ E \subset \R^2$ be a set of finite perimeter $ \vert E \vert < + \infty$, and let $ F \subset \R^2$ be a measurable set.
		We observe that if $\vert E \vert \neq \vert F \vert$ then $d_{H^{-1}}(F,E)= + \infty$. Indeed, for every $a \in \R$ we define $f: \pa E \rightarrow \R$ by $f(x)=a$. Then, by \eqref{d_{H^{-1}}}, we have
		\begin{equation}
			d_{H^{-1}}(F,E) \geq \sup_{a \in \R} a(\vert F \vert - \vert E \vert )= + \infty.
		\end{equation}
	\end{remark}
	\begin{lemma}\label{remgEF}
		Let $E \subset \R^2$ be a open bounded set of class $C^2$. Fix $\sigma>0$ be such that $\sigma< \sigma_E$. Let $F\subset \R^2 $ such that $\vert F \vert=\vert E \vert$ and $ F \Delta E \subset {\rm cl}(\mathcal{I}_{\sigma}(\pa E))$. We define
		\begin{equation}\label{lafunzgEF}
			\xi_{F,E} : \pa E \rightarrow \R \quad \xi_{F,E}(x):= \int_{-\sigma}^{\sigma} (\chi_{F}(x+t \nu_E(x))- \chi_{E}(x+t\nu_E(x)))(1+t\kappa_E(x))\, dt.
		\end{equation}
		Then, 
		\begin{equation}
			\int_{\pa E} \xi_{F,E}\, d\mathcal{H}^1=0, \quad d_{H^{-1}}(F,E)= \| \xi_{F,E} \|_{H^{-1}(\pa E)}.
		\end{equation}
		Moreover, if $\pa F$ is a normal graph respect $ \pa E $, i.e., $ \pa F = \{ x+ \psi(x)\nu_E(x) \colon x \in \pa E\} $,  then we have \begin{equation}\label{lafunzgEFgraph}
			\xi_{F,E}= \psi + \kappa_E \frac{\psi^2}{2},\quad d_{H^{-1}}(F,E)= \|\psi + \kappa_E \frac{\psi^2}{2} \|_{H^{-1}(\pa E)}.
		\end{equation}
	\end{lemma}
	\begin{proof}
		Let $t \in [-\sigma,\sigma]$, we define
		$$    \Psi_t: \pa E \rightarrow \{x \colon d_E(x)=t \},\,\, \Psi_t(x):=x+t \nu_E(x).  $$
		We have that $J_\tau \Psi_t(x) = 1+ t \kappa_E(x)$. Let $f \in H^1(\pa E)$. Using the coarea formula and a change of variables, we get
		\begin{equation}\label{asdasd1}
			\begin{split}
				&\int_{\R^2} f \circ \pi_{\pa E}(\chi_F-\chi_E)\, dx= \int_{-\sigma}^\sigma \int_{\{x \colon d_E(x)=t \} } f \circ \pi_{\pa E}(x)(\chi_F(x)- \chi_E(x)) \,d\mathcal{H}^1_x \, dt\\
				&= \int_{-\sigma}^\sigma \int_{ \Psi_t(\pa E) } f \circ \pi_{\pa E}(x)(\chi_F(x)- \chi_E(x)) \,d\mathcal{H}^1_x \, dt\\
				&= \int_{-\sigma}^\sigma \int_{\pa E } f (y)(\chi_F(y+ t \nu_E(y))- \chi_E(y+ t \nu_E(y)) J_\tau \Psi_t(y) \,d\mathcal{H}^1_y \, dt .
			\end{split}
		\end{equation}
		By \eqref{lafunzgEF} and \eqref{asdasd1} we obtain 
		\begin{equation}\label{slablig}
			\int_{\R^2}f \circ \pi_{\pa E}(x) (\chi_F(x)-\chi_E(x))\, dx= \int_{\pa E} f(y) \xi_{F,E}(y) \, d \mathcal{H}^1_y.
		\end{equation}
		In particular, for $f=1$, we find
		\begin{equation}\label{mediazero0}
			0 = \vert F \vert- \vert E \vert= \int_{\pa E} \xi_{F,E} \, d \mathcal{H}^1.
		\end{equation}
		Hence, from \eqref{slablig} and the definition of $d_{H^{-1}}(E,F)$, we obtain
		\begin{equation}
			d_{H^{-1}}(E,F)= \| \xi_{F,E}\|_{H^{-1}(\pa E)}.
		\end{equation}
		In the case where $\pa F$ is a normal graph over $\pa E$, we compute
		$$ \xi_{F,E}= \int_{0}^{\psi} 1+t\kappa_E\, dt=\psi+ \kappa_E \frac{\psi^2}{2}.$$
	\end{proof}

	Therefore, under the assumptions of the above lemma, we have
	\begin{equation}\label{normH^-1funzv}
		d_{H^{-1}}^2(F,E)= \int_{\pa E} \vert \nabla_{\tau} v_{F,E} \vert^2 \, d \mathcal{H}^{1}
	\end{equation}
	where $v_{F,E} $ is the unique solution to the equation
	\begin{equation}\label{eqv}
		\left\{
		\begin{aligned}
			& \Delta_{\pa E} v_{F,E}= \xi_{F,E} & \text{ on } \pa E , \\
			& \int_{\pa E} v_{F,E} \, d \mathcal{H}^1=0.
		\end{aligned}
		\right.
	\end{equation} 
	The function $f$ that realize the supremum in formula \eqref{d_{H^{-1}}} is given by $$f=\frac{v_{F,E}}{d_{H^{-1}}(F,E)}.$$
	
	In the next proposition, we compute the first variation of the function $F \rightarrow d_{H^{-1}}(F,E)$.
	\begin{proposition}\label{propELDH-1}
		Let $ E \subset \R^2$ be a bounded open set of class $C^2$, and let $ \sigma< \sigma_E$. Let $F \Subset \R^2$ be a set of class $C^1$ such that $ F \Delta E \subset \mathcal{I}_{\sigma}(\pa E)$. Let $ X \in C_c^1(\R^2,\R^2)$ be such that $ \div X=0$, and let $\Psi: (-\varepsilon,\varepsilon) \times \R^2 \rightarrow \R^2$ be the solution of the Cauchy problem
		\begin{equation}
			\begin{cases}
				\frac{\pa }{\pa t} \Psi (t,x)= X\circ \Psi (t,x) \quad \forall x \in \R^2\\
				\Psi(0,x)=x \quad \forall x \in \R^2. 
			\end{cases}
		\end{equation}
		Finally, let $ f_0 \in H^1(\partial E)$ with $ \int_{\pa F} f_0 \,d \mathcal{H}^{1}=0$ be the function that realizes the supremum in the definition of $d_{H^{-1}}(F,E)$.
		Then 
		\begin{equation}
			\frac{d }{d t} d_{H^{-1}}(\Psi(t,F),E) |_{t=0}= \int_{\partial F} f_0(\pi_{\pa E}(x)) X(x) \cdot \nu_{F}(x) d \mathcal{H}^{1}_x.
		\end{equation}
	\end{proposition}
	\begin{proof} We fix $\varepsilon>0$ such that $ \Psi(t,F) \triangle F  \Subset  \mathcal{I}_\sigma(\partial E)$ for all $ t \in (-\varepsilon,\varepsilon)$. 
		Set $F_t := \Psi(t,F)$ for $t \in (-\varepsilon,\varepsilon)$. 
		Define $\xi_t : \partial E \rightarrow \R$ for $t \in (-\varepsilon,\varepsilon)$, by
		\begin{equation}
			\xi_t(y):= \int_{-\sigma}^\sigma(\chi_{F_t}(y+s\nu_E(y))-\chi_E(y+s\nu_E(y))J_\tau\Phi_s(y)\,ds,\qquad y\in\pa E,
		\end{equation}
		where $\Phi_s(x)= x+ s \nu_E(x) $. For every $ t \in (-\varepsilon,\varepsilon)$ we have
		\begin{equation}\label{1407asdasd}
			d_{H^{-1}}^2(F_t,E)= \int_{\pa E} \vert \nabla_\tau v_t \vert^2 d \mathcal{H}^{1},
		\end{equation}
		where
		\begin{equation}\label{14072024matt1}
			\begin{cases}
				-\Delta_{\tau} v_t=\xi_t & \text{on $\pa E$}, \cr
				\displaystyle\int_{\pa E}v_t\,d\Ha^{1}=0.
			\end{cases}
		\end{equation}
		We note that $\xi_t \in L^{\infty}(\partial E)$ with $ \| \xi_t \|_{\infty} \leq C(\| \kappa_E \|_{\infty})$.\\
		\textit{Claim:} $\xi_t\to \xi_0$ in $L^p(\partial E)$ for all $p\geq1$. 
		
		Indeed, since $\pa F_t\subset \mathcal{I}_{Ct}(\pa F)$ for some $C>0$ depending only on $\| X\|_\infty$, we have 
		\begin{equation}
			\begin{split}
				\| \xi_t  -\xi_0 & \|_{L^p(\partial E)}^p = \int_{\partial E} \big|\int_{-\sigma}^{\sigma}  (\chi_{F_t}(y+s\nu_E(y))-\chi_E(y+s\nu_E(y))J\Phi_s(y)\,ds \big|^p d \mathcal{H}^1_x \\
				&= \int_{B(z,Ct) \cap \pa E} \big|\int_{-Ct}^{Ct}  (\chi_{F_t}(y+s\nu_E(y))-\chi_E(y+s\nu_E(y))J\Phi_s(y)\,ds \big|^p d \mathcal{H}^1_x \leq C t^{p}.
			\end{split}
		\end{equation}
		In particular this implies that $v_t\to v_0$ in $W^{2,p}(\pa E)$ for all $p\geq1$, hence also uniformly on $\pa E$.  Therefore we have 
		\begin{equation}\label{18072024pom1}
			v_t \circ \pi_{\pa E} \rightarrow v_0 \circ \pi_{\pa E} \quad \text{ as }t\rightarrow 0  \text{ in }C^0({\rm cl}\big(\mathcal{I}_\sigma(\pa E)\big)) .
		\end{equation}
		Recall now that, see \cite[Proposition 17.8]{Maggibook}, that for all $\varphi\in C_c(\R^2)$
		$$
		\lim_{t\to0}\frac{1}{t}\bigg(\int_{F_t}\varphi\,dx-\int_{F}\varphi\,dx\bigg)=\int_{\pa F}\varphi X\cdot\nu_F\,d\Ha^{1}, 
		$$	
		that is 
		\begin{equation}\label{18072024pom2}
			\frac{1}{t}(\chi_{F_t}-\chi_F)\mathcal L^2\weakstar X\cdot\nu_F\Ha^{1}\res\pa F \qquad \text{in the sense of measures.}	
		\end{equation}
		Now, using the divergence theorem, formula \eqref{18072024pom1} and coarea formula, we have that
		\begin{equation}\label{18072024pom3}
			\begin{split}
				&	\int_{\partial E} (\vert \nabla_\tau v_t \vert^2 - \vert \nabla_\tau v_0 \vert^2 )\, d \mathcal{H}^{1}= \int_{\partial E} (\nabla_\tau  v_t - \nabla_\tau  v_0 ) \cdot (\nabla_\tau  v_t+ \nabla_\tau  v_0) \,d \mathcal{H}^{1}\\
				&= \int_{\pa E} (-\Delta_\tau  v_t+\Delta_\tau  v_0)(v_t+v_0) \, d \mathcal{H}^{1}= \int_{\pa E} (\xi_t-\xi_0)(v_t+v_0) \,d \mathcal{H}^{1}\\
				& =\int_{\partial E}\int_{-\sigma}^{\sigma}(\chi_{F_t}(y+s\nu_E(y))-\chi_F(y+s\nu_E(y)) J \Phi_s(y)(v_t(y)+v_0(y)) \,ds \,d\mathcal{H}^{1}_y\\
				&= \int_{\R^2} (v_t (\pi_{\pa E}(x))+v_0(\pi_{\pa E}(x)))(\chi_{F_t}(x)-\chi_F(x))\,dx.
			\end{split}
		\end{equation}
		Therefore, by \eqref{1407asdasd}, \eqref{18072024pom1}, \eqref{18072024pom2}, \eqref{18072024pom3}, we obtain
		\begin{equation}
			\begin{split}
				&\lim_{t \rightarrow 0} \frac{d_{H^{-1}}(F_t,E)^2-d_{H^{-1}}(F,E)^2}{t}= \lim_{t \rightarrow 0} \frac{\int_{\pa E} \vert \nabla_\tau v_t \vert^2\, d \mathcal{H}^{1}-\int_{\pa E} \vert \nabla_\tau v_0 \vert^2 \,d \mathcal{H}^{1}}{t}\\
				& =\lim_{t \rightarrow 0} \frac{\int_{F_t} (v_t(\pi_{\pa E}(x))+v_0(\pi_{\pa E}(x)))\,dx-\int_{F} (v_t(\pi_{\pa E}(x))+v_0(\pi_{\pa E}(x)))\,dx}{t}\\
				&=\int_{\pa F}2 v_0(\pi_{\pa E}(y)) X(y)\cdot\nu_F(y)\,d\Ha^{1}_y.
			\end{split}
		\end{equation}
		Hence by above formula and recalling $ f_0 = v_0 / \| \nabla_\tau v_0 \|_{L^2(\pa E)}$, we obtain  the desired result.
	\end{proof}
	\subsection{Elastic energy}  Let $F \Subset \Omega$, and let $u \colon \Omega \setminus F \rightarrow \R^2$ be an  elastic displacement. We define $E(u)$, the symmetric part of $\nabla u$, as $$E(u):=\frac{\nabla u+ (\nabla u)^{T}}{2}.$$ 
	Throughout this work,
	$\mathbb{C}$ denotes a fourth-order elasticity tensor acting on symmetric $2 \times 2$ matrices $A$, satisfying the coercivity condition $$\mathbb{C}A \colon A >0 \text{ for all } A \neq 0.$$ We define the elastic energy density as $$Q(A):= \frac{1}{2} \mathbb{C}A \colon A.$$  
	\subsubsection{Constrained elastic energy}
	Let $K_{el}>0$ and $ h>0$ be fixed. Given a boundary displacement $w_0 \in C^{3,\frac{1}{4}}(\pa \Omega)$, we define the minimization problem
	\begin{equation}\label{minelvinc}
		u_F^{K_{el},h} \in \mathrm{argmin} \left\{ \int_{\Omega \setminus F} Q(E(u))\, dx \colon u \in \mathfrak{C}_{K_{el}}^{3,\frac{1}{4}}(\Omega ,\R^2) ,\, \| \nabla^4 u \|_{C^{0,\frac{1}{4}}(\Omega)} \leq \frac{K_{el}}{h^{\frac{1}{4}}},\, u |_{\pa \Omega}= \omega_0   \right\}.
	\end{equation}
	We then define the constrained elastic energy as
	\begin{equation}\label{enelvinc}
		\mathcal{E}(E(u_F^{K_{el},h})):= \int_{\Omega \setminus F} Q(E(u_F^{K_{el},h}))\, dx.
	\end{equation}
	\begin{remark}
		The existence of an minimizier for the problem \eqref{minelvinc} follows from the Arzela-Ascoli Theorem. Hence, the energy functional in  \eqref{enelvinc} is well-defined.
	\end{remark}
	In what follow, we omit the explicit dependence of $u_{F}^{K_{el},h}$ on $h$, and we write $u_{F}^{K_{el}}$ for brevity.
	\subsubsection{Elastic energy} Now, fix a boundary displacement $w_0 \in C^{3,\frac{1}{4}}(\pa \Omega)$ we define the (unconstrained) elastic problem as:
	\begin{equation}\label{minelast}
		u_F \in \arg \!\min \left\{ \int_{\Omega \setminus F} Q(E(u))\, dx \colon u \in H^1(\Omega \setminus F, \R^2)    \right\}
	\end{equation}
	and we define the corresponding energy as 
	\begin{equation}\label{enelast}
		\mathcal{E}(E(u_F)):= \int_{\Omega \setminus F} Q(E(u_F))\, dx.
	\end{equation}
	More precisely, $u_F$ is the unique solution in $H^1(\Omega \setminus F, \R^2)$ to the following elliptic system:
	\begin{equation}\label{eqelliel}
		\left\{
		\begin{aligned}
			& \div \mathbb{C}E(u_F)=0 & \text{ in } \Omega \setminus F  , \\
			& \mathbb{C}E(u_F)[\nu_F]=0 & \text{ on }  \pa F , \\
			& u_F= w_0 & \text{ on } \pa \Omega.
		\end{aligned}
		\right.
	\end{equation}  
	We recall that if $ w_0 \in C^{3,\frac{1}{4}}(\pa \Omega)$ and $F $ is of class $C^{3,\frac{1}{4}}$, then the solution $u_F \in C^{3,\frac{1}{4}}(\Omega \setminus F)$ by standard elliptic regularity theory (see \cite{AgmonDouglisNiremberg1964}, \cite[Proposition 8.9]{FM2012}). Moreover, the following estimate holds: $$ \| u_F \|_{C^{3,\frac{1}{4}}(\Omega \setminus F)} \leq C (\| w_0 \|_{C^{3,\frac{1}{4}}(\pa \Omega)}+ \|u_F \|_{C^{3,\frac{1}{4}}(\pa F)}) $$
	where $C$ is an universal constant. Thank to this observation, we have that for $K_{el}$ sufficiently large, the minimization problems \eqref{minelvinc} and \eqref{minelast} are equivalent, so that $u_F=u_F^{K_{el},0}$.
	
	In the next proposition, we compute the first variation of the function $$F \rightarrow \mathcal{E}(E(u_F^{K_{el}})).$$
	\begin{proposition}\label{PropelELASt}
		Let $F \Subset \Omega$ be a set of class $C^1$, let $X \in C^1_c(\Omega,\R^2)$, and let $( \Phi(t,\cdot))_{t \in (-\varepsilon,\varepsilon)}$ be the unique solution of the Cauchy problem:
		\begin{equation}
			\left\{
			\begin{aligned}
				& \frac{\pa }{\pa t} \Phi(t,x)= X \circ \Phi(t,x) & \forall x \in \R^2 ,\\
				& \Phi(0,x)=x & \forall x \in \R^2.
			\end{aligned}
			\right.
		\end{equation}
		We define $F_t= \Phi(t,F)$. Then the following identity holds:
		\begin{equation}\label{FVELAST}
			\frac{d}{d t} \bigg|_{t=0} \mathcal{E}(u_{F_t}^{K_{el}})= -\int_{\pa F} Q(E(u_F^{K_{el}})) X \cdot \nu_F \, d \mathcal{H}^1.
		\end{equation}
	\end{proposition}
	\begin{proof}
		Without loss of generality, we can assume that $ F_t \Subset \Omega$ for all $ t \in (-\varepsilon,\varepsilon)$. Note that the symmetric difference $ F_t \Delta F   $ is contained in a tubular neighborhood $\mathcal{I}_{t C} (\pa F)$, where $C=C(\| X \|_{\infty}) $. As a result, \begin{equation}\label{31032025form1}
			\vert F_t \Delta F \vert \rightarrow 0 \text{ as } t \rightarrow 0.
		\end{equation} 
		For every $t \in (-\varepsilon,\varepsilon)$, let $u_{F_t}^{K_{el}} \in \mathfrak{C}_{K_{el}}^{3,\frac{1}{4}}(\Omega,\R^2)$ be a minimizer of the problem \eqref{minelvinc} for $F=F_t$. Then using the Arezela-Ascoli Theorem, up to a subsequece, we have that \begin{equation}\label{31032025form2}
			u_{F_t}^{K_{el}} \rightarrow u \text{ in } \mathfrak{C}_{K_{el}}^{3,\frac{1}{4}}(\Omega,\R^2).
		\end{equation} 
		\textit{Claim} $u= u_F^{K_{el}}$.\\
		Combining \eqref{31032025form1} and \eqref{31032025form2}, we deduce:
		\begin{equation}
			\int_{\Omega \setminus F} Q(E(u))  \, d x = \lim_{t \rightarrow 0} \int_{\Omega \setminus F_t} Q(E(u_{F_t}^{K_{el}})) \, d x\leq \lim_{t \rightarrow 0} \int_{\Omega \setminus F_t} Q(E(v)) \, d x = \int_{\Omega \setminus F} Q(E(v)) \, dx
		\end{equation}
		for every admissible test function $ v \in \mathfrak{C}^{3,\frac{1}{4}}_{K_{el}}(\Omega,\R^2)$. Therefore, 
		$u$ must be a minimizer $u_{F}^{K_{el}}$, and hence: \begin{equation}\label{31032025form3}
			u_{F_t}^{K_{el}} \rightarrow u_{F}^{K_{el}} \text{ in } \mathfrak{C}_{K_{el}}^{3,\frac{1}{4}}(\Omega,\R^2).
		\end{equation}
		Finally, recall \cite[Proposition 17.8]{Maggibook} that:
		\begin{equation}
			\frac{1}{t}(\chi_{F_t}-\chi_F)\mathcal L^2\weakstar X\cdot\nu_F\Ha^{1}\res\pa F,	
		\end{equation}
		in the sense of measures. Combining this with formula \eqref{31032025form3}, we get the desired derivative formula, i.e., \eqref{FVELAST}.
	\end{proof}
	\subsection{Minimizing movement scheme and flat solution}
	Fix $h>0$ be a fixed time step discretization. Let $K_{el}>0$ be fixed. Let $E \Subset \Omega  $ be a bounded open set of class $C^2$. For every set $F \subset \R^2$ sufficiently close to $E$, we define the functional
	\begin{equation}\label{energia}
		\mathcal{F}_h(F,E):= \mathcal{G}(F)+ \frac{1}{2h}d_{H^{-1}}^2(F,E)
	\end{equation}
	where 
	\begin{equation}\label{defG}
		\mathcal{G}(F)=P_{\varphi}(F)+\mathcal{E}(E(u_F^{K_{el}})).
	\end{equation} 
	\begin{definition}[Constrained discrete flat flow] \label{12092023def1}
		Let 
		\begin{equation}
			\text{ $\beta< \min \{ \sigma_{E_0}, \mathrm{dist}(\pa \Omega, \pa E_0) \}$ and $K_{el}$ be fixed.}   
		\end{equation}
		Let $h >0$ be the time step discretization. Define the family of sets $\{ E_{hk}^{h,\beta}\}_{k \in \mathbb{N}}$ iteratively by setting $E_0^{h,\beta}:=E_0$ and, 
		$$   E_{ hk}^{h,\beta} \in \arg \min \left\{\mathcal{F}_{h}(F,\, E_{h(k-1)}^{h,\beta}) \colon \, F \Delta E_{h(k-1)}^h \subset {\rm cl}( I_{\beta}(\pa E_{h(k-1)}^{h,\beta}))\right\} \quad k\geq 1,$$
		where the functional $\mathcal{F}_h$ is defined in \eqref{energia}.
		We define 
		\begin{equation}\label{agligderigligb}
			E_{t}^{h,\beta}:= E_{hk }^{h,\beta} \quad \text{for any } t \in [k h,\, (k+1)h).
		\end{equation}
		The family $\{E_t^{h,\beta}\}_{t \geq 0}$ is called a constrained discrete flat flow with initial datum $E_0$ and time step $h$. 
	\end{definition}
	We define a flat flow solution $\{E^\beta_t\}_{t \geq 0}$ of the anisotropic surface diffusion with elasticity as any cluster point when we let  $h \to 0^+$ of $ \{ E^{h,\beta}_t \}_{t \geq 0}$. 
	\section{Preliminary estimates}\label{stimaprelininari}
	The aim of this section is to establish a regularity estimate for the set $F$ that minimizes the incremental problem
	\begin{equation}\label{probdimin1}
		\min \left\{ \mathcal{F}_h(A,E) \colon \, A \Delta E \subset {\rm cl}(\mathcal{I}_{ \eta} (\pa E)) \right\}
	\end{equation}
	where $E \in \mathfrak{H}^4_{K,\sigma_0}(E_0)$ and $\eta(K,K_{el})>0$. We recall that $E_0 \Subset \Omega$ be open and connected set of class $C^5$ such that $\vert E_0 \vert=1$. The main result of this section is the following:
	\begin{theorem}\label{MainThm375}
		Let $E$ be a set of class $C^5$ such that $E \in \mathfrak{H}^4_{K,\sigma_0}(E_0)$ and $\| \pa_{\tau}^3 \kappa_E^{\varphi} \|_{L^2(\pa E)} \leq \frac{K}{h^{\frac{1}{4}}}$. 
		Then there exist constants $\eta_0 = \eta_0(K,K_{el})$, $C_1= C_1(K,K_{el})$, and $ C_2= C_2(K,K_{el})$, such that, for every $\eta < \eta_0$, there exits $h_0$ with the following property: if $0 < h  \leq h_0$ and $F$ is a minimizer of \eqref{probdimin1},  then $\pa F \Subset \mathcal{I}_{\eta}(\pa E)$ and coincides with a graph of a smooth function $\psi \colon \pa E \rightarrow \R$ satisfying
		\begin{equation}\label{formMain1375}
			\| \psi \|_{L^2(\pa E)} \leq C_1 h, \quad \| \psi \|_{H^4(\pa E)} \leq C_1
		\end{equation}
		and
		\begin{equation}\label{FormMain2375}
			\|  \kappa_F^{\varphi} \|_{H^2(\pa F)} \leq C_2, \qquad  \| \pa_{\pa F}^3 \kappa_F^\varphi \|_{L^2(\pa F)} \leq \frac{C_2}{h^{\frac{1}{4}}}.
		\end{equation}
		Moreover, there exist constants $\hat \sigma=\hat{\sigma}(K,K_{el})$, and $ K_1=K_1(K,K_{el})$ such that $ F \in \mathfrak{H}^4_{K_1, \hat \sigma}(E_0)$.
	\end{theorem}
	\subsection{$\Lambda$-minimality estimate}
	In this subsection, we prove that any minimizer $F$ of \eqref{probdimin1} is a $\Lambda$-minimizer of the $\varphi$-perimeter, with $\Lambda$ independent from $h$.
	
	We begin by recalling the definition of a $\Lambda$-minimizer of the $\varphi$-perimeter.
	\begin{definition}
		Let $E \subset \R^2$ be a set of finite perimeter. We say that $E$ is a $\Lambda$-minimizer of the $\varphi$-perimeter if there exists $\Lambda>0$ such that 
		\begin{equation}
			P_{\varphi}(E) \leq P_{\varphi}(G)+ \Lambda \vert G \Delta E \vert
		\end{equation}
		for every $G \subset \R^2$.
	\end{definition}
	It is known that if $ E\subset \R^2$ is a $\Lambda$-minimizer of the $\varphi$-perimeter, then $ \pa E$ is of class $C^{1,\eta}$ for all $ \eta \in [0 , \frac{1}{2})$ see  \cite{AlmgrenSimonSchoen1977}, \cite{Bombieri1982} and \cite{DePMa2015}. In the case where $\varphi$ is the Euclidean norm, see also \cite[Theorem~1.9]{Tamanimi1984}. 
	
	We will use the following lemma.  The proof is similar to those in \cite[Lemma~2.8]{CFJKsd} and \cite{KLP2025}, but we include it here for the reader’s convenience. 
	\begin{lemma}\label{propdadim}
		Assume that $E \in \mathfrak{H}^{3}_{K,\sigma_0}(E_0)$ and let $F$ be an $\Lambda$-minimizer of the $\varphi$-perimeter. Then for every $\gamma \leq \frac{1}{4} $, there exists  $\delta_0 =\delta_0(K,\Lambda,\gamma)$ such that 
		if 
		$$
		F \Delta E \subset \mathrm{cl}(\mathcal{I}_{\delta_0}(\pa E)),
		$$
		then there exists a function $\psi \in C^{1,\gamma}(\pa E)$ such that 
		\begin{equation}
			\pa F = \{ x + \psi(x)\nu_E(x) : x \in \pa E\}. 
		\end{equation}
		Moreover, for every $\eps>0$ there exists $\delta_0 = \delta_0(\eps)$ such that  $\|\psi\|_{C^{1,\gamma'}(\pa E)} \leq \eps$ for $\gamma' < \gamma$.
	\end{lemma} 
	\begin{proof}
		By assumption, $E \Delta F \subset \mathrm{cl} (\mathcal{I}_{\delta_0}(\pa E))$ we have that for every $x \in \pa E$,
		$$ {\rm cl}\big(B_{\delta_0}(x)\big) \cap \pa F \neq \emptyset.$$  
		Let $\varepsilon>0$ 
		be fixed, and let $C(K)\geq K>0$ be a constant that we will choose  later. \\
		\textit{Claim:} For all $\delta_0 \in (0,\frac{\varepsilon}{100 C(K)})$  if $E$ and $F $ satisfy the assumption, then
		\begin{equation}\label{burro}
			\vert \nu_E(x)-\nu_F(y) \vert \leq \varepsilon \quad \text{for all } y \in \pa F \cap \mathrm{cl}(B_{\delta_0}(x)). 
		\end{equation}
		
		We argue by contradiction. Suppose the claim fails. Then there exist $\varepsilon>0$, sequences $\{E_n\}_{n \in \N},\, \{F_n\}_{n \in \N}  $ such that 
		\begin{enumerate}
			\item $ E_n \in \mathfrak{H}_{K,\sigma_0}^{3}(E_0)$ for all $n \in \N$,
			\item $F_n$ is a $\Lambda$-minimizer of the $ \varphi$-perimeter for all $n \in \N$,
			\item $ E_n \Delta F_n \subset \mathrm{cl}(\mathcal{I}_{\delta_0}(\pa E))  $ for all $n \in \N$,
			\item exist $x_n \in \pa E_n,\, y_n \in B_{\frac{1}{n}}(x_n) \cap \pa F_n$ such that
			\begin{equation}\label{contradict}
				\vert \nu_{E_n}(x_n)- \nu_{F_n}(y_n) \vert \geq \varepsilon \text{ for all } n \in \N.
			\end{equation}
		\end{enumerate}
		Without loss of generality and up to extracting a subsequence, we have $x_n,\,y_n \rightarrow x$ as $ n \rightarrow + \infty$,
		$$ E_n \rightarrow  E    \text{ in } \mathfrak{H}^{3}_{K,\sigma_0}(E_0), \,\, F_n \rightarrow F \text{ in Hausdorff distance,}  $$
		where $F$ is a $\Lambda$-minimizer of the $ \varphi$-perimeter. Therefore, we have $$ \nu_{E_n}(x_n) \rightarrow \nu_E(x)\text{ as }n \rightarrow + \infty .$$ 
		Now using the $ \Lambda$-minimality of $F_n$, we obtain $\nu_{F_n}(y_n) \rightarrow \nu_E(x)$; see \cite{Bombieri1982}. This contradicts \eqref{contradict}. 
		
		The conclusion of the lemma follows from \eqref{burro} and using a standard regularity argument. Indeed, let  $x_0 \in \pa E$. We may assume, without loss of generality, that $x_0 = 0$ and $\nu_E(0) = e_2$. Since $E \in \mathfrak{H}^{3}_{K,\sigma_0}(E_0)$, there exists $r_0= r_0(K) \leq  \frac{1}{C(K)}$ such that $ E \cap B_{r_0/2}$ coincides with the subgraph of a function $f: (-\frac{r_0}{2}, \frac{r_0}{2}) \to \R$, with $$\|f\|_{C^{1,\frac{1}{4}}((-\frac{r_0}{2}, \frac{r_0}{2}))} \leq \frac{10}{r_0},$$ provided $r_0 \leq 1$.  It then follows that  $|\nu_E(x) - e_2 | < \eps$   for all $x \in \pa E \cap B_{r_\eps}$, where $r_\eps = \frac{r_0}{20}\eps$. Observe that $\delta_0 < \frac{r_0}{80}\eps$ implies  $\delta_0 < \frac{r_\eps}{4}$. Then, by  \eqref{burro}, we obtain $|\nu_F(y) - e_2 | < 2\eps$ for all $y \in \pa F \cap B_{\frac{3r_\eps}{4}}$. Choose any point  $y_0 \in \pa F \cap B_{\delta_0}$ and using the previous inequality and the perimeter density estimates for $\Lambda$-minimizers of $\varphi$-perimeter, we conclude that the excess satisfies 
		\begin{equation}
			{\bf{e}}\big(F,y_0,\frac{r_\eps}{2}\big) =\min_{\omega \in \mathcal{S}^1}  \frac{1}{r_\eps}  \int_{\pa F \cap B_{\frac{r_\eps}{2}}(y_0)} |\nu_F(y) - \omega |^2 \, d \mathcal{H}^{1}_y\leq C\eps^2,
		\end{equation}
		provided $r_\eps<r_1=r_1(\Lambda,K)$, for some constant $C=C(\Lambda,K)$. Then, by the $\eps$-regularity theorem (see \cite{Bombieri1982}), and since $B_{r_\eps/4} \subset B_{r_\eps/2}(y_0) $  there exists a function $\varphi :(-r_\eps/4,r_\eps/4) \to \R$ such that 
		\[
		F \cap B_{r_\eps/4} = \{ (y_1,y_2) \in \R^2 : y_2 < \varphi(y_1) \} \cap  B_{r_\eps/4} 
		\]
		with   $\|\varphi\|_{C^{1,\gamma}((-r_\eps/4,r_\eps/4) )}\leq C$ and $\gamma \leq   \frac14 $. The existence of the heightfunction $\psi \in C^{1,\gamma}(\pa E)$ follows from the assumption that $ E \in \mathfrak{H}_{K,\sigma_0}^3(E_0)$, see \cite[Section 1.2]{Antonia}. 
		
		Finally, the smallness of the norm $\|\psi\|_{C^{1,\gamma'}(\pa E)}$ for $\gamma' <  \gamma$ ,when $\delta_0$ is small, follows from interpolation inequality \eqref{interHOLDER}, using that $\|\psi\|_{L^\infty(\pa E)} \leq \delta_0$. 
	\end{proof}
	
	We proceed to prove a technical lemma that will be instrumental at various stages of the article.
	
	\begin{lemma}\label{lemmazzo1}
		Let $E \in \mathfrak{H}^3_{K,\sigma_0}(E_0)$ be such that $\vert E \vert=1$. Then there exist constants $\sigma,
		C$ depending only on $K$ such that the following holds: if $F \subset \R^2$ with $ \pa F = \{ x+ \psi(x)\nu_E(x)\colon x \in \pa E\}$ for some function $ \| \psi \|_{C^1(\pa E)} \leq \sigma$ with $ \vert F \vert=1$, then \begin{equation}\label{tesilemmazzzo}
			\frac{1}{C} \| \nabla_{\pa E} \psi \|_{L^2(\pa E)} \leq  \| \nabla_{\pa E} \xi_{F,E} \|_{L^2(\pa E)} \leq C \| \nabla_{\pa E} \psi \|_{L^2(\pa E)}
		\end{equation}
		where $\xi_{F,E}$ is defined in Lemma \ref{remgEF}.
	\end{lemma}
	\begin{proof}
		By Lemma \ref{remgEF}, we have  
		$ \xi_{F,E}= \psi + \frac{\psi^2}{2}\kappa_E$. 
		For $\sigma$ sufficiently small, we obtain  
		\begin{equation}\label{perchèiosiloso}
			\frac{\psi^2}{2} \leq \big(\psi+ \frac{\psi^2}{2}\kappa_E\big)^2= \xi_{E,F}^2 .
		\end{equation}   
		Computing the tangential gradient of $\xi_{F,E}$, we find
		\begin{equation}\label{050502025zoo}
			\nabla_{\pa E} \xi_{F,E}= \nabla_{\pa E} \psi+ \psi \kappa_E \nabla_{\pa E} \psi+ \frac{\psi^2}{2} \nabla_{\pa E}\kappa_E.
		\end{equation}
		Since $\vert E \vert = \vert F \vert$, we have $ \int_{\pa E} \xi_{F,E} \, d \mathcal{H}^1=0$ (see formula \eqref{mediazero0}). 
		Therefore, using \eqref{perchèiosiloso}, \eqref{050502025zoo}, and the H\"{o}lder inequality, we obtain \begin{equation}\label{asdhhhhh}
			\begin{split}
				\| \nabla_{\pa E} \xi_{F,E} \|_{L^2(\pa E)} &\leq (1+ \sigma K) \| \nabla_{\pa E} \psi \|_{L^2(\pa E)}+ \frac{\sigma}{\sqrt{2}} \| \xi_{F,E} \|_{L^2(\pa E)}\\
				& \leq  (1+ \sigma K) \| \nabla_{\pa E} \psi \|_{L^2(\pa E)}+ \frac{\sigma C_1}{\sqrt{2}} \| \nabla_{\pa E} \xi_{F,E} \|_{L^2(\pa E)}
			\end{split}
		\end{equation}
		where in the last inequality we used the Poincarè inequality: $$\| \xi_{F,E} \|_{L^2(\pa E)} \leq C_1 \| \nabla_{\pa E} \xi_{F,E} \|_{L^2(\pa E)}.$$ 
		Taking $\sigma$ small enough in \eqref{asdhhhhh}, we obtain  \begin{equation}\label{05052025primacaffè1}
			\| \nabla_{\pa E} \xi_{F,E} \|_{L^2(\pa E)} \leq C \| \nabla_{\pa E} \psi \|_{L^2(\pa E)} . 
		\end{equation}
		From \eqref{050502025zoo} and using the Sobolev embedding together with the H\"{o}lder's inequality, we get
		\begin{equation}\label{050502025dcaffè}
			\begin{split}
				\| \nabla_{\pa E} \psi \|_{L^2(\pa E)} &\leq \| \nabla_{\pa E} \xi_{F,E} \|_{L^2(\pa E)} + \| \nabla_{\pa E} \psi \|_{L^2(\pa E)} K \| \psi \|_{\infty}+ \| \nabla_{\pa E} \kappa_E \|_{L^2(\pa E)}  \| \psi^2 \|_{L^2(\pa E)}\\
				& \leq \| \nabla_{\pa E} \xi_{F,E} \|_{L^2(\pa E)} + \sigma K\| \nabla_{\pa E} \psi \|_{L^2(\pa E)}+ 2 \sigma K\| \xi_{F,E}\|_{L^2(\pa E)}\\
				& \leq \| \nabla_{\pa E} \xi_{F,E} \|_{L^2(\pa E)} + \sigma K\| \nabla_{\pa E} \psi \|_{L^2(\pa E)}+ 2 \sigma KC_1\| \nabla_{\pa E} \xi_{F,E}\|_{L^2(\pa E)}
			\end{split}
		\end{equation}
		where the second inequality uses \eqref{perchèiosiloso}, and the third uses the Poincaré inequality. Taking $\sigma$ small enough in \eqref{050502025dcaffè}, we get
		\begin{equation}\label{050525dcaffètesi}
			\frac{1}{C} \| \nabla_{\pa E} \psi \|_{L^2(\pa E)} \leq  \| \nabla_{\pa E} \xi_{F,E} \|_{L^2(\pa E)}.   
		\end{equation}
		Combining  \eqref{05052025primacaffè1} and \eqref{050525dcaffètesi} yields the desired formula, i.e., \eqref{tesilemmazzzo}.
	\end{proof}

	\begin{remark}
		Recalling that $ \varphi$ is a regular strictly convex norm, the following inequality holds:
		\begin{equation}\label{remform}
			\exists J_{\varphi}>0 \, \colon D^2 \varphi (\nu) \xi \cdot \xi \geq J_{\varphi} \vert \xi \vert^2 \quad \forall \nu \in \mathrm{cl}(\mathcal{I}_\frac{1}{4}(\mathcal{S}^1)),\,\xi \in \R^2 \text{ such that } \nu \bot \xi.
		\end{equation}  
	\end{remark}
	
	\begin{lemma}\label{lego123og}
		Let $E\in\mathfrak{H}^{3}_{K,\sigma_0}(E_0)$ with $ \vert E \vert=1$. Then there exist constants $\Lambda,\,\Lambda',\, \sigma$, depending on $K,K_{el},\Omega$, such that the following holds: if $F \subset \R^2$ is such that $ \pa F= \{ x +\psi(x) \nu_{E}(x) : x \in \pa E\},$ with $ \psi \in C^{1}(\pa F)$ and $ \| \psi \|_{C^{1}(\pa F)} \leq  \sigma,$ then
		\begin{equation}\label{lemmatesi1}
			\frac{J_\varphi}{4} \|\nabla_{\tau} \psi \|_{L^2(\pa F)}^2+\mathcal{G}(E) \leq \mathcal{G}(F)+ \Lambda d_{H^{-1}}(F,E)
		\end{equation}
		where $J_{\varphi}$ is defined in \eqref{remform}. 
		Furthermore, if $ F \subset \R^2 $ is a set of finite perimeter such that $ F\Delta E \subset \mathcal{I}_\sigma(\pa E),$ then
		\begin{equation}\label{lemmatesi2}
			\mathcal{G}(E) \leq \mathcal{G}(F)+ \Lambda' d_{H^{-1}}(F,E).
		\end{equation}  
	\end{lemma}
	\begin{proof}
		We fix $\sigma_1 \leq  \min \{ \sigma_E, \mathrm{dist}(\pa \Omega, \pa E)\}$ and divide the proof into two steps. \\
		\textit{Step 1:} Proof of  \eqref{lemmatesi1}. 
		
		By Lemma \ref{remgEF} we have $ \xi_{F,E}= \psi + \kappa_E \frac{\psi^2}{2}$. If $\sigma_1$ is sufficiently small, then \begin{equation}\label{fomr07112024}
			\frac{\psi^2}{2}\leq \big(\psi+ \frac{\psi^2}{2}\kappa_E\big)^2= \xi_{F,E}^2 \leq 2 \psi^2,\qquad 2\vert \psi \kappa_E \vert + \psi^2 \kappa_E^2 \leq \frac{1}{16}.
		\end{equation} 
		\textit{Claim 1:} There exists a constant $C(K,K_{el})$ such that for all $\eta>0$,
		\begin{equation}\label{18032025claim1}
			\mathcal{E}(E(u_E^{K_{el}})) \leq  \mathcal{E}(E(u_F^{K_{el}}))+ \frac{C(K,K_{el})}{\eta}\| \xi_{F,E} \|_{H^{-1}(\pa E)}+\eta \| \nabla_{\pa E} \psi \|_{L^2(\pa E)}^2  .
		\end{equation}
		
		Using the definition of $F \rightarrow \mathcal{E}(E(u_F^{k_{el}}))$, the minimality of $ u_E^{K_{el}}$, and the very definition of $d_{H^{-1}}(F,E)$, we have that
		\begin{equation}\label{18032025form1}
			\begin{split}
				\mathcal{E}(E(u_F^{K_{el}})) &= \int_{\Omega \setminus F} Q ( E (u_F^{K_{el}})) \, dx = \int_{\R^2}Q ( E (u_F^{K_{el}})) (\chi_{\Omega}- \chi_F) \, dx \\
				&= \int_{\Omega \setminus E} Q ( E (u_F^{K_{el}})) \, dx + \int_{\R^2} Q ( E (u_F^{K_{el}})) (\chi_E- \chi_F)\,dx \\
				& \geq \mathcal{E}(E(u_E^{K_{el}})) -\int_{\R^2} Q ( E (u_F^{K_{el}}))\circ \pi_{\pa E} (\chi_F- \chi_E)\,dx\\
				& \qquad  + \int_{\R^2}(Q ( E (u_F^{K_{el}}))-Q ( E (u_F^{K_{el}}))\circ \pi_{\pa E} )(\chi_E- \chi_F)\,dx\\
				& \geq \mathcal{E}(E(u_E^{K_{el}}))- C(K_{el}) d_{H^{-1}}(F,E)
				\\
				& \qquad \qquad + \int_{\R^2}(Q ( E (u_F^{K_{el}}))-Q ( E (u_F^{K_{el}}))\circ \pi_{\pa E} )(\chi_E- \chi_F)\,dx.
			\end{split}
		\end{equation}
		We observe that
		\begin{equation}
			\vert Q ( E (u_F^{K_{el}}))(x)-Q ( E (u_F^{K_{el}}))\circ \pi_{\pa E}(x)\vert  \leq C(K_{el}) \vert x - \pi_{\pa E}(x)\vert . 
		\end{equation}
		By this formula and
		using the coarea formula and \eqref{fomr07112024}, we obtain \begin{equation}\label{18032025form2}
			\begin{split}
				& \int_{\R^2}(Q ( E (u_F^{K_{el}}))-Q ( E (u_F^{K_{el}}))\circ \pi_{\pa E} )(\chi_E- \chi_F)\,dx \\
				&\leq  C(K_{el})\int_{\R^2} \vert x- \pi_{\pa E}(x)\vert \chi_{E \Delta F}(x)\, dx \\
				& = C(K,K_{el}) \int_{\pa E} \vert \psi(x)\vert  \int_{0}^{\vert \psi (x)\vert } (1+ t \kappa_E(x))\, dt \, d \mathcal{H}^1_x \\
				& \leq C(K,K_{el}) \| \xi_{F,E}\|^2_{L^2(\pa E)} \leq C(K,K_{el}) \| \xi_{F,E} \|_{H^{-1}(\pa E)} \| \nabla_{\pa E} \xi_{F,E} \|_{L^2(\pa E)} \\
				& \leq \frac{C(K,K_{el})}{\eta} \| \xi_{F,E}\|_{H^{-1}(\pa E)}^2+ \eta \| \nabla_{\pa E} \xi_{F,E}\|_{L^2(\pa E)}^2\\
				& \leq \frac{C(K,K_{el})}{\eta} \| \xi_{F,E}\|_{H^{-1}(\pa E)}^2+ \eta C(K) \| \nabla_{\pa E} \psi \|_{L^2(\pa E)}^2
			\end{split}
		\end{equation}
		where in the third inequality we have used the interpolation inequality; i.e.,
		$$ \| f \|_{L^2(\pa E)} \leq \| f \|^{\frac{1}{2}}_{H^{-1}(\pa E)} \| \nabla f \|_{L^2(\pa E)}^{\frac{1}{2}} \quad \text{for all } f \text{ such that } \int_{\pa E} f \, d\mathcal{H}^1=0, $$ and in the fourth inequality we have used the Young inequality and in the last inequality we have used Lemma \ref{lemmazzo1}.
		Combining \eqref{18032025form1} and \eqref{18032025form2}, we obtain \eqref{18032025claim1}. \\
		\textit{Claim 2:} There exists a constant $C(K)$ such that
		\begin{equation}\label{25032025form1}
			\frac{3J_{\varphi}}{8}\| \nabla_{\pa E} \psi \|_{L^2(\pa E)}^2+ P_{\varphi}(E) \leq P_\varphi(F)+ C(K) d_{H^{-1}}(F,E).
		\end{equation}
		
		Define $ \Psi : \pa E \rightarrow \pa F$ defined as $ \Psi(x): =x+ \psi(x)\nu_E(x)$. The tangential Jacobian is $$J_\tau \Psi = \sqrt{(1+ \psi \kappa_E)^2 +\vert \nabla_{\pa E} \psi\vert^2}.$$ Using the area formula and the expansion of $\nu_F$, see \eqref{nu_Fexp}, we get: \begin{equation}\label{per01102024}
			\begin{split}
				P_{\varphi}(F)=\int_{\pa F} \varphi(\nu_F)d \mathcal{H}^1&= \int_{\Psi(\pa E)} \varphi (\nu_F \circ \Psi \circ \pi_{\pa E} |_{\pa F}) d\mathcal{H}^1\\
				&=\int_{\pa E}  \varphi(\nu_F(\Psi))\sqrt{(1+ \psi \kappa_E)^2+ \vert \nabla_{\pa E} \psi \vert^2} d \mathcal{H}^1\\
				& = \int_{\pa E} \varphi(-\nabla_{\pa E}\psi(x)+(1+ \psi(x)\kappa_{E}(x))\nu_E(x))d \mathcal{H}^1_x.
			\end{split}
		\end{equation} 
		We observe that 
		\begin{equation}
			\begin{split}
				\nabla_{\pa E} ( \nabla_{\pa E} \varphi (\nu_E))= \nabla^2_{\pa E} \varphi (\nu_E) \nabla_{\pa E}( \nu_E )= \kappa_E  \nabla^2_{\pa E} \varphi (\nu_E) \tau_E \otimes \tau_E,
			\end{split}
		\end{equation}
		then \begin{equation}\label{26022025f1}
			\begin{split}
				\div_{\pa E} ( \nabla_{\pa E} \varphi (\nu_E))= \mathrm{Tr} \nabla_{\pa E} (\nabla_{  \pa E} \varphi (\nu_E) )=\kappa_E \nabla_{\pa E}^2 \varphi(\nu_E )\tau_E\cdot \tau_E.
			\end{split}
		\end{equation}
		By the convexity of $ \varphi$ and using \eqref{remform} (up to take $ \sigma_1$ small enough) we have
		\begin{equation}\label{oggi27feb}
			\begin{split}
				&\varphi((1+\psi\kappa_E)\nu_E-\nabla_{\pa E} \psi ) \\
				&\geq \varphi(\nu_E+\psi\kappa_E\nu_E)- \nabla \varphi( \nu_E) \cdot \nabla_{\pa E} \psi  + \frac{J_{\varphi}}{2} \vert  \nabla_{\pa E} \psi \vert^2\\
				& \geq \varphi(\nu_E)+ \psi\kappa_E \nabla \varphi(\nu_E)\cdot \nu_E - \nabla \varphi( \nu_E) \cdot \nabla_{\pa E} \psi  + \frac{J_{\varphi}}{2} \vert  \nabla_{\pa E} \psi \vert^2\\
				& = \varphi(\nu_E)+ \psi \kappa_E \varphi(\nu_E)- \nabla \varphi( \nu_E) \cdot \nabla_{\pa E} \psi+ \frac{J_{\varphi}}{2} \vert  \nabla_{\pa E} \psi \vert^2
			\end{split}
		\end{equation}
		where in the last equality we have used $ \nabla \varphi (x) \cdot x = \varphi (x)$.  Let $\varepsilon>0$ that we choice later, using the divergence theorem and formula \eqref{26022025f1}, we get
		\begin{equation}\label{19032025form1}
			\begin{split}
				\int_{\pa E} &\nabla \varphi ( \nu_E) \cdot \nabla_{\pa E} \psi  = \int_{\pa E} \nabla_{\pa E} \varphi (\nu_E) \cdot \nabla_{\pa E} \psi  =\int_{\pa E} \div_{\pa E}(\nabla_{\pa E} \varphi (\nu_E)) \psi  \\
				&\leq  \int_{\pa E} \div_{\pa E}(\nabla_{\pa E} \varphi (\nu_E)) (\psi + \kappa_E \frac{\psi^2}{2})+ C(K) \int_{\pa E} \psi^2  \\
				& \leq C(K)\| \psi + \kappa_E \frac{\psi^2}{2}\|_{H^{-1}(\pa E)} + C(K) \int_{\pa E} \xi_{F,E}^2 \\
				& \leq C(K) \| \xi_{F,E}\|_{H^{-1}(\pa E)}+ C(K) \| \xi_{F,E}\|_{H^{-1}(\pa E)} \| \nabla_{\pa E} \xi_{F,E} \|_{L^2(\pa E)}\\
				& \leq \frac{C(K)}{\varepsilon} \| \xi_{F,E}\|_{H^{-1}(\pa E)} + \varepsilon \| \nabla_{\pa E} \xi_{F,E} \|_{L^2(\pa E)}^2 \\
				& \leq \frac{C(K)}{\varepsilon} \| \xi_{F,E}\|_{H^{-1}(\pa E)} + \varepsilon C(K) \| \nabla_{\pa E} \psi \|_{L^2(\pa E)}^2
			\end{split}
		\end{equation}
		where in the third inequality we have used the interpolation of $L^2$ between $H^{-1}$ and $H^{1}$, in the fourth inequality we have used Young's inequality and in the last inequality we have used Lemma \ref{lemmazzo1}. Integrating \eqref{oggi27feb} over $\pa E$ and using \eqref{per01102024}, \eqref{19032025form1}, we get
		\begin{equation}
			P_{\varphi}(F) \geq P_{\varphi}(E)- \frac{C(K)}{\varepsilon } \| \xi_{F,E} \|^{2}_{H^{-1}(\pa E)}- C(K) \varepsilon \| \nabla_{\pa E} \psi \|^2_{L^2(\pa E)}+\frac{J_{\varphi}}{2} \| \nabla_{\pa E} \psi\|_{L^2(\pa E)}^2.
		\end{equation}
		Choosing $ \varepsilon= \frac{J_{\varphi}}{8C(K)}$ we obtain \eqref{25032025form1}. 
		
		We choose $\eta= \frac{J_{\varphi}}{8}$. Combining \eqref{18032025claim1}, \eqref{25032025form1} and recalling that $$ d_{H^{-1}}(F,E)= \| \xi_{F,E} \|_{H^{-1}(\pa E)},$$ we obtain \eqref{lemmatesi1}.\\
		\textit{Step 2:} Proof of \eqref{lemmatesi2}.
		
		Let $\widehat{\Lambda}= \widehat{\Lambda}(K,K_{el})$ (to be defined later, see formula \eqref{lacost28042025}). By Lemma \ref{propdadim} applied with $\Lambda=\widehat{\Lambda}$, we obtain $\delta_0= \delta_0 (K,K_{el})$.
		Set 
		\begin{equation}
			\mathcal{J}(F):= \mathcal{G}(F)+ (\Lambda+1)d_{H^{-1}}(F,E).
		\end{equation}
		Fix $\sigma \leq \min \{ \sigma_1, \delta_0\}$. The thesis of the step 2 is equivalent the claim.\\
		\textit{Claim 3:} The set $E$ is the minimizer of the problem
		\begin{equation}\label{lego1}
			\min \left\{ \mathcal{J}(F) \colon F \Delta E \subset {\rm cl} (\mathcal{I}_{\sigma}(\pa E))   \right\}.
		\end{equation}
		
		Existence of a minimizer follows by the direct method of the calculus of variations. Let $F$ be such a minimizer. \\
		\textit{Subclaim:} The minimizer of \eqref{lego1} is an $\widehat{\Lambda}$-minimizer of the $\varphi$-perimeter.
		
		Let $S(K)>0$ denote the Sobolev embedding constant of $H^1(\pa E)$ into $L^{\infty}(\pa E)$, which depends only on $K$.
		Let $ G \subset \R^2$ with $ G \Delta E \subset {\rm cl} (\mathcal{I}_{\sigma}(\pa E))$ and $\vert G \vert =1$. Then, by minimality of $F$, we get:
		\begin{equation}\label{26032025form1}
			\begin{split}
				\mathcal{G}(F)-& \mathcal{G}(G) \leq (\Lambda+1) (d_{H^{-1}}(G,E)-d_{H^{-1}}(F,E))\\
				& \leq  (\Lambda+1) \int_{\R^2} f_G \circ \pi_{\pa E}(x) (\chi_G(x)- \chi_F(x))\,dx 
				\leq (\Lambda+1) \| f_G \circ \pi_{\pa E} \|_{\infty} \vert G \Delta F \vert \\
				& \leq S(K) (\Lambda+1) \| \nabla_{\pa E} f_G \|_{L^2(\pa E)} \vert G \Delta F \vert \leq S(K) (\Lambda+1) \vert G \Delta F \vert
			\end{split}
		\end{equation}
		where $f_G $ is the function that realize the supremum in the definition of $d_{H^{-1}}(G,E)$. We consider now the elastic term: 
		\begin{equation}\label{26032025form2}
			\begin{split}
				\mathcal{E}(E(u_F^{K_{el}}))-\mathcal{E}(E(u_G^{K_{el}}))&=  \int_{\Omega \setminus F} Q(E(u_F^{K_{el}})) dx - \int_{\Omega \setminus G} Q(E(u_G^{K_{el}})) dx\\
				& \leq \int_{\Omega \setminus F} Q(E(u_G^{K_{el}})) dx - \int_{\Omega \setminus G} Q(E(u_G^{K_{el}})) dx \\
				& \leq  K_{el} \vert G \Delta F \vert  ,
			\end{split}
		\end{equation}
		here, the first inequality follows from the minimality of  $u_F^{K_{el}}$. Combining formulas \eqref{26032025form1}, \eqref{26032025form2}, and recalling the definition of $ \mathcal{G}$, we obtain 
		\begin{equation}\label{lminconvfix}
			P_{\varphi}(F)- P_{\varphi}(G) \leq (S(K)(\Lambda+1)+ K_{el}) \vert F \Delta G \vert \text{ for all $ G \Delta E \subset {\rm cl} (\mathcal{I}_{\sigma}(\pa E))$ and $\vert G \vert =1$}.
		\end{equation}
		We now conclude the proof of the subclaim using a standard calibration argument, which we proceed to explain. Let us denote $C_1(K,K_{el}):=(S(K)(\Lambda+1)+ K_{el})$, and fix a set $ G \subset \R^2$.\\
		\textit{Case 1:} $\vert G \Delta F \vert \geq 1$. 
		
		By the minimality of $F$ in \eqref{lego1}, we have
		\begin{equation}
			P_{\varphi}(F) \leq \mathcal{J}(F) \leq \mathcal{J}(E)= \mathcal{G}(E) \leq C_2(K,K_{el}).
		\end{equation}
		Therefore,
		\begin{equation}\label{21052025lego375}
			P_{\varphi}(F) \leq P_{\varphi}(G)+ P_{\varphi}(F) \leq P_{\varphi}(G)+ C_2(K,K_{el}) \leq P_{\varphi}(G)+ C_2(K,K_{el}) \vert G \Delta F\vert.
		\end{equation}
		\textit{Case 2:} $ \vert G \Delta F \vert < 1$.\\
		Define $E_s:= \{  x \in \R^2 \colon d_{E}(x)< s \}$ for $s \in [-\sigma,\sigma]$, and set
		\begin{equation}
			\widehat{G}:= (G \cap E_\sigma) \cup E_{-\sigma}.
		\end{equation}
		Then $\widehat{G} \Delta E \subset {\rm cl} \big( \mathcal{I}_\sigma(\pa E) \big)$ and $$ \widehat{G} \Delta G =\big( G \setminus E_\sigma \big) \cup \big( E_{-\sigma} \setminus G  \big).$$ 
		Since $ F \Delta E \subset {\rm cl} \big( \mathcal{I}_\sigma(\pa E) \big)$, we have $ F \subset {\rm cl} (E_\sigma)$ and ${\rm int}(E_{-\sigma}) \subset F$, yielding
		\begin{equation}\label{2105difsim}
			\vert G \Delta \widehat{G} \vert \leq \vert G \Delta F \vert.
		\end{equation}
		\textit{Subsubclaim:} 
		\begin{equation}\label{slabligngg1212}
			P_{\varphi}(\widehat{G}) \leq P_{\varphi}(G)+ C_3(K) \vert G \Delta \widehat{G} \vert.
		\end{equation}
		We analyze the case $G \cap E_{-\sigma} = E_{-\sigma}$; the other cases are similar. Define the vector field
		$$ Y : \R^2 \rightarrow \R^2, \quad Y:= -\nabla \varphi (\nu_{E_{\sigma}} \circ \pi_{\pa E_{\sigma} } ) \xi$$
		where $\xi \in C^{\infty}_c(\mathcal{I}_{\sigma}(\pa E_{\sigma}))$ and $\xi(x)=1$ for all $x \in \mathcal{I}_{\frac{\sigma}{2}}(\pa E_{\sigma}) $. By the divergence Theorem, we obtain \begin{equation}\label{21052025intpart}
			\begin{split}
				\int_{G \setminus \widehat{G}} \div Y\, dx &= \int_{G \setminus E_\sigma} \div Y \, dx \\
				&= \int_{\pa E_\sigma \cap G} \nu_{E_\sigma}\cdot \nabla \varphi (\nu_{E_{\sigma}}  )- \int_{\pa^* G \cap E_\sigma^c} \nu_G \cdot \nabla \varphi (\nu_{E_{\sigma}} \circ \pi_{\pa E_{\sigma} } ) \xi.
			\end{split}
		\end{equation}
		By the convexity of $\varphi$ and the triangle inequality, we obtain
		\begin{equation}\label{convtringoggi27}
			\nabla \varphi(\nu_{E_{\sigma}} \circ \pi_{\pa E_{\sigma} } )\cdot \nu_G \leq \varphi(\nu_G+ \nu_{E_{\sigma}} \circ \pi_{\pa E_{\sigma} })- \varphi(\nu_{E_{\sigma}} \circ \pi_{\pa E_{\sigma} }) \leq \varphi(\nu_G).
		\end{equation}
		Furthermore, the one homogeneity of $\varphi$ gives
		\begin{equation}\label{homoggi27}
			\nabla \varphi (\nu_{E_{\sigma}} \circ \pi_{\pa E_{\sigma} }) \cdot \nu_{E_{\sigma}} \circ \pi_{\pa E_{\sigma} } = \varphi (\nu_{E_{\sigma}} \circ \pi_{\pa E_{\sigma} }).
		\end{equation}
		Combining  \eqref{21052025intpart}, \eqref{convtringoggi27}, and \eqref{homoggi27}, we obtain 
		\begin{equation}
			P_\varphi(\widehat{
				G}) \leq P_\varphi(G)+ \int_{G \setminus \widehat{G}} \div Y \, d x,
		\end{equation}
		which implies
		\eqref{slabligngg1212}.
		
		We now consider two cases: $ \vert \widehat{G} \vert \geq \vert F \vert $ or $ \vert \widehat{G} \vert \leq \vert F \vert$. We analyze the former; the latter is analogous. For all $s \in [-\sigma,\sigma]$, we have 
		\begin{equation}
			\vert E_{-\sigma} \vert \leq \vert E_{-\sigma} \cap \widehat{G} \vert \leq \vert E_s \cap \widehat{G} \vert \leq \vert E_\sigma \cap \widehat{G} \vert = \vert \widehat{ G } \vert.
		\end{equation}
		By continuity of $ s \rightarrow \vert E_s \cap \widehat{G} \vert$, there exists $ \hat{s} \in [-\sigma,\sigma]$ such that $ \vert E_{\hat{s}} \cap \widehat{G}\vert=\vert F \vert $. 
		Denoting $ G_{\hat{s}}:= E_{\hat{s}} \cap \widehat{G}  \subset \widehat{G}$ and using \eqref{2105difsim}, we get
		\begin{equation}\label{21052025919}
			\begin{split}
				\vert \widehat{G} \Delta G_{\hat{s}} \vert& = \vert \widehat{G} \vert - \vert G_{\hat{s}} \vert = \vert \widehat{G} \vert - \vert F \vert \leq \vert \widehat{G} \Delta F \vert \\
				& \leq \vert G \Delta F \vert+ \vert G \Delta \widehat{G} \vert \leq 2 \vert G \Delta F \vert.
			\end{split}
		\end{equation}
		Applying a calibration argument similar to the one in \eqref{slabligngg1212}, we obtain
		\begin{equation}
			P_\varphi(\widehat{G}_s) \leq P_\varphi(\widehat{G})+ C_4(K) \vert \widehat{G}_s \Delta \widehat{G} \vert.
		\end{equation}
		Combining this with \eqref{2105difsim}, \eqref{slabligngg1212}, and \eqref{21052025919} we conclude
		\begin{equation}\label{210525lambda}
			P_\varphi(\widehat{G}_s) \leq P_\varphi(G)+ C_5(K)\vert G \Delta F\vert.
		\end{equation}
		From \eqref{21052025lego375}, \eqref{lminconvfix}, and \eqref{210525lambda}, we obtain
		\begin{equation}\label{vivaverdi}
			P_\varphi(F) \leq P_\varphi(G)+ \widehat{\Lambda}(K,K_{el})\vert G \Delta F \vert,
		\end{equation}
		where
		\begin{equation}\label{lacost28042025}
			\widehat{\Lambda}(K,K_{el}):=  C_1(K,K_{el})+C_2(K,K_{el}) +C_5(K).
		\end{equation}
		Thus, from \eqref{vivaverdi} we deduce that $ F$ is an $\widehat{\Lambda}$-minimizer of the $\varphi$-perimeter. Therefore subclaim has been proven. 
		
		Applying Lemma \ref{propdadim}, we obtain that $ \pa F $ coincide with a normal graph of a $C^1$ function over $\pa E$, i.e., $$ \pa F = \{ x + \psi(x)\nu_E(x) \colon x \in \pa E \}.$$ Applying Step 1, we get
		\begin{equation}\label{primacro}
			\begin{split}
				\frac{J_\varphi}{4} \|\nabla_{\tau} \psi \|_{L^2(\pa F)}^2+\mathcal{G}(E) &\leq \mathcal{G}(F)+ \Lambda d_{H^{-1}}(F,E) \\
				& \leq \mathcal{J}(F) \leq \mathcal{J}(E)= \mathcal{G}(E)
			\end{split}
		\end{equation}
		where in the last inequality we have used the minimality of $F$. Hence, from \eqref{primacro}, we conclude that $\psi$  must be constant. Since $\psi $ is constant and  $ \vert F \vert= \vert E \vert$, it follows that $\psi=0$, and therefore $F=E$.
	\end{proof}
	
	In the next lemma, we show that every minimizer $F$ of the problem \eqref{probdimin1} is an $\Lambda$-minimizer of the $ \varphi$-perimeter, where $\Lambda$ depends only on $K, K_{el}$. Moreover, we establish that the discrete velocity in $H^{-1}$ is bounded 
	$$\frac{d_{H^{-1}}(E,F)}{h}\leq C.$$
	\begin{lemma}\label{lalalambdamin}
		Let $E \in \mathfrak{H}_{K,\sigma_0}^3(E_0)$ with $ \vert E \vert=1$ and let $ \sigma$ be the constant from \ref{lego123og}. Then, for every minimizer $F$ of the problem \eqref{probdimin1} with  $\eta < \sigma$, the following properties hold:
		\begin{itemize}
			\item[1)]  Let $\Lambda'$ be the constant from Lemma \ref{lego123og}. Then \begin{equation}\label{2803e'arri}
				d_{H^{-1}}(F,E) \leq 2 \Lambda' h.
			\end{equation} 
			\item[2)] There exists a constant $\lambda= \lambda(K,K_{el})$ such that
			\begin{equation}\label{ildalma}
				P_{\varphi}(F) \leq P_{\varphi}(G)+ \lambda\vert G \Delta F \vert \quad \text{ for all } G \subset \R^2.
			\end{equation}
		\end{itemize}
	\end{lemma}
	\begin{proof}
		We divide the proof into two steps.\\
		\textit{Step 1} In this step, we prove \eqref{2803e'arri}.
		
		Using formula \eqref{lemmatesi2} and the minimality of $F$, we obtain
		\begin{equation}
			\mathcal{G}(F)+ \frac{1}{2h} d_{H^{-1}}^2(F,E) \leq \mathcal{G}(E) \leq \mathcal{G}(F) + \Lambda'd_{H^{-1}}(F,E).
		\end{equation}
		Hence, inequality \eqref{2803e'arri} follows. \\
		\textit{Step 2} In this step, we prove \eqref{ildalma}.\\
		\textit{Claim:} For every set $G \subset \R^2 $ such that $ G \Delta E \subset {\rm cl} (\mathcal{I}_{\eta}(\pa E))$ and $\vert G \vert =1$, the following inequality holds:
		\begin{equation}\label{lamdaminvinc}
			\mathcal{G}(F) \leq \mathcal{G}(G)+ 3\Lambda' (d_{H^{-1}}(G,E)- d_{H^{-1}}(F,E)).
		\end{equation}
		
		\textit{Case 1)} $ d_{H^{-1}}(G,E) > 4 \Lambda' h$.\\
		By \eqref{2803e'arri}, we have
		\begin{equation}
			2 d_{H^{-1}}(F,E) \leq 4 \Lambda' h < d_{H^{-1}}(G,E) 
		\end{equation}
		and hence,
		\begin{equation}\label{nadorla}
			d_{H^{-1}}(F,E) \leq d_{H^{-1}}(G,E)- d_{H^{-1}}(F,E).
		\end{equation}
		Using the minimality of $F$ along with formulas \eqref{lemmatesi2} and \eqref{nadorla}, we deduce
		\begin{equation}\label{travaglio1}
			\begin{split}
				\mathcal{G}(F) &\leq \mathcal{G}(F)+ \frac{1}{2 h} d_{H^{-1}}^2(F,E) \leq \mathcal{G}(E)\\
				& \leq \mathcal{G}(G)+ \Lambda' d_{H^{-1}}(F,E) \leq \mathcal{G}(G)+ \Lambda' (d_{H^{-1}}(G,E)- d_{H^{-1}}(F,E)).
			\end{split}
		\end{equation}
		
		\textit{Case 2)} $ d_{H^{-1}}(G,E) \leq  4 \Lambda' h$.\\
		Using the minimality of $F$ and inequality \eqref{2803e'arri}, we obtain
		\begin{equation}\label{trvaglio2}
			\begin{split}
				\mathcal{G}(F)- \mathcal{G}(G) & \leq \frac{1}{2 h} (d_{H^{-1}}(G,E)+ d_{H^{-1}}(F,E) )(d_{H^{-1}}(G,E)- d_{H^{-1}}(F,E) ) \\
				&\leq 3 \Lambda' (d_{H^{-1}}(G,E)- d_{H^{-1}}(F,E) ). 
			\end{split}
		\end{equation}
		The conclusion of the claim follows from inequalities \eqref{travaglio1} and \eqref{trvaglio2}.\\
		\textit{Claim:} There exits $\lambda_1= \lambda_1(K,K_{el})$ such that for every set $G \subset \R^2 $ with $ G \Delta E \subset {\rm cl} (\mathcal{I}_\sigma(\pa E))$ and $\vert G \vert =1$, the following holds:
		\begin{equation}\label{22052025caffè1}
			P_{\varphi}(F) \leq P_{\varphi}(G)+ \lambda_1 \vert F \Delta G \vert.
		\end{equation}
		
		Using the definition of $\mathcal{G}$ (see formula \eqref{defG}), and the previous claim, we have
		\begin{equation}
			P_{\varphi}(F)- P_{\varphi}(G) \leq \mathcal{E}(E(u_G^{K_{el}}))- \mathcal{E}(E(u_F^{K_{el}})) + 3 \Lambda' (d_{H^{-1}}(G,E)- d_{H^{-1}}(F,E) )
		\end{equation}
		The claim then follows by applying the same reasoning used in formulas \eqref{26032025form1} (to estimate the difference between $d_{H^{-1}}(G,E)$ and $d_{H^{-1}}(F,E)$) and \eqref{26032025form2} (to estimate the difference between $ \mathcal{E}(E(u_G^{K_{el}}))$ and $\mathcal{E}(E(u_F^{K_{el}}))$).\\
		\textit{Claim:} There exits $\lambda= \lambda(K,K_{el})$ such that for every $G \subset \R^2$, the inequality \eqref{ildalma}.
		
		This claim follows by \eqref{22052025caffè1} and adapting the same argument used in the subclaim of Step 2 in Lemma \ref{lego123og}.
	\end{proof}
	
	\subsection{Estimate for the heightfunction} In this subsection, we prove that $\pa F$ coincides with the graph of a smooth function $\psi : \pa E \rightarrow \R$, where $F$ is a  minimizer of \eqref{probdimin1}. Moreover, we establish regularity estimates for $\psi$.

	We begin with a remark that provides an analogue of formula \eqref{Cdacit} for the anisotropic curvature defined in \eqref{ancurvE}.
	\begin{remark}
		Let $E \in \mathfrak{C}^{3}_{K,\sigma_0}(E_0)$ and let $A \subset \R^2$ be a set of class $C^2$ such that $ \pa A $ is a normal graph over $\pa E$, i.e. $$ \pa A = \{ x + \psi(x)\nu_E(x)\colon x \in \pa E\}.$$ Let $g \in C^\infty(\R^2 \setminus \{0 \})$ be the function defined in \eqref{ancurvE}. Then, using formula \eqref{nu_Fexp} and the Taylor expansion of $g$, we obtain for all $x \in \pa E$ \begin{equation}\label{gepsansione}
			g(\nu_A (x+ \psi(x)\nu_E(x)))= g(\nu_E(x))+ R_1(\psi(x)\kappa_E(x), \pa_{\tau} \psi(x), \nu_E(x)), 
		\end{equation}
		where $R_1 \in C^\infty$.
		Using formulas \eqref{Cdacit}, \eqref{Rdacit}, \eqref{ancurvE} and \eqref{gepsansione}, we obtain
		\begin{equation}\label{L'ESPANSOIONE}
			\kappa_A^{\varphi}(x+ \psi(x)\nu_E(x))= -g(\nu_E(x)) \pa^2_{\tau} \psi(x) + \kappa_E^{\varphi}(x)+ R(x) \,\, x \in \pa E,
		\end{equation}
		where \begin{equation}\label{ILRESTO}
			R=r_0(\psi, \pa_{\tau} \psi,\kappa_E, \nu_E)+ r_1(\psi \kappa_E, \pa_\tau \psi, \nu_E) \pa_\tau^2 \psi+ r_2(\psi \kappa_E, \pa_\tau \psi, \nu_E) \pa_\tau(\psi \kappa_E)
		\end{equation}
		and $r_0 , r_1,r_2 $ are smooth functions satisfying $$ r_0(0,0,\cdot,\cdot)= r_1(0,0,\cdot)= r_2(0,0,\cdot)=0.$$
	\end{remark}
	
	We now state and prove three propositions that will enable us to prove the Theorem \ref{MainThm375}.
	\begin{proposition}\label{primapropreg}
		Let $E \in \mathfrak{H}^4_{K,\sigma_0}(E_0)$.Then there exist constants $\eta_0= \eta_0(K,K_{el})$, $h_0=h_0(K,K_{el}) $ and $C=C(K,K_{el})$  such that, if $0 < h \leq h_0$ then any minimizer of the problem \eqref{probdimin1} (with $\eta= \eta_0$) $ F\subset \R^2$ has the property that $ \pa F $ coincides with the graph of a smooth function $\psi : \pa E \rightarrow \R$ satisfying
		\begin{equation}\label{tesipropPreg}
			\| \psi \|_{L^2(\pa E)} \leq C h^{\frac{3}{4}},\quad\| \nabla_{\pa E} \psi \|_{L^2(\pa E)} \leq  C \sqrt{h} , \quad \| \kappa_{F}^{\varphi}\|_{H^1(\pa F)} \leq C.
		\end{equation}
	\end{proposition}
	\begin{proof}
		Let $\sigma$ be the constant from Lemma \ref{lego123og}, and let $\delta_0$  be the constant obtained in Lemma \ref{probdimin1}  for $\Lambda= \lambda(K,K_{el})$, where $\lambda(K,K_{el})$ is the constant defined in Lemma \ref{lalalambdamin} (see formula \eqref{ildalma}). We set $\eta_0:= \min \{ \sigma, \delta_0\}$. Let $F$ be a minimizer of \eqref{probdimin1} for $\eta= \eta_0$. By Lemma \ref{lalalambdamin}, we have that $F$ is a $\lambda(K,K_{el})$-minimizer of the $\varphi$-perimeter. Applying Lemma \ref{propdadim}, we obtain $$\pa F = \{ x+ \psi (x)\nu_E(x)\colon x \in \pa E\},$$ where $\psi \in C^1(\pa E)$ and $\| \psi \|_{C^1(\pa E)} \leq C(K,K_{el})$. Using formula \eqref{lemmatesi1}, the minimality of $F$, and formula \eqref{2803e'arri}, we get
		\begin{equation}\label{01042025formz1}
			\| \nabla_{\pa E} \psi \|_{L^2(\pa E)} \leq C(K) \sqrt{h}.
		\end{equation} 
		Using formulas \eqref{tesilemmazzzo}, \eqref{perchèiosiloso}, and the Poincaré inequality, we have
		\begin{equation}\label{ludovicoariosto}
			\| \psi \|_{L^2(\pa E)} \leq C(K)\| \xi_{F,E} \|_{L^2(\pa E)} \leq C(K) \| \nabla_{\pa E} \xi_{F,E} \|_{L^2(\pa E)} \leq C(K) \| \nabla_{\pa E} \psi \|_{L^2(\pa E)} ,
		\end{equation}
		where $\xi_{F,E}= \psi+ \frac{\psi \kappa_E^2}{2}$.
		Therefore by the Sobolev embedding, \eqref{01042025formz1} and \eqref{ludovicoariosto}, we obtain 
		\begin{equation}
			\| \psi \|_{L^\infty(\pa E)} \leq C(K)  \| \psi \|_{H^1(\pa E)}  \leq C(K) \sqrt{h}.
		\end{equation}
		Hence, for $h_0$ small enough, we have $ \pa F \Subset \mathcal{I}_{\eta}(\pa E)$. We are now in a position to compute the Euler–Lagrange equation for the functional $ F \rightarrow \mathcal{F}_h(F,E)$. Applying formula \eqref{FeulPANI}, Proposition \ref{propELDH-1}, and Proposition \ref{PropelELASt} we obtain
		\begin{equation}\label{01042025euler1}
			\kappa_F^{\varphi}(y)- Q(E(u_{F}^{K_{el}}))(y)+ \frac{d_{H^{-1}}(F,E)}{h}f (\pi_{\pa E}(y))=L \quad \text{ for all }y \in \pa F
		\end{equation}
		where $f \in H^1(\pa E)$  is the function that attains the supremum in \eqref{d_{H^{-1}}}, and $L$ is the Lagrange multiplier. Integrating equation \eqref{01042025euler1} over $\pa F$, we get
		\begin{equation}\label{gelato21052025}
			\begin{split}
				L  &\leq \frac{1}{P(F)} \bigg[ \int_{\pa F} \kappa_F^{\varphi}\, d \mathcal{H}^1+ C(K,K_{el})+\frac{d_{H^{-1}}(F,E)}{h} \int_{\pa F } f \circ \pi_{\pa E} \, d \mathcal{H}^1    \bigg]\\
				& = \frac{1}{P(F)} \bigg[ \int_{\pa F} \kappa_F^{\varphi}\, d \mathcal{H}^1+ C(K,K_{el}) \\
				& \qquad\qquad \qquad+\frac{d_{H^{-1}}(F,E)}{h} \int_{\pa E} f(x) \sqrt{ (1+ \psi(x)\kappa_E(x))^2+ \vert\nabla_{\pa E} \psi(x) \vert^2}  \, d \mathcal{H}^1_x \bigg] \\
				& \leq \frac{1}{\sqrt{4 \pi}} \bigg[ C_{\varphi}+2 \Lambda' C(K)    \bigg]= C(K,K_{el})
			\end{split} 
		\end{equation}
		where we have used: a change of variable $x \in \pa E \rightarrow x+ \psi(x)\nu_E(x) \in \pa F$, whose tangential Jacobian is $x \rightarrow \sqrt{ (1+ \psi(x)\kappa_E(x))^2+ \vert\nabla_{\pa E} \psi(x) \vert^2} $, the isoperimeteric inequality $ 1= \vert F \vert \leq \frac{1}{4 \pi} P(F)^{2}$, Lemma \ref{lGaussB}, formula \eqref{2803e'arri}, and the bound $\| f \|_{L^2(\pa E)} \leq C(K) $. By the Euler–Lagrange equation and \eqref{gelato21052025}, we can now estimate the $L^2$ norm of $\kappa_F^\varphi$, and we get \begin{equation}\label{01042025formz2}
			\| \kappa_F^{\varphi} \|_{L^2(\pa F)} \leq C(K,K_{el}).
		\end{equation}  Differentiating the equation \eqref{01042025euler1}, we obtain 
		\begin{equation}
			\nabla_{\pa F} \kappa_F^{\varphi}(y)- \nabla_{\pa F} Q(E(u_F^{K_{el}}))(y)+ \frac{d_{H^{-1}}(F,E)}{h} \nabla_{\pa F} f (\pi_{\pa E}(y))=0  \quad \text{ for all }y \in \pa F.
		\end{equation}
		Using this equation, formula \eqref{cappelloliberoF}, and $ \| \nabla_{\pa E} f \|_{L^2(\pa E)} \leq 1$, we get \begin{equation}\label{01042025formz3}
			\| \nabla_{\pa F} \kappa_{F}^{\varphi} \|^2_{L^2(\pa F)} \leq C(K,K_{el})+ \int_{\pa E} \frac{\vert \nabla_{\pa E} f \vert^2}{\sqrt{(1+ \psi \kappa_E)+\vert\nabla_{\pa E} \psi  \vert^2}}\, d \mathcal{H}^1 \leq C(K,K_{el}).
		\end{equation}
		From formulas \eqref{lemmazzo1} and \eqref{01042025formz1}, we deduce $$ \| \nabla_{\pa E} \xi_{F,E} \|_{L^2(\pa E)} \leq C \sqrt{h}.$$
		Now using formula \eqref{fomr07112024} and the interpolation of $L^2(\pa E)$ between $H^{1}(\pa E)$ and $H^{-1}(\pa E)$, we obtain
		\begin{equation}\label{dacitare}
			\frac{1}{\sqrt{2}}\| \psi \|_{L^2(\pa E)} \leq  \| \xi_{F,E} \|_{L^2(\pa E)} \leq \| \nabla_{\pa E} \xi_{F,E} \|_{L^2(\pa E)}^{\frac{1}{2}} \| \xi_{F,E} \|_{H^{-1}(\pa E)}^{\frac{1}{2}} \leq C(K,K_{el})h^{\frac{3}{4}}
		\end{equation}
		where we have used  formula \eqref{2803e'arri}, i.e., $d_{H^{-1}}(F,E)= \| \xi_{F,E} \|_{H^{-1}(\pa E)} \leq 2 \Lambda' h$.
		Finally, the estimate \eqref{tesipropPreg} follows from \eqref{01042025formz1}, \eqref{01042025formz2}, \eqref{01042025formz3}, and \eqref{dacitare}.
	\end{proof}
	
	An important consequence of Proposition \ref{primapropreg} is that the boundary of any minimizer of problem \eqref{probdimin1} does not intersect the boundary of the constraint $\pa \mathcal{I}_{\eta_0}(\pa E)$ for $h \leq h_0$. This allows us to write the Euler-Lagrange equation for $F$, and by applying  \eqref{eqv} and \eqref{01042025euler1}, we obtain:
	\begin{equation}\label{LEEULERO}
		\left\{
		\begin{aligned}
			& \kappa_F^{\varphi}- Q(E(u_F^{K_{el}}))+ \frac{d_{H^{-1}}(F,E)}{h} f \circ \pi_{\pa E}=L & \text{ on } \pa F  , \\
			& -\Delta_{\pa E} f=\frac{ \xi_{F,E} }{d_{H^{-1}}(F,E)} & \text{ on }  \pa E 
		\end{aligned}
		\right.
	\end{equation}
	where $\xi_{F,E}$ is defined in \eqref{lafunzgEF}(see also formula \eqref{lafunzgEFgraph}) and $L$ is the Lagrange multiplier. We remark that if $E$ is $C^{5,\gamma}$-regular for some $\gamma \in (0,1)$, then
	by the elliptic regularity theory implies that $F$ is also $C^{5}$-regular. Using formula \eqref{L'ESPANSOIONE}, we can combine the two equations above into the following single equation:
	\begin{multline}\label{LEEULEROCOMB}
		\frac{1}{h}\big( \psi(x)+ \kappa_E(x) \frac{\psi(x)^2}{2} \big)\\= \pa_{\tau}^2 (-g(\nu_E(x))\pa_{\tau}^2 \psi(x)+ \kappa_E^\varphi(x))- \pa_{\tau}^2 \big(Q(E(u_F^{K_{el}}))(x+ \psi(x)\nu_E(x))\big)+ \pa_{\tau}^2 R(x) 
	\end{multline}
	for all $ x \in \pa E$, where $R$ is define in \eqref{ILRESTO}.
	
	Let us recall a lemma that will be useful in the upcoming proofs; see \cite[Lemma 2.3]{CFJKsd}, \cite[Lemma 2.5 \& Proposition 2.6]{JN}.
	\begin{lemma}\label{michele}
		Let $A \subset \R^2$ be a set of class $C^5$ and such that $A \in \mathfrak{C}_M^2(E_0)$. For all $f \in C^4(\pa A)$ it holds
		\begin{equation}
			\begin{split}
				&\| f \|_{H^1(\pa A)} \leq C \big(    \| \pa_\tau f \|_{L^2(\pa A)}+ \| f \|_{L^\infty(\pa A)} (1+ \|   \kappa_E \|_{L^2(\pa A)})\big),\\
				&\| f \|_{H^2(\pa A)} \leq C \big(    \| \pa^2_\tau f \|_{L^2(\pa A)}+ \| f \|_{L^\infty(\pa A)} (1+ \|  \pa_\tau \kappa_A \|_{L^2(\pa A)})\big),\\
				& \| f \|_{H^3(\pa A)} \leq C \big(    \| \pa^3_\tau f \|_{L^2(\pa A)}+ \| f \|_{L^\infty(\pa A)} (1+ \|  \pa_\tau^2 \kappa_A \|_{L^2(\pa A)})\big),\\
				& \| f \|_{H^4(\pa A)} \leq C \big(    \| \pa^4_\tau f \|_{L^2(\pa A)}+ \| f \|_{L^\infty(\pa A)} (1+ \| \pa^3_\tau \kappa_A \|_{L^2(\pa A)}) \big),
			\end{split}
		\end{equation}
		where $C$ is a universal constant.
	\end{lemma}
	
	In the next proposition, we will prove a sharp estimate for the $L^2$-norm of the heightfunction in \eqref{tesipropPreg}, namely: $$\| \psi \|_{L^2(\pa E)} \precsim h.$$
	
	\begin{proposition}\label{Mthmproof1}
		Let $E$ be a set of class $C^5$ such that $E \in \mathfrak{H}^4_{K,\sigma_0}(E_0)$ and $\| \pa_{\tau}^3 \kappa_E^{\varphi} \|_{L^2(\pa E)} \leq \frac{K}{h^{\frac{1}{4}}}$.  Let $F \subset \R^2$ be a minimizer of \eqref{probdimin1} for $\eta=\eta_0$, where $\eta_0$ is given in Proposition \ref{primapropreg}. Then, for the heightfunction in \eqref{tesipropPreg}, we have
		\begin{equation}\label{Lego375}
			\| \psi \|_{L^2(\pa E)} \leq C_1 h, \quad \| \psi \|_{H^4(\pa E)} \leq C_1
		\end{equation}
		for all $h \leq h_0$ where $h_0$ is the constant from Proposition \ref{primapropreg}. The constant $C_1$ depends on $K$ and $K_{el}$.
	\end{proposition}
	
	\begin{proof}
		By the assumption on $E$, we deduce 
		\begin{equation}\label{guccini}
			\nu_E \in H^3(\pa E) \text{ and } \| \nu_E \|_{H^3(\pa E)} \leq C(K).
		\end{equation}
		We multiply the Euler equation \eqref{LEEULEROCOMB} by $\pa_\tau^4 \psi$ and integrate over
		$\pa E$, obtaining
		\begin{equation}\label{napoli1}
			\begin{split}
				\frac{1}{h} \int_{\pa E} \vert \pa_{\tau}^2 \psi \vert ^2  +& \int_{\pa E} g(\nu_E) \vert \pa_{\tau}^4 \psi \vert^2= \frac{1}{h}\int_{\pa E}\kappa_E \frac{\psi^2}{2} \pa_{\tau}^4 \psi - \int_{\pa E} \pa_{\tau}^2 g(\nu_E) \pa_{\tau}^2 \psi \pa_{\tau}^4 \psi\\
				& \qquad\qquad  -\int_{\pa E}\pa_{\tau}g(\nu_E) \pa_{\tau}^3 \psi \pa_{\tau}^4 \psi + \int_{\pa E} \pa_\tau^2 \kappa_E^{\varphi} \pa_{\tau}^4 \psi \\
				& \qquad\qquad  - \int_{\pa E}\pa_\tau^2 (Q(E(u_F^{K_{el}})))(\cdot+ \psi (\cdot)\nu_{E}(\cdot)) \pa_\tau^4 \psi +\int_{\pa E} \pa_{\tau}^2 R \pa_\tau^4 \psi.
			\end{split}
		\end{equation}
		We now proceed to estimate the right-hand side of the above equation. Let us fix $\varepsilon>0$ to be chosen later.\\
		\textit{Estimate of $\frac{1}{h}\int_{\pa E}\kappa_E \frac{\psi^2}{2} \pa_{\tau}^4 \psi  $.}
		
		Using the Cauchy–Schwarz and Young inequalities, together with formula \eqref{tesipropPreg} and the Sobolev embedding, we obtain
		\begin{equation}\label{napoli2}
			\begin{split}
				\frac{1}{h} \int_{\pa E} \kappa_E \psi^2 \pa_{\tau}^4 \psi &\leq \frac{C(K)}{h}\| \psi \|_{L^\infty(\pa E)}\| \psi \|_{L^2(\pa E)} \| \pa_{\tau}^4 \psi \|_{L^2(\pa E)}  \\
				&\leq \frac{C(K,\varepsilon)}{h^2} \| \psi \|^2_{L^\infty(\pa E)} \| \psi \|^2_{L^2(\pa E)}+ \varepsilon \| \pa_{\tau}^4 \psi\|_{L^2(\pa E)}^2\\
				&\leq \frac{C(K,\varepsilon)}{h^2} h^{2 (\frac{1}{2}+\frac{3}{4})}+ \varepsilon \| \pa_{\tau}^4 \psi\|_{L^2(\pa E)}^2\\
				& \leq C(K,\varepsilon)+ \varepsilon \| \pa_{\tau}^4 \psi \|^2_{L^2(\pa E)}.
			\end{split}
		\end{equation}
		\textit{Estimate of $\int_{\pa E} \pa_{\tau}^2 g(\nu_E) \pa_{\tau}^2 \psi \pa_{\tau}^4 \psi  $.}
		
		Using the Cauchy–Schwarz and Young inequalities, we obtain
		\begin{equation}\label{napoli3}
			\begin{split}
				\int_{\pa E} \pa_\tau^2 g(\nu_E) \pa_{\tau}^2 \psi \pa_{\tau}^4 \psi &\leq \| \pa_\tau^2 g(\nu_E) \|_{L^\infty(\pa E)} \| \pa_\tau^2 \psi \|_{L^2(\pa E)} \| \pa_{\tau}^4 \psi \|_{L^2(\pa E)} \\
				& \leq C(K,\varepsilon) \| g(\nu_E) \|_{H^3(\pa E)}^2 \| \pa_\tau^2 \psi \|_{L^2(\pa E)}^2+ \varepsilon \| \pa^4_\tau \psi \|^2_{L^2(\pa E)}\\
				& \leq C(K,\varepsilon) \| \pa_\tau^2 \psi \|_{L^2(\pa E)}^2+ \varepsilon \| \pa^4_\tau \psi \|^2_{L^2(\pa E)}
			\end{split}
		\end{equation}
		where in the last inequality we have used the smoothness of $ g $ is smooth and formula \eqref{guccini}.\\
		\textit{Estimate of $\int_{\pa E} \pa_{\tau} g(\nu_E) \pa_\tau^3 \psi \pa_\tau^4 \psi$. }
		
		Recalling the interpolation inequality (see Proposition \ref{PROPINTER}), 
		$$ \| \pa_\tau^3 \psi \|_{L^2(\pa E)} \leq C \| \psi \|_{H^4(\pa E)}^{\frac{3}{4}} \|  \psi \|_{L^2(\pa E)}^{\frac{1}{4}} $$
		and using the Cauchy–Schwarz and Young inequalities, we obtain
		\begin{equation}\label{napoli4}
			\begin{split}
				\int_{\pa E} &\pa_\tau g(\nu_E) \pa_\tau^3 \psi \pa_\tau^4 \psi \leq C(K) \| \pa_\tau^3 \psi \|_{L^2(\pa E)} \|  \pa_\tau^4 \psi \|_{L^2(\pa E)} \\
				& \leq C(K,\varepsilon) \| \pa _\tau^3 \psi \|_{L^2(\pa E)}^2 + \frac{\varepsilon}{2} \| \pa_\tau^4 \psi \|_{L^2(\pa E)}^2 \\
				& \leq C(K,\varepsilon) \| \psi \|_{L^2(\pa E)}^2 + \varepsilon \|  \psi \|_{H^4(\pa E)}^2\\
				& \leq C(K,\varepsilon) \| \psi \|^2_{L^2(\pa E)}+ \varepsilon C\big( \| \pa_\tau^4 \psi \|_{L^2(\pa E)}+ \| \psi \|_{L^\infty(\pa E)}(1+\| \pa^3_\tau \kappa_E \|_{L^2(\pa E)}) \big)^2\\
				& \leq C(K,\varepsilon) \| \psi \|^2_{L^2(\pa E)} + \varepsilon C \| \pa^4_\tau \psi \|^2_{L^2(\pa E)}+ C(K),
			\end{split}
		\end{equation}
		where in the fourth inequality we have used Lemma \ref{michele} and \eqref{tesipropPreg} that gives $$ \| \psi \|_{L^\infty(\pa A)} (1+\|\pa^3_\tau \kappa_E \|_{L^2(\pa E)})^2 \leq C(K) h^{\frac{1}{2}} \frac{1}{h^{\frac{1}{2}}}\leq C(K). $$\\
		\textit{Estimate of $\int_{\pa E} \pa_\tau^2 \kappa_E^{\varphi} \pa_{\tau}^4 \psi$.}
		
		Using the Cauchy–Schwarz and Young inequalities and the bound $ \| \pa_\tau^2 \kappa_E^{\varphi} \|_{L^2(\pa E)} \leq C(K)$, we obtain
		\begin{equation}\label{napoli5}
			\int_{\pa E} \pa_\tau^2 \kappa_E^{\varphi} \pa_{\tau}^4 \psi \leq \| \pa^2_\tau \kappa_E^\varphi \|_{L^2(\pa E)} \| \pa^4_\tau \psi \|_{L^2(\pa E)} \leq C(K,\varepsilon)+ \varepsilon \| \pa^4_\tau \psi \|^2_{L^2(\pa E)}.
		\end{equation}
		\textit{Estimate of $ \int_{\pa E}\pa_\tau^2 (Q(E(u_F^{K_{el}})))(\cdot+ \psi (\cdot)\nu_{E}(\cdot)) \pa_\tau^4 \psi $.}\\
		\textit{Claim:} It is holds \begin{equation}\label{ciserviràsicuro}
			\| \pa_\tau^2 (Q(E(u_F^{K_{el}})))(\cdot+ \psi (\cdot)\nu_{E}(\cdot)) \|_{L^2(\pa E)} \leq C(K,K_{el})\big(\|  \psi \|_{L^2(\pa E )}+ \| \pa_{\tau}^2 \psi \|_{L^2(\pa E )}\big).
		\end{equation}
		
		Set $\hat F(x):=Q(E(u_F^{K_{el}}))(x) $ for all $x \in \Omega$.
		Then, for all $ x \in \pa E$, we have
		\begin{multline}
			\pa_\tau \big(\hat F(x+ \psi(x)\nu_E(x))\big) 
			= \nabla \hat F(x + \psi(x)\nu_E(x) )\cdot  (\pa_{\tau} \psi (x) \nu_E(x)+ (1+ \psi(x)\kappa_E(x))\tau_E(x)    ).
		\end{multline}
		Moreover,
		\begin{equation}
			\pa_\tau \nabla \hat F(x+\psi(x)\nu_E(x))= \nabla^2 \hat F(x+ \psi(x)\nu_E(x)) [ (1+ \psi(x) \kappa_E(x))\tau_E(x)+ \pa_\tau \psi(x)\nu_E(x)  ]
		\end{equation}
		and 
		\begin{equation}
			\pa_{\tau }[   (1+ \psi\kappa_E)\tau_E+ \pa_{\tau} \psi \nu_E ] 
			= \tau_E[2 \kappa_E\pa_\tau \psi+ \psi \pa_\tau \kappa_E]+ \nu_E [\kappa_E + \psi \kappa_E^2+ \pa_\tau^2 \psi]. 
		\end{equation}
		Therefore, by the Leibniz rule, we obtain
		\begin{equation}\label{30042025form1}
			\begin{split}
				\pa_\tau^2 &\big(\hat F(x+ \psi(x)\nu_E(x))\big) \\
				&= \nabla^2 \hat F(x+ \psi(x)\nu_E(x))G (x, \psi(x)\kappa_E(x), \pa_\tau \psi(x)  ) \cdot G (x,\psi(x)\kappa_E(x), \pa_\tau \psi(x)  ) \\
				& \quad+ \nabla \hat F(x+ \psi(x)\nu_E(x)) \cdot \bar{G}(x,\pa_\tau \psi (x) \kappa_E(x), \psi (x) \pa_\tau \kappa_E(x), \psi(x) \kappa_E^2(x), \kappa_E^2(x), \pa^2_\tau \psi(x)   ),
			\end{split}
		\end{equation}
		where $G \in C^\infty(\R^3) $ and $ \bar G \in C^\infty(\R^6,\R^2)$ satisfy $G(\cdot,0,0)=0$ and $ \bar G(\cdot, 0,0,0,0,0)=0$. Using formula \eqref{30042025form1} and recalling the very definition of $\hat F$ and that $u_F^{K_{el}} \in \mathfrak{C}^{3,\frac{1}{4}}_{K_{el}}(\Omega,\R^2) $, see \eqref{minelast} and using the Soblev embedding we can estimate the $L^2(\pa E)$-norm of $ x \rightarrow \pa_\tau^2 \big( Q(E(u_F^{K_{el}}))(x+ \psi(x)\nu_E(x))  \big) $ and we obtain the claim.
		
		Using the claim, along with the Cauchy–Schwarz and Young inequalities and Lemma \ref{michele}, we finally obtain
		\begin{equation}\label{napoli6}
			\int_{\pa E}\pa_\tau^2 (Q(E(u_F^{K_{el}})))(\cdot+ \psi (\cdot)\nu_{E}(\cdot)) \pa_\tau^4 \psi  \leq C(K,K_{el},\varepsilon) \| \pa_\tau^2 \psi \|_{L^2(\pa E )}^2+ \varepsilon \| \pa^4_\tau \psi \|_{L^2(\pa E)}^2.
		\end{equation}
		\textit{Estimate of $\int_{\pa E} \pa_\tau^2 R \pa^4_\tau \psi$.}
		
		Using the Cauchy–Schwarz and Young inequalities, we obtain
		\begin{equation}
			\int_{\pa E} \pa_\tau^2 R \pa^4_\tau \psi \leq \| \pa^2_\tau R \|_{L^2(\pa E)} \|\pa_\tau^4 \psi \|_{L^2(\pa E)}\leq C(\varepsilon) \|  \pa_\tau^2 R \|_{L^2(\pa E)}^2+ \varepsilon \| \pa^4_\tau \psi \|_{L^2(\pa E)}^2.
		\end{equation}
		Hence we need to estimate $ \|  \pa_\tau^2 R \|_{L^2(\pa E)}$.
		To estimate this term, we recall the form of $R$, see \eqref{ILRESTO}. Applying the Leibniz rule, we obtain
		\begin{equation}\label{04042025f1}
			\begin{split}
				\| \pa_\tau^2 R \|_{L^2(\pa E)} \leq &C \sum_{j+k=2} \|  \pa_\tau^j r_1(\psi\kappa_E, \pa_\tau \psi, \nu_E)  \pa_\tau^{2+k} \psi  \|_{L^2(\pa E)}\\
				&\qquad +C \sum_{j+k=2} \| \vert \pa_\tau^j r_2 (\psi \kappa_E, \pa_\tau \psi, \nu_E) \vert \vert \pa_\tau^{1+k}(\psi \kappa_E) \vert \|_{L^2(\pa E)} \\
				&\qquad + \| r_0(\psi, \pa_\tau \psi, \kappa_E, \nu_E) \|_{H^2(\pa E)}.
			\end{split}
		\end{equation}
		Let $j,k \in \N$ be such that $j+k = 2$, we apply Lemma \ref{LemmLeibinz} with $f_1=r_1(\psi \kappa_E,\pa_\tau \psi,\nu_E) $ and $f_2=\pa_\tau \psi $ to estimate
		\begin{equation}\label{04042025f2}
			\begin{split}
				\| \pa_{\tau}^j (r_1(\psi \kappa_e,\pa_{\tau} \psi, \nu_E)) \pa_{\tau}^{2+k} \psi \|_{L^2(\pa E)} \leq& C \| r_1(\psi \kappa_E, \pa_\tau \psi, \nu_E )  \|_{L^\infty(\pa E)} \| \psi \|_{H^4(\pa E)}\\
				& \quad + C \| \psi \|_{C^1(\pa E)} \| r_1(\psi \kappa_E, \pa_{\tau} \psi, \nu_E) \|_{H^3(\pa E)}.
			\end{split}
		\end{equation}
		Similarly, with $f_1= r_2(\psi\kappa_E,\pa_\tau \psi ,\nu_E) $ and $f_2= \psi \kappa_E$
		\begin{equation}\label{04042025f3}
			\begin{split}
				\| \pa_{\tau}^j ( r_2(\psi \kappa_E, \pa_{\tau} \psi, \nu_E)  ) \pa_{\tau}^{1+k}(\psi \kappa_E) \|_{L^2(\pa E)} \leq & C \| r_2(\psi \kappa_E, \pa_{\tau} \psi, \nu_E) \|_{L^\infty(\pa E)} \| \psi \kappa_E \|_{H^3(\pa E)}  \\
				& \quad +C \| \psi \kappa_E \|_{L^\infty(\pa E)} \| r_2 (\psi \kappa_E, \pa_{\tau} \psi, \nu_E) \|_{H^3(\pa E)}.
			\end{split}
		\end{equation}
		Since $r_i$ is smooth and satisfies $r_i(0,0,\cdot)=0$, we have that
		\begin{equation}
			\| r_i (\psi \kappa_E, \pa_\tau \psi, \nu_E)\|_{L^\infty(\pa E)} \leq C \| \psi \|_{C^1(\pa E)}.
		\end{equation}
		Furthermore, by the smoothness of $r_i$ and the chain rule, we obtain the following pointwise estimate
		\begin{equation}
			\begin{split}
				& \vert \pa_\tau r_i(\psi \kappa_E, \pa_\tau \psi ,\nu_E) \vert \leq C \big( 1+ \vert \pa_{\tau}^2 \psi \vert+ \vert \pa_\tau (\psi \kappa_E)\vert   \big),  \\
				& \vert \pa_\tau^2 r_i(\psi \kappa_E, \pa_\tau \psi ,\nu_E) \vert \leq C \sum_{\alpha \in \N^6, \vert \alpha \vert\leq 2} \prod_{k=1}^{2}(1+ \vert \pa_\tau^{\alpha_k} (\psi \kappa_E) \vert )(1+ \vert \pa_\tau^{1+\alpha_{2+k}} \psi \vert ),  \\
				&   \vert \pa_\tau^3 r_i(\psi \kappa_E, \pa_\tau \psi ,\nu_E) \vert \leq C \sum_{\alpha \in \N^6, \vert \alpha \vert\leq 3} \prod_{k=1}^{3}(1+ \vert \pa_\tau^{\alpha_k} (\psi \kappa_E) \vert )(1+ \vert \pa_\tau^{1+\alpha_{3+k}} \psi \vert ). 
			\end{split}
		\end{equation}
		Therefore, using Lemma \ref{LemmLeibinz} with $f_1=f_2=f_3= \psi \kappa_E$ and $f_4=f_5=f_6= \pa_\tau \psi$, we get
		\begin{equation}\label{04042025f4}
			\begin{split}
				\| r_i(\psi \kappa_E &,\pa_\tau \psi, \nu_E) \|_{H^3(\pa E)} \\ &\leq C(1+ \| \psi \kappa_E \|_{L^\infty(\pa E)}) (1+ \|\psi  \|_{H^4(\pa E)})+ (1+ \| \psi \|_{C^1(\pa E)})(1+ \| \psi \kappa_E \|_{H^3(\pa E)})
				\\
				&\leq C (1+ \| \psi \|_{H^4(\pa E)}+ \| \psi \kappa_E \|_{H^3(\pa E)}) \\
				& \leq C(1+ \| \psi \|_{H^4(\pa E)}+ \| \psi \|_{L^\infty(\pa E)} \|\kappa_E \|_{H^3(\pa E)})\\
				& \leq C(1+ \| \psi \|_{H^{4}(\pa E)}),
			\end{split}
		\end{equation}
		where we used the assumption $\| \pa_{\tau}^3 \kappa_E^{\varphi} \|_{L^2(\pa E)} \leq \frac{K}{h^{\frac{1}{4}}}$ and formula \eqref{tesipropPreg}.
		To estimate the third term, we again apply the chain rule, the regularity of $r_0$, Lemma \ref{LemmLeibinz} and Proposition \ref{PROPINTER}:
		\begin{equation}\label{04042025f5}
			\begin{split}
				\| r_0 (\psi, \pa_\tau \psi, \kappa_E, \nu_E) \|_{H^2(\pa E)} &\leq C(1+ \| \psi \|_{H^3(\pa E)}+ \| \kappa_E \|_{H^2(\pa E)}) \leq C (1+ \| \psi \|_{H^3(\pa E)}) \\
				& \leq \varepsilon \| \psi \|_{H^4(\pa E)}+ C(K,\varepsilon) \| \psi \|_{H^2(\pa E)}+ C(K).
			\end{split}
		\end{equation}
		Combining estimates  \eqref{04042025f1}, \eqref{04042025f2}, \eqref{04042025f3}, \eqref{04042025f4}, \eqref{04042025f5} and using Lemma \ref{michele} we obtain
		\begin{equation}\label{napoli7}
			\| \pa_\tau^2 R \|_{L^2(\pa E)} \leq C(K,\varepsilon)\| \psi \|_{H^2(\pa E)}^2+ \varepsilon \| \pa_\tau^4 \psi \|_{L^2(\pa E)}^2.
		\end{equation}
		
		Using \eqref{napoli1},\eqref{napoli2}, \eqref{napoli3}, \eqref{napoli4},\eqref{napoli5},\eqref{napoli6}, \eqref{napoli7} and recalling \eqref{lafunzghaprop} and for $\varepsilon, h$ sufficiently small we deduce
		\begin{equation}
			\frac{1}{h} \| \pa^2_\tau \psi \|^2_{L^2(\pa E)}+ \frac{C_g}{2} \| \pa^4_\tau \psi \|^2_{L^2(\pa E)} \leq C(K,K_{el}).
		\end{equation}
		Therefore, we obtain the bound
		\begin{equation}\label{lego6080}
			\| \psi \|_{H^4(\pa E)} \leq C(K,K_{el}).
		\end{equation} 
		This implies that the right-hand side of equation \eqref{LEEULEROCOMB}  is bounded in $L^2(\pa E)$, i.e., \begin{equation}\label{laFPDELULTF1}
			\| \pa_{\tau}^2 (-g(\nu_E)\pa_{\tau}^2 \psi+ \kappa_E^\varphi)- \pa_{\tau}^2 (Q(E(u_F^{K_{el}}))(\cdot+ \psi(\cdot)\nu_E(\cdot)))+ \pa_{\tau}^2 R \|_{L^2(\pa E)} \leq C(K,K_{el}) . 
		\end{equation}
		
		To prove $ \| \psi \|_{L^2(\pa E)} \leq C(K,K_{el}) h$,  we multiply the Euler–Lagrange equation \eqref{LEEULEROCOMB} by $ \psi + \frac{\psi^2}{2}\kappa_E$, integrate over $\pa E$, and apply the Cauchy–Schwarz inequality along with \eqref{laFPDELULTF1}:
		\begin{equation}
			\begin{split}
				\frac{1}{h}& \| \psi + \frac{\psi^2}{2}\kappa_E \|_{L^2(\pa E)}^2 \\
				&\leq \| (\pa_{\tau}^2 (-g(\nu_E)\pa_{\tau}^2 \psi+ \kappa_E^\varphi)- \pa_{\tau}^2 (Q(E(u_F^{K_{el}}))(\cdot+ \psi(\cdot)\nu_E(\cdot)))+ \pa_{\tau}^2 R) (\psi + \frac{\psi^2}{2}\kappa_E) \|_{L^2(\pa E)}\\
				& \leq C(K,K_{el}) \| \psi + \frac{\psi^2}{2}\kappa_E \|_{L^2(\pa E)}.
			\end{split}
		\end{equation}
		Therefore we obtain $ \| \psi + \frac{\psi^2}{2}\kappa_E \|_{L^2(\pa E)} \leq C(K,K_{el})h$. Recalling that $\frac{\psi^2}{2} \leq \big(\psi+ \frac{\psi^2}{2}\kappa_E\big)^2 $ (see formula \eqref{fomr07112024}), we obtain \begin{equation}\label{lego60801}
			\| \psi \|_{L^2(\pa E)} \leq C(K,K_{el})h.
		\end{equation}
		Combining \eqref{lego6080} and \eqref{lego60801}, we conclude the proof of \eqref{Lego375}.
	\end{proof}
	We need the following technical lemma; see \cite[Lemma 5.3]{JN} for the proof. We state the lemma in  $\R^2$, as this is the setting relevant to our context.
	\begin{lemma}
		Let $E\subset \R^2$ of class $C^5$ be such that $E \in \mathfrak{H}^4_{K,\sigma_0}(E_0)$.  Then, for all $u \in C^3(\pa E)$, the following estimates hold: \begin{equation}\label{maperchènonti}
			\begin{split}
				&\vert \nabla (u \circ \pi_{\pa E})(x) \vert \leq C(1+\vert \kappa_E \circ \pi_{\pa E}(x)\vert ) \vert \pa_{\pa E} u \circ \pi_{\pa E}(x)\vert,\\
				&\vert \nabla^2 (u \circ \pi_{\pa E}) (x)\vert \leq C \sum_{i=0,1} (1+\vert \pa_{\pa E}^i \kappa_E \circ \pi_{\pa E}(x)\vert ) \vert \pa_{\pa E}^{2-i} u \circ \pi_{\pa E}(x)\vert,\\
				& \vert \nabla^3 (u \circ \pi_{\pa E})(x)\vert \leq C \sum_{i=0,1,2} (1+\vert \pa_{\pa E}^i \kappa_E \circ \pi_{\pa E}(x)\vert ) \vert \pa_{\pa E}^{3-i} u \circ \pi_{\pa E}(x)\vert,
			\end{split}
		\end{equation}
		for all $ x \in \mathcal{I}_{\sigma_E}(\pa E)$.  
	\end{lemma}
	In the next proposition, we prove that if $F \subset \R^2$ is a minimizer of \eqref{probdimin1}, then the following estimates hold:
	$$\| \kappa_F^{\varphi} \|_{H^2(\pa F)} \leq  C_2, \quad \| \kappa_F^{\varphi} \|_{H^3(\pa F)}  \leq \frac{C_2}{h^{\frac{1}{4}}} ,  $$
	where $C_2:= C_2(K,K_{el})$.
	
	\begin{proposition}\label{Mthmproof2}
		Let $E\subset \R^2$ of class $C^5$ be such that $E \in \mathfrak{H}^4_{K,\sigma_0}(E_0)$ and $ \| \pa^3_{\pa E} \kappa_E^\varphi \|_{L^2(\pa E)} \leq \frac{ K}{h^{\frac{1}{4}}}$. Let $F \subset \R^2$ be a minimizer of \eqref{probdimin1} for $\eta=\eta_0$, where $\eta_0$ is given by the Proposition \ref{primapropreg}.  Let $\psi $ be the heightfunction in \eqref{tesipropPreg} satisfying  \eqref{Lego375}, that is,
		$$ \| \psi \|_{L^2(\pa E)} \leq C_1 h, \qquad \| \pa_{\pa E}^4 \psi  \|_{L^2(\pa E)} \leq C_1.$$
		Then there exists a constant $C_2$, depending only on $K, K_{el}$, such that
		\begin{equation}\label{finsuperlong}
			\| \pa_{\pa F}^2 \kappa_F^{\varphi} \|_{L^2(\pa F)} \leq C_2, \qquad  \| \pa_{\pa F}^3 \kappa_F^\varphi \|_{L^2(\pa F)} \leq \frac{C_2}{h^{\frac{1}{4}}}.
		\end{equation}
	\end{proposition}
	\begin{proof} 
		In what follows, we denote by $C$  a generic constant that depends on $K,K_{el},C_1$.
		We derive the estimates from the Euler–Lagrange equations \eqref{LEEULERO}. We define $$ \tilde{f}:= \frac{d_{H^{-1}}(F,E)}{h}f $$
		where $f$  is the function that realizes the supremum in \eqref{d_{H^{-1}}}.
		The Euler-Lagrange equation then becomes 
		\begin{equation}\label{18042024form1}
			\left\{
			\begin{aligned}
				& \kappa_F^{\varphi}- Q(E(u_F^{K_{el}}))+ \tilde{f} \circ \pi_{\pa E}=L & \text{ on } \pa F  , \\
				& -\Delta_{\pa E} \tilde{f}=\frac{ \xi_{F,E} }{h} & \text{ on }  \pa E, 
			\end{aligned}
			\right.
		\end{equation}
		where $\xi_{F,E}$ is defined in \eqref{lafunzgEFgraph} and $L$ is the Lagrange multiplier. From \eqref{Lego375} and interpolation of $H^1(\pa E)$ between $L^2(\pa E)$ and $H^4(\pa E)$, see Proposition \ref{PROPINTER}, we obtain
		\begin{equation}
			\| \psi \|_{H^1}(\pa E) \leq Ch^{\frac{3}{4}}.
		\end{equation}
		Using this estimate together with \eqref{tesilemmazzzo} and \eqref{Lego375}, we deduce
		\begin{equation}\label{18042025form0}
			\| \xi_{F,E}\|_{L^2(\pa E)} \leq C h \qquad \| \xi_{F,E} \|_{H^1(\pa E)} \leq C h^{\frac{3}{4}}.
		\end{equation}
		Therefore, by \eqref{18042025form0} and the second equation in \eqref{18042024form1}, we conclude
		\begin{equation}\label{forma18042025}
			\begin{split}
				&\| \tilde{f} \|_{H^2(\pa E)} \leq C \frac{ \| \xi_{F,E} \|_{L^2(\pa E)}}{h} \leq C,  \\
				&\| \tilde{f} \|_{H^3(\pa E)} 
				\leq C (1+ \| \kappa_E \|_{H^2(\pa E)}+ \| \pa_{\pa E} \tilde{f}\|_{H^1(\pa E)}) \leq C(1+ \frac{\| \xi_{F,E} \|_{H^1(\pa E)}}{h})
				\leq  C h^{-\frac{1}{4}}.
			\end{split}
		\end{equation}
		We now need to estimate the derivatives of $\tilde{f} \circ \pi_{\pa E}$ on $\pa F$. 
		From formula \eqref{maperchènonti}, we find that for all $x \in \mathcal{I}_{\eta_0}(\pa E)$,
		\begin{equation}\label{levariederiv}
			\begin{split}
				&\vert \nabla (\tilde{f} \circ \pi_{\pa E})(x) \vert \leq C(1+\vert \kappa_E \circ \pi_{\pa E}(x)\vert ) \vert \pa_{\pa E} \tilde{f} \circ \pi_{\pa E}(x)\vert,\\
				&\vert \nabla^2 (\tilde{f} \circ \pi_{\pa E}) (x)\vert \leq C \sum_{i=0,1} (1+\vert \pa_{\pa E}^i \kappa_E \circ \pi_{\pa E}(x)\vert ) \vert \pa_{\pa E}^{2-i} \tilde{f} \circ \pi_{\pa E}(x)\vert,\\
				& \vert \nabla^3 (\tilde{f} \circ \pi_{\pa E}) (x)\vert \leq C \sum_{i=0,1,2} (1+\vert \pa_{\pa E}^i \kappa_E \circ \pi_{\pa E}(x)\vert ) \vert \pa_{\pa E}^{3-i} \tilde{f} \circ \pi_{\pa E}(x)\vert.
			\end{split}
		\end{equation}
		To obtain \eqref{finsuperlong}, we need to estimate $ \| \pa_{\pa F}^2 \kappa_F^{\varphi} \|_{L^2(\pa F)}$ and $\| \pa_{\pa F}^3 \kappa_F^{\varphi} \|_{L^2(\pa F)}$.\\
		\textit{Estimate of $ \| \pa_{\pa F}^2 \kappa_F^{\varphi} \|_{L^2(\pa F)}$.} 
		
		Recalling the first equation of \eqref{18042024form1}, we need to estimate 
		$$ \|  \pa_{\pa F}^2 (\tilde{f}\circ \pi_{\pa E}) \|_{L^2(\pa F)} \text{ and } \| \pa_{\pa F}^2Q(E(u_F^{K_{el}}))\|_{L^2(\pa F)}.$$
		We begin with the estimate of $ \|  \pa_{\pa F}^2 (\tilde{f}\circ \pi_{\pa E}) \|_{L^2(\pa F)}$.
		Using formula \eqref{levariederiv}, we obtain
		\begin{equation}
			\vert \nabla^2 (\tilde{f} \circ \pi_{\pa E})(x)\vert \leq C \vert \pa_{\pa E}^2 \tilde{f} \circ \pi_{\pa E}(x) \vert + C(1+ \vert \pa_{\pa E} \kappa_E \circ \pi_{\pa E}(x) \vert ) \vert \pa_{\pa E} \tilde{ f} \circ \pi_{\pa E}(x) \vert
		\end{equation}
		for all $ x \in \pa F$. Therefore, by Sobolev embedding and \eqref{forma18042025}, we get 
		\begin{equation}\label{stimadev2}
			\| \nabla^2 (\tilde{f} \circ \pi_{\pa E}) \|_{L^2(\pa F)} \leq C \big( \| \tilde{f} \|_{H^2(\pa E)}+ (1+ \| \pa_{\pa E} \kappa_E \|_{\infty} \|\pa_{\pa E} \tilde{f} \|_{L^2(\pa E)} )   \big) \leq C.
		\end{equation}
		Next, we recall that the Laplacian of $ \tilde{f} \circ \pi_{\pa E}$ on $\pa F$ can be written as
		\begin{equation}\label{Formulavera?}
			\pa^2_{\pa F} (\tilde{f} \circ \pi_{\pa E})= \Delta_{\R^2} (\tilde{f} \circ \pi_{\pa E})- \nabla^2 (\tilde{f} \circ \pi_{\pa E}) \nu_F \cdot \nu_F - \kappa_F \nabla (\tilde{f} \circ \pi_{\pa E}) \cdot \nu_F.
		\end{equation}
		Thus, using this formula together with \eqref{stimadev2}, we obtain
		\begin{equation}\label{01052025form1}
			\| \pa_{\pa F}^2 (\tilde{f}\circ \pi_{\pa E}) \|_{L^2(\pa F)} \leq C.
		\end{equation}
		It remains to estimate
		$$\| \pa_{\pa F}^2 Q(E(u_F^{K_{el}}))\|_{L^2(\pa F)}  .$$  A direct computation yields
		\begin{equation}\label{grimbatul}
			\begin{split}
				\pa_{\pa F}^2 Q(E(u_F^{K_{el}}))&= \pa_{\pa F} [  \pa_{\pa F} Q(E(u_F^{K_{el}}))]= \pa_{\pa F}[ \nabla Q(E(u_F^{K_{el}})) \cdot \tau_F ]\\
				&= \nabla^2 Q(E(u_F^{K_{el}})) \tau_F \cdot \tau_F + \kappa_F \nabla Q(E(u_F^{K_{el}})) \cdot \nu_F 
			\end{split}
		\end{equation}
		where we have used $ \pa_{\pa F} \tau_F= \kappa_F \nu_F$.
		Applying this formula, we deduce
		\begin{equation}\label{01052025form2}
			\| \pa_{\pa F}^2Q(E(u_F^{K_{el}}))\|_{L^2(\pa F)} \leq C.
		\end{equation}
		Combining estimates \eqref{01052025form1} and \eqref{01052025form2}, we conclude that
		\begin{equation}\label{trattoria1}
			\| \pa_{\pa F}^2 \kappa_F^{\varphi} \|_{L^2(\pa F)} \leq C.
		\end{equation}
		\textit{Estimate of $\| \pa_{\pa F}^3 \kappa_F^{\varphi} \|_{L^2(\pa F)}$.}
		
		Recalling the first equation in \eqref{18042024form1}, we need to estimate 
		$$  \|  \pa_{\pa F}^3 (\tilde{f}\circ \pi_{\pa E}) \|_{L^2(\pa F)} \text{ and } \| \pa_{\pa F}^3 Q(E(u_F^{K_{el}}))\|_{L^2(\pa F)} .  $$
		We begin by estimating $ \|  \pa_{\pa F}^3 (\tilde{f}\circ \pi_{\pa E}) \|_{L^2(\pa F)}$.
		
		Using formula \eqref{levariederiv}, we obtain
		\begin{equation}\label{entropia1}
			\begin{split}
				\vert \nabla^3 ( \tilde{f} \circ \pi_{\pa E}) (x) \vert \leq C \vert \pa^3_{\pa E}& \tilde{f} \circ \pi_{\pa E}(x) \vert + C(1+ \vert \pa_{\pa E} \kappa_E \circ \pi_{\pa E}(x)\vert )\vert \pa^2_{\pa E} \tilde{ f} \circ \pi_{\pa E}(x)\vert\\
				&C \big(  1+ \vert \pa_{\pa E}^2 \kappa_E \circ \pi_{\pa E}(x) \vert + \vert \pa_{\pa E} \kappa_E \circ \pi_{\pa E} (x)\vert^2   \big) \vert \pa_{\pa E} \tilde{f} \circ \pi_{\pa E} (x)\vert 
			\end{split}
		\end{equation}
		for all $x \in \pa F$. 
		Using formula \eqref{Formulavera?}, we compute $ \pa_{\pa F}^3 (\tilde{f} \circ \pi_{\pa E})$
		\begin{equation}
			\begin{split}
				\pa_{\pa F}^3 (\tilde{f} \circ \pi_{\pa E})&= \nabla \Delta_{\pa F} (\tilde{f} \circ \pi_{\pa E}) \cdot \tau_F\\
				& = \nabla \big[ \Delta_{\R^2} (\tilde{f} \circ \pi_{\pa E})- \nabla^2(\tilde{f} \circ \pi_{\pa E}) \nu_F \cdot \nu_F- \kappa_F \nabla(\tilde{f} \circ \pi_{\pa E}) \cdot \nu_F   \big] \cdot \tau_F\\
				&= T(\nabla^3 (\tilde{f} \circ \pi_{\pa E}), \nabla^2 (\tilde{f} \circ \pi_{\pa E}),\pa_{\pa F} \kappa_F \nabla (\tilde{f} \circ \pi_{\pa E})   ),
			\end{split}
		\end{equation}
		where $T \in C^\infty$ such that $T(0,0,0)=0$. Hence, using the regularity of $T$, we deduce the pointwise estimate: for all $ x \in \pa F$, \begin{equation}\label{entropia2}
			\vert \pa_{\pa F}^3 (\tilde{f} \circ \pi_{\pa E}) \vert (x) \leq C \big( \vert \nabla^3 (  \tilde{f} \circ \pi_{\pa E}) (x) \vert + \vert \nabla^2 ( \tilde{f} \circ \pi_{\pa E}) (x) \vert + \vert \nabla ( \tilde{f} \circ \pi_{\pa E}) (x) \vert  \big) .
		\end{equation}
		Therefore, combining \eqref{entropia2}, \eqref{entropia1}, \eqref{stimadev2}, and \eqref{forma18042025}, we obtain
		\begin{equation}\label{teggare0}
			\| \pa_{\pa F}^3 (\tilde{f} \circ \pi_{\pa E}) \|_{L^2(\pa F)} \leq C\big(   \| \tilde{f} \|_{H^3(\pa E)}+ \| \kappa_E \|_{H^2(\pa E)} \big) \leq C h^{-\frac{1}{4}} .
		\end{equation}
		It remains to estimate
		$\| \pa_{\pa F}^3 Q(E(u_F^{K_{el}}))\|_{L^2(\pa F)}  $. 
		
		Differentiating formula \eqref{grimbatul}, we get
		\begin{equation}\label{primapalestra0105}
			\begin{split}
				\pa_{\pa F}^3 Q(E(u_F^{K_{el}}))&= \pa_{\pa F} \big[ \nabla^2 Q(E(u_F^{K_{el}})) \tau_F \cdot \tau_F    \big]+ \pa_{\pa F} [ \kappa_F \nabla Q(E(u_F^{K_{el}})) \cdot \nu_F ]\\
				&= 2 \kappa_F \nabla^2 Q(E(u_F^{K_{el}})) \tau_F \cdot \nu_F + M( \nabla^3 Q(E(u_F^{K_{el}})),\tau_F) \tau_F \cdot \tau_F\\
				& \qquad - \kappa_F^2 \tau_F \cdot \nabla Q(E(u_F^{K_{el}}))+ \pa_{\pa F} \kappa_F \nabla Q(E(u_F^{K_{el}})) \cdot \nu_F\\
				& \qquad + \kappa_F \nabla^2 Q(E(u_F^{K_{el}})) \tau_F \cdot \nu_F
			\end{split}
		\end{equation}
		where $ M( \nabla^3 Q(E(u_F^{K_{el}})),\tau_F) $ is a matrix $2 \times 2$ matrix whose coefficients depend on $\nabla^3 Q(E(u_F^{K_{el}}))$ and $ \tau_F $, 
		and satisfy $$ \vert M( \nabla^3 Q(E(u_F^{K_{el}})),\tau_F) \vert \leq C \vert  \nabla^3 Q(E(u_F^{K_{el}})) \vert.$$
		Therefore, using \eqref{forma18042025}, \eqref{primapalestra0105},  and recalling that $$ \| \nabla^3 Q(E(u_F^{K_{el}}))\|_{L^\infty(\Omega)} \leq \frac{K_{el}}{h^{\frac{1}{4}}},$$
		(see formula \eqref{minelvinc}), we obtain
		\begin{equation}\label{teggare1}
			\| \pa_{\pa F}^3 Q(E(u_F^{K_{el}})) \|_{L^2(\pa F)}  \leq C \big(  \| \nabla^3 Q(E(u_F^{K_{el}})) \|_{L^2(\pa F)}+ C \| \pa_{\pa F} \kappa_F \|_{L^2(\pa F)}  \big) \leq  \frac{C}{h^{\frac{1}{4}}}.
		\end{equation}
		Therefore, combining \eqref{teggare0} and  \eqref{teggare1}, we conclude that
		\begin{equation}\label{trattoria2}
			\| \pa_{\pa F}^3 \kappa_F^{\varphi} \|_{L^2(\pa F)} \leq \frac{C}{h^{\frac{1}{4}}}.
		\end{equation}
		
		Combining \eqref{trattoria1} and \eqref{trattoria2}, we finally obtain\eqref{finsuperlong}.
	\end{proof}
	We are now in position to prove Theorem \ref{MainThm375}.
	\begin{proof}[Proof of Theorem \ref{MainThm375}]
		The existence of constants $\eta_0$ and $h_0$ is guaranteed by Proposition \ref{primapropreg}. Using this proposition, it is also established that $ \pa F \Subset \mathcal{I}_{\eta_0}(\pa E)$ and that 
		$$ \pa F= \{ x + \psi(x)\nu_E(x): \, x \in \pa E\}.$$  
		Proposition \ref{Mthmproof1} establishes the existence of a constant $C_1$ and the validity of formula \eqref{formMain1375}. Similarly, Proposition \ref{Mthmproof2} proves the existence of a constant $C_2$ and formula \eqref{FormMain2375}.\\
		\textit{Claim} There exists $\hat{\sigma}$  such that $ F \in \mathfrak{H}^4_{K_1,  \hat \sigma}(E_0)$, for some $K_1=K_1(K,K_{el})$.
		
		By Lemma \ref{lalalambdamin}, the set $F$ is a $\lambda$-minimizer of the $\varphi$-perimeter, with $\lambda=\lambda(K,K_{el})$. Applying Lemma \ref{propdadim} with $E=E_0$ and $\Lambda= \lambda$ we obtain the existence of the constant $\delta_0=\delta_0(\lambda)$. Now we take $\sigma_0, \eta_0$ such that $\sigma_0+\eta_0 \leq \frac{\delta_0}{2}$. Then we have:
		\begin{equation}\label{belgodi1}
			\pa F \Subset \mathcal{I}_{\eta_0}(\pa E),\, \pa E \Subset \mathcal{I}_{\sigma_0}(\pa E_0) \text{ and } \sigma_0+\eta_0 \leq \frac{\delta_0}{2} \implies \pa F \Subset \mathcal{I}_{\delta_0}(\pa E_0).
		\end{equation}
		Applying Lemma \ref{propdadim} once again, we obtain the existence of a function $u: \pa E_0 \rightarrow \R$ such that 
		$$ \pa F= \{ x+ u(x)\nu_{E_0}(x): x \in \pa E_0\},$$ 
		with $ u \in C^{1,\gamma}(\pa E_0)$. Therefore, using \eqref{FormMain2375} we conclude that
		\begin{equation}\label{bengodi2}
			u \in H^4(\pa E_{0}),\quad \| u \|_{H^4(\pa E_0)} \leq K_1 \quad \text{for some }K_1= K_1(K,K_{el}). 
		\end{equation}
		Combining \eqref{belgodi1} and \eqref{bengodi2}, we obtain 
		$ F \in \mathfrak{H}^4_{K_1,\hat \sigma}$ as claimed.
	\end{proof}
	
	\section{Iteration}\label{iterazionesezione}
	In this section, we prove a crucial iteration formula.  To this end, we fix a set $E \subset\R^2$ of class $C^5$ such that $E \in \mathfrak{H}^4_{K,\sigma_0}(E_0)$ and 
	$\| \pa_{\pa E}^3 \kappa_E^{\varphi} \|_{L^2(\pa E)} \leq \frac{K}{h^{\frac{1}{4}}}$. We recall that $E_0 \Subset \Omega$ be open and connected set of class $C^5$.
	We consider two sets $F,G\subset \R^2$ constructed as follows. By the Theorem \ref{MainThm375}, there exist constants $ h_0,\eta_0,C_1,C_2,K_1, \hat{\sigma}$, depending only on $K$ and $K_{el}$, such that if  $0 < h \leq h_0$ and
	\begin{equation}\label{ierminF}
		F \in \mathrm{argmin}\{  \mathcal{F}_h(A,E)\colon A \Delta E \subset \mathrm{cl} (\mathcal{I}_{\eta_0}(\pa E))\} 
	\end{equation}
	then $F$ is of class $C^5$ and $F \in \mathfrak{H}^4_{K_1,\hat{ \sigma}}(E_0)$. Again, by Proposition \ref{PROPINTER} and Theorem \ref{MainThm375}, we have $ \pa F \Subset \mathcal{I}_{\eta_0}(\pa E)$, and $\pa F= \{   x+ \psi_{F,E}(x) \nu_E(x): x \in \pa E\}$, with the following estimates for $\psi_{F,E}$: 
	\begin{equation}\label{iterstimF}
		\begin{split}
			&\| \psi_{F,E} \|_{L^2(\pa E)} \leq C_1 h, \,\, \| \psi_{F,E} \|_{H^4(\pa E)} \leq C_1, \,\, \|  \kappa_F^{\varphi} \|_{H^2(\pa F)} \leq C_2, \,\,  \| \pa_{\pa F}^3 \kappa_F^\varphi \|_{L^2(\pa F)} \leq \frac{C_2}{h^{\frac{1}{4}}}\\
			&\| \pa_{\pa E} \psi_{F,E} \|_{L^2(\pa E)} \leq C_1 h^{\frac{3}{4}},\, \, \| \pa_{\pa E}^2 \psi_{F,E} \|_{L^2(\pa E)} \leq C_1 h^{\frac{1}{2}}, \,\, \| \pa_{\pa E}^3 \psi_{F,E} \|_{L^2(\pa E)} \leq C_1 h^{\frac{1}{4}},
		\end{split}
	\end{equation}
	where the second line follows from the first and an application of Proposition \ref{PROPINTER}.
	Applying Theorem \ref{MainThm375} again this time with $F,K_1,\hat{\sigma}$ in place of $E,K,\sigma_0$, we get new constants $\eta_1,h_1,C_3,C_4, K_2, \tilde{\sigma}$ 
	%(instead of $\eta_0,h_0,C_1,C_2,K_1,\hat{\sigma}$)
	, depending only on $ K_1$ and $K_{el}$ (and hence ultimately on $K$ and $K_{el}$). If $\eta \leq  \eta_1$ and $ 0 < h \leq h_2:= \min \{ h_0,h_1\}$ the set $G$ is given by
	\begin{equation}\label{iterminG}
		G \in \mathrm{argmin}\{  \mathcal{F}_h(A,F)\colon A \Delta E \subset \mathrm{cl} (\mathcal{I}_{\eta}(\pa F))\} 
	\end{equation}
	and $G$ is of class $C^5$, $G \in \mathfrak{H}^4_{K_2,\tilde{ \sigma}}(E_0)$. Again, by Proposition \ref{PROPINTER} and Theorem \ref{MainThm375}, we have $ \pa G \Subset \mathcal{I}_{\eta_1}(\pa F)$, and $\pa G= \{   x+ \psi_{G,F}(x) \nu_F(x): x \in \pa F\}$, with the following estimates for $ \psi_{G,F}$:
	\begin{equation}\label{iterstimG}
		\begin{split}
			&\| \psi_{G,F} \|_{L^2(\pa F)} \leq C_3 h, \,\, \| \psi_{G,F} \|_{H^4(\pa F)} \leq C_3, \,\, \|  \kappa_G^{\varphi} \|_{H^2(\pa G)} \leq C_4, \,\,  \| \pa_{\pa G}^3 \kappa_G^\varphi \|_{L^2(\pa G)} \leq \frac{C_4}{h^{\frac{1}{4}}},\\
			& \| \pa_{\pa F} \psi_{G,F} \|_{L^2(\pa F)} \leq C_3 h^{\frac{3}{4}},\, \, \| \pa_{\pa F}^2 \psi_{G,F} \|_{L^2(\pa F)} \leq C_3 h^{\frac{1}{2}}, \, \, \| \pa_{\pa F}^3 \psi_{G,F} \|_{L^2(\pa F)} \leq C_3 h^{\frac{1}{4}},
		\end{split}
	\end{equation}
	where the second line again follows from the first using Proposition \ref{PROPINTER}.
	Throughout this section, we will use the notation just introduced. We now state a lemma that will be essential for proving the main result of this section.
	\begin{lemma}\label{spaghettiska}
		Let $\eta < \eta_1$, where $\eta_1$ is as defined above. Let $E,F$ and $G$ as above, and we set
		$\xi_{G,F}= \psi_{G,F}+ \kappa_G \frac{\psi_{G,F}^2}{2} $. There exists a constant $h_3>0$, depending only on $K$ and $K_{el}$, such that the following inequality holds:
		\begin{equation}\label{07052025tes1lemmaiter}
			\int_{\pa F} (1-Ch)\xi^2_{G,F} + \frac{3h}{4} g(\nu_F) \vert \pa^2_{\pa F} \psi_{G,F} \vert^2 \, d \mathcal{H}^1   \leq h\int_{\pa F} \kappa_F^{\varphi} 
			\pa_{\pa F}^2 \xi_{G,F}  \, d \mathcal{H}^1.
		\end{equation}
		for $0 < h \leq h_3$, where $C= C(K,K_{el})$.
	\end{lemma}
	\begin{proof} In what follows, we denote by $C$ a generic constant depending on $K$ and $K_{el}$. We recall that, as stated in formula \eqref{fomr07112024} 
		\begin{equation}\label{050525pranzo1}
			\frac{1}{\sqrt{2}}\psi_{G,F}^2 \leq \xi_{G,F}^2 \leq \sqrt{2} \psi_{G,F}^2. 
		\end{equation}
		From the discussion at the beginning of the section, the Euler–Lagrange equation \eqref{LEEULEROCOMB} for the set $G$ can be written as
		\begin{equation}
			\begin{split}
				\frac{1}{h}\xi_{G,F}(x)=& \pa_{\pa F}^2 (-g(\nu_F(x))\pa_{\pa F}^2 \psi_{G,F}(x)+ \kappa_F^\varphi(x))\\ 
				&\qquad - \pa_{\pa F}^2 \big(Q(E(u_F^{K_{el}}))(x+ \psi_{G,F}(x)\nu_F(x))\big)+ \pa_{\pa F}^2 R(x) \text{ for $x \in \pa F$ }
			\end{split}
		\end{equation}
		where $R$ is define in \eqref{ILRESTO}. Multiplying the equation above by $\xi_{G,F}$ and integrating by parts yields
		\begin{equation}
			\begin{split}
				\int_{\pa F}  \frac{\xi_{G,F}^2}{h} \, d \mathcal{H}^1&= \int_{\pa F} - g(\nu_F) \pa_{\pa F}^2 \psi_{G,F} \pa_{\pa F}^2 \xi_{G,F} \, d \mathcal{H}^1+ \int_{\pa F} \kappa_F^{\varphi} 
				\pa_{\pa F}^2 \xi_{G,F}  \, d \mathcal{H}^1\\
				&\,\,- \int_{\pa F} \xi_{G,F} \pa_{\pa F}^2 \big(Q(E(u_F^{K_{el}}))(\cdot+ \psi_{G,F}(\cdot)\nu_F(\cdot))\big) \, d \mathcal{H}^1+ \int_{\pa F} R \pa_{\pa F}^2 \xi_{G,F} \, d \mathcal{H}^1.
			\end{split}
		\end{equation}
		By the very definition of $\xi_{G,F}$, we obtain
		\begin{equation}\label{finalmento070525}
			\begin{split}
				&\int_{\pa F} \frac{\xi^2_{G,F}}{h} \, d \mathcal{H}^1+ \int_{\pa F} g(\nu_F) \vert \pa_{\pa F}^2 \psi_{G,F} \vert^2 \, d \mathcal{H}^1 \\
				&=  \int_{\pa F} -g(\nu_F) \pa_{\pa F}^2 \psi_{G,F} \pa_{\pa F}^2 \bigg( \frac{\kappa_F\psi_{G,F}^2}{2} \bigg)\, d \mathcal{H}^1 + \int_{\pa F} \kappa_F^{\varphi} 
				\pa_{\pa F}^2 \xi_{G,F}  \, d \mathcal{H}^1\\
				& \qquad - \int_{\pa F} \xi_{G,F} \pa_{\pa F}^2 \big(Q(E(u_F^{K_{el}}))(\cdot+ \psi_{G,F}(\cdot)\nu_F(\cdot))\big) \, d \mathcal{H}^1+ \int_{\pa F} R \pa_{\pa F}^2 \xi_{G,F} \, d \mathcal{H}^1.
			\end{split}
		\end{equation}
		We now need to estimate the following integrals:
		\begin{align}
			&\int_{\pa F} g(\nu_F) \pa_{\pa F}^2 \psi_{G,F} \pa_{\pa F}^2 \bigg( \frac{\kappa_F\psi_{G,F}^2}{2} \bigg)\, d \mathcal{H}^1; \label{05052025form1} \\
			&\int_{\pa F} \xi_{G,F} \pa_{\pa F}^2 \big(Q(E(u_F^{K_{el}}))(\cdot+ \psi_{G,F}(\cdot)\nu_F(\cdot))\big) \, d \mathcal{H}^1; \label{05052025form2} \\
			&\int_{\pa F} R \pa_{\pa F}^2 \xi_{G,F} \, d \mathcal{H}^1. \label{05052025form3} 
		\end{align}
		Let $\varepsilon>0$ be fixed, to be chosen later.\\
		\textit{Estimate of \eqref{05052025form1}. }
		
		A straightforward computation yields
		\begin{equation}\label{0505primacrossfit1}
			\begin{split}
				\pa_{\pa F}^2 \bigg( \kappa_F \frac{\psi_{G,F}^2}{2}   \bigg)=& \pa_{\pa F}^2 \kappa_F \frac{\psi_{G,F}^2}{2}+ 2\pa_{\pa F} \kappa_F \psi_{G,F} \pa_{\pa F} \psi_{G,F}  \\
				& \quad\qquad\qquad \,+ \kappa_F (\pa_F \psi_{G,F})^2+ \kappa_F \psi_{G,F} \pa_{\pa F}^2 \psi_{G,F}.
			\end{split}
		\end{equation}
		Using formulas \eqref{iterstimG} and \eqref{050525pranzo1}, a together with the Sobolev embedding and the H\"{o}lder inequality, we obtain
		\begin{equation}\label{0505primacrossfit2}
			\begin{split}
				&\| \pa_{\pa F}^2 \kappa_F \frac{\psi_{G,F}^2}{2}+2 \pa_{\pa F} \kappa_F \psi_{G,F} \pa_{\pa F} \psi_{G,F}  \|_{L^2(\pa F)} \\
				&\quad \qquad \leq \big(\frac{\| \psi_{G,F} \|_{L^\infty(\pa F)}}{2} \| \kappa_{F} \|_{H^2(\pa F)} +2 \| \pa_{\pa F} \kappa_F \|_{L^2(\pa F)}\| \pa_{\pa F} \psi_{G,F} \|_{L^2(\pa F)} \big)  \| \psi_{G,F} \|_{L^2(\pa F)} \\
				&\quad \qquad \leq C\big( \frac{C_2 h^{\frac{1}{2}}}{2}+ 2C_2 h^{\frac{3}{4}} \big) \| \xi_{G,F} \|_{L^2(\pa E)}= C h^{\frac{1}{2}} \| \xi_{G,F} \|_{L^2(\pa E)}.
			\end{split}
		\end{equation}
		Using again formula \eqref{iterstimG}, the Sobolev embedding, and the H\"{o}lder inequality, we get
		\begin{equation}\label{0505primacrossfit3}
			\begin{split}
				\int_{\pa F} g(\nu_F) \pa_{\pa F}^2 \psi_{G,F} \kappa_F (\pa_F \psi_{G,F})^2 \, d \mathcal{H}^1 &\leq C \| \pa_{\pa F} \psi_{G,F} \|_{L^\infty(\pa F)}^2 \| \pa_{\pa F}^2 \psi_{G,F} \|_{L^2(\pa F)} \\
				&\leq C h^{\frac{1}{2}} \| \pa_{\pa F}^2 \psi_{G,F} \|_{L^2(\pa F)}^2. 
			\end{split}
		\end{equation}
		Still using \eqref{iterstimG}, the Sobolev embedding, and the H\"{o}lder inequality, we deduce
		\begin{equation}\label{0505primacrossfit4}
			\begin{split}
				\int_{\pa F} g(\nu_F) \big( \pa_{\pa F}^2 \psi_{G,F} \big)^2 \kappa_F \psi_{G,F} \, d \mathcal{H}^1 & \leq C \| \psi_{G,F} \|_{L^\infty(\pa F)} \| \pa_{\pa F}^2 \psi_{G,F} \|_{L^2(\pa F)}^2 \\
				&\leq C h^{\frac{3}{4}} \| \pa_{\pa F}^2 \psi_{G,F} \|_{L^2(\pa F)}^2 .
			\end{split}
		\end{equation}
		Combining \eqref{05052025form1}, \eqref{0505primacrossfit1}, \eqref{0505primacrossfit2}, \eqref{0505primacrossfit3}, \eqref{0505primacrossfit4}, and using the H\"{o}lder inequality, we obtain that
		\begin{equation}\label{finestima0705251}
			\int_{\pa F} g(\nu_F) \pa_{\pa F}^2 \psi_{G,F} \pa_{\pa F}^2 \bigg( \frac{\kappa_F\psi_{G,F}^2}{2} \bigg)\, d \mathcal{H}^1 \leq \| \xi_{G,F} \|_{L^2(\pa F)}^2 + h^{\frac{1}{2}}C \| \pa_{\pa F}^2 \psi_{G,F}\|^2_{L^2(\pa F)}.
		\end{equation}
		\textit{Estimate of \eqref{05052025form2}.}
		
		Recalling formula \eqref{ciserviràsicuro} and applying  the Cauchy–Schwarz and Young inequalities, we get
		\begin{multline}\label{finestima0705252}
			\int_{\pa F} \xi_{G,F} \pa_{\pa F}^2 \big(Q(E(u_F^{K_{el}}))(\cdot+ \psi_{G,F}(\cdot)\nu_F(\cdot))\big) \, d \mathcal{H}^1 \\
			\leq C(K,K_{el},\varepsilon) \| \xi_{G,F} \|_{L^2(\pa F)}^2 + \varepsilon    \| \pa_{\pa F}^2 \psi_{G,F}\|^2_{L^2(\pa F)}.
		\end{multline}
		\textit{Estimate of \eqref{05052025form3}.}\\
		\textit{Claim: }  \begin{equation}\label{06052025formpvia}
			\| \pa_{\pa F} \psi_{G,F} \|_{L^2(\pa F)} \leq  \varepsilon C \| \pa_{\pa F}^2 \psi_{G,F} \|_{L^2(\pa F)}+ C(\varepsilon) \| \psi_{G,F} \|_{L^2(\pa F)}.
		\end{equation}
		
		From formulas \eqref{tesilemmazzzo}, \eqref{iterstimG}, and \eqref{0505primacrossfit1}, and using the Sobolev embedding and Proposition \ref{PROPINTER},
		we deduce
		\begin{equation}\label{stac'èpriovani}
			\begin{split}
				\|& \pa_{\pa F} \psi_{G,F} \|_{L^2(\pa F)} \leq C \| \pa_{\pa F} \xi_{G,F} \|_{L^2(\pa F)} \leq   C \| \pa_{\pa F}^2 \xi_{G,F} \|_{L^2(\pa F)}^{\frac{1}{2}} \|  \xi_{G,F}\|^{\frac{1}{2}}_{L^2(\pa F)}\\
				& = \varepsilon \| \pa_{\pa F}^2 \psi_{G,F} + \pa_{\pa F}^2 \big( \kappa_F \frac{\psi_{G,F}^2}{2}   \big)  \|_{L^2(\pa F)}+ C(\varepsilon) \| \xi_{G,F}  \|_{L^2(\pa F)}\\
				& \leq \varepsilon \| \pa_{\pa F}^2 \psi_{G,F}   \|_{L^2(\pa F)} + C \| \pa_{\pa F}^2 \kappa_F \|_{L^2(\pa F)} \| \psi_{G,F}\|_{\infty} \| \psi_{G,F} \|_{L^2(\pa F)}\\
				& \qquad+C \| \pa_{\pa F} \psi_{G,F} \|_{L^2(\pa F)}^2 + C(\varepsilon)\| \xi_{G,F}  \|_{L^2(\pa F)} \\
				& \leq \varepsilon \| \pa_{\pa F}^2\psi_{G,F} \|_{L^2(\pa F)}+ C h^{\frac{3}{4}} \| \pa_{\pa F} \psi_{G,F} \|_{L^2(\pa F)}+ C(\varepsilon)\| \xi_{G,F}  \|_{L^2(\pa F)}.
			\end{split}
		\end{equation}
		Thus, for sufficiently small $h$, estimate \eqref{06052025formpvia} follows.\\
		\textit{Claim:} $$ \| R \|_{L^2(\pa F)} \leq \varepsilon C\| \pa_{\pa F}^2 \psi_{G,F}\|_{L^2(\pa F)}+ C(\varepsilon) \| \psi_{G,F} \|_{L^2(\pa F)}.$$
		
		To this end, we first require a pointwise estimate for $R$ on $\pa F$. By the very definition of $R$, and using formula \eqref{ILRESTO} along with the smallness of $ \| \psi_{G,F} \|_{C^1(\pa F)} \leq C h^{\frac{1}{2}}$ (this estimate follows from formula \eqref{iterstimG} and the Sobolev embedding), we obtain the pointwise estimate 
		\begin{equation}\label{bart1}
			\vert R \vert \leq C \big(   \vert \psi_{G,F} \vert + \vert \pa_{\pa F} \psi_{G,F} \vert \big) \big( 1+ \vert \pa_{\pa F}^2 \psi_{G,F} \vert + \vert \pa_{\pa F} (\psi_{G,F} \kappa_F ) \vert \big)\quad \text{on }\pa F.
		\end{equation}
		From formula \eqref{iterstimG} and the Sobolev embedding, we derive the following estimates:
		\begin{equation}\label{bart2}
			\begin{split}
				&\| \pa_{\pa F}(\psi_{G,F} \kappa_F) \|_{L^\infty(\pa F)} \leq \| \psi_{G,F} \|_\infty \| \pa_{\pa F} \kappa_F \|_\infty+ \| \pa_{\pa F} \psi_{G,F} \|_\infty \| \kappa_F \|_\infty \leq C h^{\frac{1}{2}},\\
				& \| \pa_{\pa F}^2 \psi_{G,F} \|_{L^\infty(\pa F)} \leq C h^{\frac{1}{4}}.
			\end{split}
		\end{equation}
		Combining  \eqref{bart1}, \eqref{bart2}, and using \eqref{fomr07112024},  \eqref{06052025formpvia}, we obtain
		\begin{equation}\label{Rinfinutostim}
			\begin{split}
				\| R \|_{L^2(\pa F)} &\leq C \big( \| \pa_{\pa F} \psi_{G,F} \|_{L^2(\pa F)} + \|  \psi_{G,F}\|_{L^2(\pa F)} \big) \\
				& \leq \varepsilon C \| \pa_{\pa F}^2 \psi_{G,F}\|_{L^2(\pa F)}+ C(\varepsilon) \| \xi_{G,F} \|_{L^2(\pa F)}.
			\end{split}
		\end{equation}
		We are now in a position to estimate \eqref{05052025form3}.
		Using formula \eqref{Rinfinutostim}, the definition of $\xi_{G,F}$, the Cauchy–Schwarz inequality, and the Young's inequality, we obtain 
		\begin{equation}\label{blackfalcon1}
			\begin{split}
				\int_{\pa F} &R \pa_{\pa F}^2 \xi_{G,F}\, d \mathcal{H}^1 = \int_{\pa F} R \pa_{\pa F}^2 \psi_{G,F} \, d \mathcal{H}^1 + \int_{\pa F} R \pa_{\pa F}^2 \bigg( \kappa_F \frac{\psi_{G,F}^2}{2}   \bigg) \, d \mathcal{H}^1\\
				& \leq \| R \|_{L^2(\pa F)} \| \pa_{\pa F}^2 \psi_{G,F} \|_{L^2(\pa F)}+ \int_{\pa F} R \pa_{\pa F}^2 \bigg( \kappa_F \frac{\psi_{G,F}^2}{2}   \bigg) \, d \mathcal{H}^1                        \\
				& \leq \varepsilon C \| \pa_{\pa F}^2 \psi_{G,F} \|_{L^2(\pa F)}^2+ C(\varepsilon) \| \psi_{G,F} \|_{L^2(\pa F)}^2+\int_{\pa F} R \pa_{\pa F}^2 \bigg( \kappa_F \frac{\xi_{G,F}^2}{2}   \bigg) \, d \mathcal{H}^1.
			\end{split}
		\end{equation}
		It remains to estimate the term $\int_{\pa F} R \pa_{\pa F}^2 \bigg( \kappa_F \frac{\psi_{G,F}^2}{2}   \bigg) \, d \mathcal{H}^1$.
		Using \eqref{iterstimG}, \eqref{0505primacrossfit1}, \eqref{0505primacrossfit2}, \eqref{06052025formpvia}, the Sobolev embedding, we get
		\begin{equation}\label{nnnnchedobali}
			\begin{split}
				\| \pa_{\pa F}^2 \bigg( \kappa_F \frac{\psi_{G,F}^2}{2}   \bigg) \|_{L^2(\pa F)} & \leq C\big(  h^{\frac{1}{2}} (\| \xi_{G,F}\|_{L^2(\pa F)}+ \| \pa_{\pa F} \psi_{G,F}\|_{L^2(\pa F)} ) + h \| \pa_{\pa F}^2 \psi_{G,F} \|_{L^2(\pa F)}\big)\\
				& \leq C h^{\frac{1}{2}}   \big( \| \xi_{G,F} \|_{L^2(\pa F)}+ \| \pa_{\pa F}^2 \psi_{G,F} \|_{L^2(\pa F)}   \big).
			\end{split}
		\end{equation}
		Therefore, using the Cauchy–Schwarz inequality, estimates \eqref{Rinfinutostim}, \eqref{nnnnchedobali}, and Young’s inequality, we obtain
		\begin{equation}\label{finestima0705253}
			\begin{split}
				\int_{\pa F} R \pa_{\pa F}^2 \big( \kappa_F \frac{\psi_{G,F}^2}{2} \big) \, d \mathcal{H}^1 &\leq C\| R \|_{L^2(\pa F)}  h^{\frac{1}{2}}   \big( \| \xi_{G,F} \|_{L^2(\pa F)}+ \| \pa_{\pa F}^2 \psi_{G,F} \|_{L^2(\pa F)}   \big) \\
				&\leq C(\varepsilon) \| \xi_{G,F} \|_{L^2(\pa F)}^2 + \varepsilon \| \pa_{\pa F}^2 \psi_{G,F} \|^2_{L^2(\pa F)}.
			\end{split}
		\end{equation}
		Finally, inserting the estimates \eqref{finestima0705251} for \eqref{05052025form1}, \eqref{finestima0705252} for \eqref{05052025form2}, and \eqref{finestima0705253} for \eqref{05052025form3}, into formula \eqref{finalmento070525}, we obtain
		\begin{equation}
			\bigg(\frac{1}{h}- C(\varepsilon) \bigg) \int_{\pa F} \xi_{G,F}^2 \, d \mathcal{H}^1 + \int_{\pa F} (g(\nu_F)-\varepsilon) \vert \pa_{\pa F}^2 \psi_{G,F} \vert \, d \mathcal{H}^1 \leq \int_{\pa F} \kappa_F^{\varphi} 
			\pa_{\pa F}^2 \xi_{G,F}  \, d \mathcal{H}^1.
		\end{equation}
		Recalling that $g \geq m_\varphi >0$ (see formula \eqref{mMvarphi}), the inequality above implies formula \eqref{07052025tes1lemmaiter} for sufficiently small $h$ and $\varepsilon$. 
	\end{proof}
	
	In the proof of the next proposition, we will use the following well-known inequality, whose proof follows from a classical homogenization argument and the Sobolev embedding of $H^1$ in $L^\infty$. Let $A \in \mathfrak{C}^1_M(E_0)$, for some $M>0$. If $f$ is a smooth function on $\pa A$, then there exists a constant $C(M)$ such that for every $\varepsilon \in (0,1)$,
	\begin{equation}\label{eleminter}
		\| f \|_{L^\infty(\pa A)}^2 \leq C(M) \big( \frac{1}{\varepsilon} \| f \|_{L^2(\pa A)}^2+  \varepsilon \| \pa_{A} f \|_{L^2(\pa A)}^2  \big).
	\end{equation}
	We are now in a position to prove the main result of this section.
	\begin{proposition}[Iteration] \label{propdiiterazione}
		Let $E,F,G$ be as in Lemma \ref{spaghettiska} and, we set
		$$\xi_{F,E}= \psi_{F,E}+ \kappa_E \frac{\psi_{F,E}^2}{2}, \, \,\xi_{G,F}= \psi_{G,F}+ \kappa_G \frac{\psi_{G,F}^2}{2}. $$
		There exist $M,h_4$, depending only on $K$ and $K_{el}$, such that 
		\begin{equation}\label{TDLDITERAZIONE}
			\int_{\pa F} \big( \xi_{G,F}^2+ \frac{h}{2}g(\nu_F) \vert \Delta_{\pa F} \psi_{G,F} \vert^2 \big)\, d \mathcal{H}^1 \leq (1+M h) \int_{\pa E} \big(  \xi_{F,E}^2 + \frac{h}{8}g(\nu_F) \vert \Delta_{\pa E} \psi_{F,E} \vert^2 \big) \, d \mathcal{H}^1  
		\end{equation}
		for $0 < h \leq h_4$.
	\end{proposition}
	\begin{proof}
		In what follows, we denote by $C$  a generic constant depending on $K$ and $K_{el}$. To prove the thesis, we need to estimate the term on the right-hand side of inequality \eqref{07052025tes1lemmaiter}, namely
		\begin{equation}\label{daintegrareperparti}
			h\int_{\pa F} \kappa_F^\varphi \pa_{\pa F}^2 \xi_{G,F}\, d \mathcal{H}^1.
		\end{equation}
		To this end, we consider the diffeomorphism
		\begin{equation}
			\Psi_{F,E}: \pa E \rightarrow \pa F \quad \Psi_{F,E}(x)=x+ \psi_{F,E}(x)\nu_E(x)
		\end{equation}
		and we define
		$$\hat{\kappa}_F^\varphi(x):= \kappa_F^\varphi ( \Psi_{F,E}(x)),\quad \hat{\xi}_{G,F}(x):= \xi_{G,F}( \Psi_{F,E}(x)), \quad \forall x \in \pa E.  $$
		We fix $\varepsilon>0$, to be chosen later. We set 
		$$J_{F,E}:= \sqrt{(1+\kappa_E \psi_{F,E})^2+ \vert \pa_{\pa E} \psi_{F,E} \vert^2}.$$ By integrating by parts in \eqref{daintegrareperparti}, and using formula \eqref{cappelloliberoF}, the Young inequality and the Taylor expansion of the function $t \rightarrow \frac{1}{\sqrt{1+t}}$, we obtain
		\begin{equation}\label{08052025form1}
			\begin{split}
				&h\int_{\pa F} \kappa_F^\varphi \pa_{\pa F}^2 \xi_{G,F}\, d \mathcal{H}^1=- \int_{\pa F} \nabla_{\pa F} \kappa_F^\varphi \cdot \nabla_{\pa F} \xi_{G,F}\, d \mathcal{H}^1= - h\int_{\pa E} \frac{\nabla_{\pa E} \hat{\kappa}_F^\varphi \cdot \nabla_{\pa E} \hat{\xi}_{G,F}}{J_{F,E}}\, d \mathcal{H}^1\\
				& = -h\int_{\pa E} \nabla_{\pa E} \hat{\kappa}_F^\varphi \cdot \nabla_{\pa E} \hat{\xi}_{G,F}\, d \mathcal{H}^1+h \int_{\pa E }\nabla_{\pa E} \hat{\kappa}_F^\varphi \cdot \nabla_{\pa E} \hat{\xi}_{G,F} \bigg( 1- \frac{1}{J_{F,E}}   \bigg) \, d \mathcal{H}^1\\
				& \leq -h\int_{\pa E} \pa_{\pa E} \hat{\kappa}_F^\varphi  \pa_{\pa E} \hat{\xi}_{G,F}\, d \mathcal{H}^1 + \varepsilon h C\int_{\pa F} \vert \pa_{\pa F} \xi_{G,F} \vert^2 \, d \mathcal{H}^1\\
				&  \,\,\,\qquad \qquad + \frac{C}{\varepsilon} h \int_{\pa E} \vert \pa_{\pa E} \hat{\kappa}_{F}^\varphi \vert^2 \big( \psi_{F,E}^2+ \psi_{F,E}^4 + \vert \pa_{\pa E} \psi_{F,E} \vert^4   \big) \, d \mathcal{H}^1 .
			\end{split}
		\end{equation}
		Using \eqref{eleminter}, \eqref{iterstimF}, and the Sobolev embedding, and assuming $h$ is sufficiently small with respect to $\varepsilon$, we estimate the last integral:
		\begin{equation}\label{08052025form2}
			\begin{split}
				\frac{Ch}{\eps}\int_{\pa E }|\pa_{\pa E} \hat{\kappa}_F^\varphi|^2 &(\psi_{F,E}^2+\psi_{F,E}^4+|\pa_{\pa E} \psi_{F,E}|^4) \, d \mathcal{H}^1
				\\
				&\leq \frac{Ch}{\eps}\big(\|\psi_{F,E}\|_{L^\infty(\pa E)}^2+\|\psi_{F,E}\|_{L^\infty(\pa E)}^4+\|\pa_{\pa E}\psi_{F,E}\|_{L^\infty(\pa E)}^4\big)\cr
				& \leq \frac{Ch}{\eps}\Big[\frac{1}{\eps^2}\|\psi_{F,E}\|_{L^2(\pa E)}^2+\eps^2\|\pa_{\pa E}\psi_{F,E}\|_{L^2(\pa E)}^2\\
				& \qquad \qquad \quad+h^{\frac{3}{4}}\Big(\frac{1}{\eps^2}\|\pa_{\pa E}\psi_{F,E}\|_{L^2(\pa E)}^2+\eps^2\|\pa_{\pa E}^2\psi_{F,E}\|_{L^2(\pa E)}^2\Big)\Big]\cr
				& \leq C(\eps) h\|\psi_{F,E}\|_{L^2(\pa E)}^2+Ch\eps(\|\pa_{\pa E}\psi_{F,E}\|_{L^2(\pa E)}^2+\|\pa_{\pa E}^2\psi_{F,E}\|_{L^2(\pa E)}^2)\\
				& \leq C(\varepsilon)h \|\xi_{F,E}\|_{L^2(\pa E)}^2 + Ch \varepsilon \|\pa_{\pa E}^2\psi_{F,E}\|_{L^2(\pa E)}^2,
			\end{split}
		\end{equation}
		where in the last inequality we have used \eqref{050525pranzo1} and \eqref{06052025formpvia}.
		Using the same reasoning as in \eqref{stac'èpriovani}, and recalling \eqref{050525pranzo1}, we obtain
		\begin{equation}\label{08052025form3}
			\| \pa_{\pa F } \xi_{G,F} \|_{L^2(\pa F)}^2 \leq C \big( \| \xi_{G,F} \|_{L^2(\pa F)}^2+ \| \pa_{\pa F}^2 \psi_{G,F} \|_{L^2(\pa F)}^2 \big).
		\end{equation}
		Plugging \eqref{08052025form2} and \eqref{08052025form3} into formula \eqref{08052025form1} and performing integration by parts yields
		\begin{equation}\label{08052025form4}
			\begin{split}
				h\int_{\pa F} \kappa_F^\varphi \pa_{\pa F}^2 \xi_{G,F}\, d \mathcal{H}^1 \leq& h\int_{\pa E} \pa_{\pa E}^2 \hat{\kappa}_F^\varphi   \hat{\xi}_{G,F}\, d \mathcal{H}^1 + \varepsilon h  C \big( \| \xi_{G,F} \|_{L^2(\pa F)}^2+ \| \pa_{\pa F}^2 \psi_{G,F} \|_{L^2(\pa F)}^2 \big) \\
				& \qquad + C(\varepsilon)h \|\xi_{F,E}\|_{L^2(\pa E)}^2 + Ch \varepsilon \|\pa_{\pa E}^2\psi_{F,E}\|_{L^2(\pa E)}^2.
			\end{split}
		\end{equation}
		Recalling that $F$ satisfies the Euler–Lagrange equation \eqref{LEEULEROCOMB}, we have
		\begin{equation}\label{080525euler}
			\pa_{\pa E}^2 \hat{\kappa}_F^\varphi(x)= \frac{1}{h} \xi_{F,E}(x)+ \pa_E^2 \big( Q(E(u_F^{K_{el}}))(x+ \psi_{F,E}(x)\nu_E(x))  \big), \quad x \in \pa E.
		\end{equation}
		We need to estimate $$h\int_{\pa E} \pa_{\pa E}^2 \hat{\kappa}_F^\varphi   \hat{\xi}_{G,F}\, d \mathcal{H}^1.$$ 
		
		Using \eqref{080525euler}, the Cauchy–Schwarz and Young inequalities, and \eqref{ciserviràsicuro}, we obtain
		\begin{equation}\label{05062025sera1}
			\begin{split}
				h \int_{\pa E} & \pa_{\pa E}^2 \hat{\kappa}_F^\varphi   \hat{\xi}_{G,F}\, d \mathcal{H}^1= \int_{\pa E} \xi_{F,E} \hat{\xi}_{G,F}\, d \mathcal{H}^1\\
				& \qquad \qquad\qquad\qquad+ h\int_{\pa E}\pa_E^2 \big( Q(E(u_F^{K_{el}}))(x+ \psi_{F,E}(x)\nu_E(x))  \big) \hat{\xi}_{G,F}(x)\, d \mathcal{H}^1_x\\
				& \leq \frac{1}{2} \int_{\pa E} \xi_{F,E}^2 \, d \mathcal{H}^1+ \frac{1}{2} \int_{\pa E} \hat{\xi}_{G,F}^2 \, d \mathcal{H}^1  \\
				& \qquad+ h \bigg( \int_{\pa E}(\pa_E^2 \big( Q(E(u_F^{K_{el}}))(x+ \psi_{F,E}(x)\nu_E(x))  \big))^2 \, d \mathcal{H}^1_x  \bigg)^{\frac{1}{2}} \bigg( \int_{\pa E} \hat{\xi}_{G,F}^2 \, d \mathcal{H}^1 \bigg)^{\frac{1}{2}}\\
				& \leq \frac{1}{2} \int_{\pa E} \xi_{F,E}^2 \, d \mathcal{H}^1+ \frac{1}{2} \int_{\pa E} \frac{\hat{\xi}_{G,F}^2}{J_{F,E}} \, d \mathcal{H}^1+ \frac{1}{2} \int_{\pa E} \hat{\xi}_{G,F}^2 \bigg(   1- \frac{1}{J_{F,E}}\bigg) \, d \mathcal{H}^1\\
				& \qquad+ Ch \varepsilon \|\pa_{\pa E}^2\psi_{F,E}\|_{L^2(\pa E)}^2+ h C(\varepsilon) \int_{\pa E} \frac{\hat{\xi}_{G,F}^2}{J_{F,E}} \, d \mathcal{H}^1 \\
				& \leq \frac{1}{2} \int_{\pa E} \xi_{F,E}^2 \, d \mathcal{H}^1+ \frac{1}{2} \int_{\pa F} \xi^2_{G,F}\, d \mathcal{H}^1+ \frac{1}{2} \int_{\pa E} \hat{\xi}_{G,F}^2 \kappa_E \psi_{F,E}\, d \mathcal{H}^1\\
				& \qquad + Ch \varepsilon \|\pa_{\pa E}^2\psi_{F,E}\|_{L^2(\pa E)}^2+ h C(\varepsilon) \int_{\pa F} \xi_{G,F}^2 \, d \mathcal{H}^1 +Ch \| \psi_{F,E}\|^2_{L^2(\pa E)}\\
				& \leq \frac{1}{2} \int_{\pa E} \xi_{F,E}^2 \, d \mathcal{H}^1+ \frac{1}{2} \int_{\pa F} \xi^2_{G,F}\, d \mathcal{H}^1+ \frac{1}{2} \int_{\pa E} \hat{\xi}_{G,F}^2 \kappa_E \psi_{F,E}\, d \mathcal{H}^1\\
				& \qquad + Ch \varepsilon \|\pa_{\pa E}^2\psi_{F,E}\|_{L^2(\pa E)}^2+ h C(\varepsilon) \int_{\pa F} \xi_{G,F}^2 \, d \mathcal{H}^1 +Ch \| \psi_{F,E}\|^2_{L^2(\pa E)}.
			\end{split} 
		\end{equation}
		We need to estimate $\frac{1}{2} \int_{\pa E} \hat{\xi}_{G,F}^2 \kappa_E \psi_{F,E}\, d \mathcal{H}^1$. To this aim we recall that
		\begin{equation}
			\left\{
			\begin{aligned}
				& -\Delta_{\pa E} v_{F,E}= \xi_{F,E}= \psi_{F,E}+ \frac{\psi_{F,E}^2}{2}\kappa_E & \text{ on } \pa E , \\
				& \int_{\pa E} v_{F,E} \, d \mathcal{H}^1=0
			\end{aligned}
			\right.
		\end{equation} 
		and $\| \nabla_{\pa E} v_{F,E} \|_{L^2(\pa E)} = d_{H^{-1}}(F,E)\leq Ch$, see formulas \eqref{normH^-1funzv}, \eqref{eqv} and \eqref{2803e'arri}. Therefore we have that
		\begin{equation}\label{05062025sera2}
			\int_{\pa E} \hat{\xi}^2_{G,F}\kappa_E \psi_{F,E} \, d \mathcal{H}^1= -\int_{\pa E} \hat{\xi}^2_{G,F}\kappa_E \frac{\psi_{F,E}^2}{2} \, d \mathcal{H}^1- \int_{\pa E }\hat{\xi}^2_{G,F}\kappa_E\Delta_{\pa E} v_{F,E} \, d \mathcal{H}^1.
		\end{equation}
		Using formula \eqref{iterstimF} we obtain
		\begin{equation}\label{05062025sera3}
			\int_{\pa E} \hat{\xi}^2_{G,F}\kappa_E \frac{\psi_{F,E}^2}{2} \leq h^{\frac{3}{2}}C \int_{\pa E} \hat{\xi}^2_{G,F} \leq h C \int_{\pa F} \xi^2_{G,F} \, d \mathcal{H}^1.
		\end{equation}
		Now by the divergence theorem and using \eqref{2803e'arri}, \eqref{iterstimF} and the Poincarè inequality and the Sobolev embedding we have that
		\begin{equation}\label{05062025sera4}
			\begin{split}
				\int_{\pa E} -\hat{\xi}^2_{G,F}& \kappa_E \Delta_{\pa E} v_{F,E} \, d \mathcal{H}^1\\
				&= \int_{\pa E} \nabla_{\pa E} \hat{\xi}_{G,F} \cdot \nabla_{\pa E} v_{F,E} \kappa_E \hat{\xi}_{G,F}\, d \mathcal{H}^1 + \int_{\pa E}\hat{\xi}^2_{G,F} \nabla_{\pa E} v_{F,E} \cdot \nabla_{\pa E} \kappa_E \, d \mathcal{H}^1\\
				& \leq C \| \nabla_{\pa E} v_{F,E} \|_{L^2(\pa E)} \big( \| \hat{\xi}^2_{G,F} \|_{L^2(\pa E)}+\| \nabla_{\pa E} \hat{\xi}_{G,F} \|^2_{L^2(\pa E)} \big) \\
				& \leq C h\big( \| \hat{\xi}_{G,F}\|_{L^\infty(\pa E)} \| \hat{\xi}_{G,F} \|_{L^2(\pa E)}+   \| \nabla_{\pa F} \xi_{G,F} \|^2_{L^2(\pa F)} \big)\\
				& \leq Ch \big( \| \nabla_{\pa E} \hat{\xi}_{G,F}\|_{L^2(\pa E)} \| \hat{\xi}_{G,F} \|_{L^2(\pa E)}+ \| \xi_{G,F}\|^2_{L^2(\pa F)}+ \varepsilon\| \Delta_{\pa F} \psi_{G,F}\|^2_{L^2(\pa F)}   \big)\\
				& \leq Ch \big( \| \nabla_{\pa E} \hat{\xi}_{G,F}^2 \|_{L^2(\pa E)}^2 + \| \xi_{G,F}\|^2_{L^2(\pa F)}+ \varepsilon\| \Delta_{\pa F} \psi_{G,F}\|^2_{L^2(\pa F)} \big) \\
				& \leq Ch \big( \| \xi_{G,F}\|^2_{L^2(\pa F)}+ \varepsilon\| \Delta_{\pa F} \psi_{G,F}\|^2_{L^2(\pa F)} \big)
			\end{split}
		\end{equation}
		where we have used Lemma \ref{cappelloliberoL} to get \begin{equation}\label{cappelloliberoap}
			\| \nabla_{\pa E} \hat{\xi}_{G,F}  \|^2_{L^2(\pa E)} \leq C(K) \int_{\pa E } \frac{\vert \nabla_{\pa E} \hat{\xi}_{G,F} \vert^2}{J_{F,E}}  \, d \mathcal{H}^1 = C\| \nabla_{\pa F} \xi_{G,F} \|^2_{L^2(\pa F)}  . 
		\end{equation} 
		Using formulas \eqref{05062025sera1}, \eqref{05062025sera2}, \eqref{05062025sera3}, and \eqref{05062025sera4}, we get
		\begin{equation}\label{08052025form5}
			\begin{split}
				h \int_{\pa E}  \pa_{\pa E}^2 \hat{\kappa}_F^\varphi   \hat{\xi}_{G,F}\, d \mathcal{H}^1\leq &\frac{1}{2} \int_{\pa E} \xi_{F,E}^2 \, d \mathcal{H}^1+ \frac{1}{2} \int_{\pa F} \xi^2_{G,F}\, d \mathcal{H}^1+ Ch \varepsilon \|\pa_{\pa E}^2\psi_{F,E}\|_{L^2(\pa E)}^2\\
				& \,\,+ h C(\varepsilon) \int_{\pa F} \xi_{G,F}^2 \, d \mathcal{H}^1 +Ch \| \psi_{F,E}\|^2_{L^2(\pa E)} + \varepsilon h \| \pa^2_{\pa F} \psi_{G,F}\|^2_{L^2(\pa F)}.
			\end{split}
		\end{equation}
		Therefore, combining \eqref{07052025tes1lemmaiter}, \eqref{08052025form4}, and \eqref{08052025form5}, and recalling that $M_\varphi\geq g \geq m_\varphi >0$ (see formula \eqref{mMvarphi}), we conclude:
		\begin{multline}
			\big( \frac{1}{2}-hC  \big)\int_{\pa F} \xi_{G,F}^2 \, d \mathcal{H}^1 +h\big( \frac{3}{4} -\varepsilon C \big)  \int_{\pa F}  g(\nu_F) \vert \pa_{\pa F}^2 \psi_{G,F} \vert^2 \, d \mathcal{H}^1\\
			\leq \big( \frac{1}{2}+hC  \big) \int_{\pa E} \xi_{F,E}^2 \, d \mathcal{H}^1+ h C \varepsilon \int_{\pa E} g(\nu_E) \vert \pa_{\pa E}^2 \psi_{F,E} \vert^2 \,d \mathcal{H}^1.
		\end{multline}
		Choosing $\varepsilon$ and $h$ sufficiently small concludes the proof of \eqref{TDLDITERAZIONE}.
	\end{proof}
	
	\section{Proof of the main theorems}\label{sezionefinale}
	In this section, we use the iteration estimates proved in the previous section to show that the constrained discrete flat flow, defined in \ref{12092023def1}, converge, as $h \rightarrow 0$, to the classical solution of the equation \eqref{MAINEQsol}, provided $K_{el}$ is sufficiently large. We recall that $E_0 \Subset \Omega$ be open and connected set of class $C^5$.

	Here and in the following, we reuse the notation introduced in formula \eqref{minelvinc} and \eqref{enelvinc}. Specifically, we denote by $u_{F}^{K_{el},h}$ 
	a solution to the minimization problem
	\begin{equation}\label{12052025probmin1}
		\min  \left\{ \int_{\Omega \setminus F} Q(E(u))\, dx \colon u \in \mathfrak{C}_{K_{el}}^{3,\frac{1}{4}}(\Omega ,\R^2) ,\, \| \nabla^4 u \|_{C^{0,\frac{1}{4}}(\Omega)} \leq \frac{K_{el}}{h^{\frac{1}{4}}} , u |_{\pa \Omega}= w_0  \right\}
	\end{equation}
	where $ w_0 \in C^{3,\frac{1}{4}}(\partial \Omega)$ is the prescribed boundary displacement.
	We denote by $u_{F}^{K_{el},0}$ a solution to the minimization problem
	\begin{equation}\label{12052025probmin2}
		\min  \left\{ \int_{\Omega \setminus F} Q(E(u))\, dx \colon u \in \mathfrak{C}_{K_{el}}^{3,\frac{1}{4}}(\Omega ,\R^2),\,  u |_{\pa \Omega}= w_0  \right\}
	\end{equation}
	where $w_0$ is as above.
	
	Before proving the first theorem of this section, we establish a lemma that ensures, under suitable assumptions, that the minimizers of problem \eqref{12052025probmin1} converge to those of problem \eqref{12052025probmin2} as $h \rightarrow 0^+$.
	\begin{lemma}\label{lemmafacilespero}
		Let $F_h \Subset \Omega$ be such that $\chi_{F_h} \rightarrow \chi_{F}$ in $L^1(\Omega)$ with $F \Subset \Omega$. Let  $u_{F_h}^{K_{el},h}$ be a minimizer of \eqref{12052025probmin1} for  $h >0$. Then $u_{F_h}^{K_{el},h} \rightarrow u_F^{K_{el},0}$ as $h \rightarrow 0^+$ in $\mathfrak{C}_{K_{el}}^{3,\frac{1}{4}}(\Omega ,\R^2)$ where $u_F^{K_{el},0} $ is a minimizer of \eqref{12052025probmin2}.
	\end{lemma}
	\begin{proof}
		In the proof of the lemma, we will omit explicitly mentioning 'up to subsequences' for the sake of brevity.
		By the Ascoli-Arzelà Theorem we have that $ u_{F_h}^{K_{el},h} \rightarrow u$ in $C^{3,\frac{1}{4}}(\Omega)$.
		Let $ v \in C^4(\Omega)$ with $ \| v \|_{C^{3,\frac{1}{4}}(\Omega)}\leq K_{el} $ then for $h$ sufficiently small we get
		\begin{equation}
			\begin{split}
				\int_{\Omega \setminus F} Q(E(u)) \, d x& = \lim_{h \rightarrow 0^+} \int_{\Omega \setminus F_h} Q(E(u_{F_h}^{K_{el},h})) \, d x \\
				&\leq \lim_{h \rightarrow 0^+}\int_{\Omega \setminus F_h} Q(E(v))\, d x  = \int_{\Omega \setminus F} Q(E(v))\, d x. 
			\end{split}
		\end{equation}
		Therefore by the above formula and using a standard density argument we get the thesis.
	\end{proof}
	\begin{theorem}\label{thmausilmain}
		Let $K_{el}>0$ be fixed. There exist $T_0,C_0,\beta_0,\sigma_1$ with the following property: for every $\beta < \beta_0$ there exists $\tilde{h}$ such that $E^{h,\beta}_t \in \mathfrak{H}^4_{C_0,\sigma_1}(E_0)$, i.e.,
		\begin{equation}\label{minchiachedobali}
			\pa E^{h,\beta}_t= \{x+ f^{h,\beta}(t,x)\nu_{E_0}(x)\colon x \in \pa E_0    \}, \, \| f^{h,\beta} \|_{H^4(\pa E_0)} \leq C_0,\, \| f^{h,\beta} \|_{L^\infty(\pa E_0)} \leq \sigma_1,
		\end{equation}
		for all $t \in [0,T_0]$ and $ 0 < h \leq \tilde{h}$, where $ \{ E^{h,\beta}_t \}_{t \geq0}$ is a discrete constrained flat flow starting from $E_0$.
		
		The function $f^{h,\beta}$ converge in 
		$L^\infty([0,T_0], H^4(\pa E_0))$ to a function $f^{\beta}$ such that the family $ \{ E^\beta_t \}_{t \in [0,T_0]}$ with
		\begin{equation}
			\pa E^\beta_t= \{  x + f^\beta(t,x) \nu_{E_0}(x) : x \in \pa E_0 \}
		\end{equation}
		is a distributional solution of the problem
		\begin{equation}\label{EQ1primasoluzione}
			\left\{
			\begin{aligned}
				& V_t= \pa_{\pa E^\beta_t}^2 \big(   \kappa^\varphi_{E^\beta_t}-Q (E(u_{E^\beta_t}^{K_{el},0}))\big), \text{ on } \pa E_t^\beta  \\
				& E^\beta_0=E_0, \\
				&u_{E^\beta_t}^{K_{el},0} \in \arg \!\min \left\{ \int_{\Omega \setminus E^\beta_t } Q(E(u))\, dx \colon u \in \mathfrak{C}_{K_{el}}^{3,\frac{1}{4}}(\Omega ,\R^2) , \, u |_{\pa \Omega}=w_0    \right\}.
			\end{aligned}
			\right.
		\end{equation}  
		Moreover $ f^\beta\in \mathrm{Lip}([0,T_0],L^2(\pa E_0))$ and
		\begin{equation}\label{l'interfinale}
			\| f^\beta(t,\cdot) \|_{C^{3,\frac{1}{4}}(\pa E_0)} \leq C t^\frac{1}{21} 
		\end{equation}
		where $C=C(K_{el})$.
	\end{theorem}
	\begin{proof}
		In the proof of the theorem, we will omit explicit mention of 'up to subsequences' for the sake of brevity, unless it is strictly necessary for clarity.
		We fix a large constant $K_0= K_0(K_{el},\sigma_0) $, which will be chosen later. Let $\beta_0 < \eta_1$ where $\eta_1$ is the constant from Lemma \ref{spaghettiska}. We fix $\beta \leq \beta_0$. Let $ \{ E_{hk}^{h,\beta}\}_{k \in \N} $ be a constrained discrete flat flow starting from $E_0$; see definition \ref{12092023def1}. To simplify notation, we write $E_k= E_{hk}^{h,\beta}$ for $k \geq 0$. We are now in a position to apply Theorem \ref{MainThm375}, which yields 
		\begin{equation}\label{12052025form-1}
			\begin{split}
				&\pa E_1= \{  x +\psi_1(x)\nu_{E_0}(x)\colon x \in \pa E^0 \},\\
				&\| \psi_1 \|_{L^2(\pa E^0)} \leq L_0 h, \quad \| \psi_1 \|_{H^4(\pa E^0)} \leq L_0, \\
				&  \|  \kappa_{E_1}^{\varphi} \|_{H^2(\pa E_1)} \leq K_0, \quad  \| \pa_{\pa E_1}^3 \kappa_{E_1}^\varphi \|_{L^2(\pa E_1)} \leq \frac{K_0}{h^{\frac{1}{4}}},
			\end{split} 
		\end{equation}
		where $L_0=L_0(K_{el})$. Moreover, using Proposition \ref{PROPINTER}, we have
		\begin{equation}\label{12052025form-1/2}
			\| \pa_{\pa E_0} \psi_1 \|_{L^2(\pa E_0)} \leq L_0 h^{\frac{3}{4}},\quad \| \pa_{\pa E_0}^2 \psi_1 \|_{L^2(\pa E_0)} \leq L_0 h^{\frac{1}{2}},\quad\| \pa_{\pa E_0}^3 \psi_1 \|_{L^2(\pa E_0)} \leq L_0 h^{\frac{1}{4}}.
		\end{equation}
		We denote by $k_0 \in \N$ the largest index such that it holds
		\begin{equation}
			\pa E_k \subset \mathcal{I}_{\beta }(\pa E_0) \quad \forall k \leq k_0.
		\end{equation}
		We set $T_0:= k_0 h$.\\
		\textit{Claim 1:} For every $k \leq k_0$, the following holds:
		\begin{equation}\label{12052025form1}
			\|  \kappa_{E_k}^{\varphi} \|_{H^2(\pa E_k)} \leq K_0, \quad  \| \pa_{\pa E_k}^3 \kappa_{E_k}^\varphi \|_{L^2(\pa E_k)} \leq \frac{K_0}{h^{\frac{1}{4}}}.
		\end{equation}
		
		We prove \eqref{12052025form1} by induction. The base case is verified since the claim holds for $k=1$; see formula \eqref{12052025form-1}. Assume that the claim holds for all integers up to $k-1$. Then, by applying Theorem \eqref{MainThm375}, we obtain 
		\begin{equation}\label{12052025form2}
			\begin{split}
				&\pa E_{k}= \{  x +\psi_{k}(x)\nu_{E_{k-1}} (x)\colon x \in \pa E_{k-1} \},\\
				&\| \psi_{k} \|_{L^2(\pa E_{k-1})} \leq L_1 h, \quad \| \psi_{k} \|_{H^4(\pa E_{k-1})} \leq L_1, 
				%\\
				%&  \|  \kappa_{E^{k}}^{\varphi} \|_{H^2(\pa E^{k})} \leq K_0, \quad  \| \pa_{\pa E^{k}}^3 \kappa_{E^{k}}^\varphi \|_{L^2(\pa E^{k})} \leq \frac{K_0}{h^{\frac{1}{4}}},
			\end{split} 
		\end{equation}
		where $L_1=L_1(K,K_{el})$. For every $j\geq 1$, we set $\xi_j= \xi_{E^j,E^{j-1}}$; see \eqref{lafunzgEFgraph}. Using Proposition \ref{propdiiterazione}, we get
		\begin{multline}\label{12052025form3
			}
			\int_{\pa E_{j-1}} \big( \xi_{j}^2+ \frac{h}{2}g(\nu_{E_{j-1}}) \vert \Delta_{\pa E_{j-1}} \psi_{j} \vert^2 \big)\, d \mathcal{H}^1 \\
			\leq (1+M h) \int_{\pa E_{j-2}} \big(  \xi_{j-1}^2 + \frac{h}{8}g(\nu_{E_{j-2}}) \vert \Delta_{\pa E_{j-2}} \psi_{j-1} \vert^2 \big) \, d \mathcal{H}^1 
		\end{multline}
		for every $ 1 \leq j \leq k$.  We recall (see formula \eqref{fomr07112024}) that 
		\begin{equation}\label{zizzagna}
			\frac{1}{\sqrt{2}}\psi_{j}^2 \leq \xi_{j}^2 \leq \sqrt{2} \psi_{j}^2 \text{ for every } j. 
		\end{equation}
		By iterating the estimate above and using \eqref{12052025form-1} and \eqref{12052025form-1/2}, we obtain
		\begin{equation}
			\begin{split}
				\int_{\pa E_{k-1}} \big( \xi_{k}^2+ \frac{h}{4}\sum_{j=1}^kg(\nu_{E_{j-1}}) \vert \Delta_{\pa E_{j-1}} \psi_{j} \vert^2 \big) & \leq (1+Mh )^{k-1} \int_{\pa E^{0}} \big( \xi_{1}^2+ \frac{h}{2}g(\nu_{E_{0}}) \vert \Delta_{\pa E_{0}} \psi_{1} \vert^2 \big)\\
				& \leq e^{2M h k} L_0^2 h^2 \\
				& \leq e^{2M h k_0} L_0^2 h^2 \leq 2 L_0^2 h^2,
			\end{split}
		\end{equation}
		where we have used $h k_0 = T_0$ and that $T_0$ is small. Possibly increasing the value of $L_0$, and using the above inequality along with \eqref{zizzagna}, we obtain
		\begin{equation}\label{zizzagna2}
			\| \psi_k \|^2_{L^2(\pa E_{k-1})}+ h \sum_{j=1}^k \| \Delta_{\pa E_{k-1}} \psi_k \|^2_{L^2(\pa E_{k-1})} \leq L_0^2 h^2.
		\end{equation}
		Therefore, we can apply Proposition \ref{Mthmproof2} and conclude the proof of the claim \eqref{12052025form1}, possibly after increasing the constant $K_0$.\\
		\textit{Claim 2:} $T_0>0$.
		
		We can assume that for the set $E_{k_0}$, there exists a point $x_0 \in \pa E_{k_0}$ such that $\mathrm{dist}(x_0,E_0)\geq  \frac{\beta}{2}$. The set $E_{k_0}$ satisfies the assumptions of Lemma \ref{lalalambdamin} and therefore it is a $\Lambda$-minimizer of the $\varphi$-perimeter for a $\Lambda$ that is independent of $h$. Consequently, for $E_{k_0}$ the density estimates are satisfied, both for the perimeter and for the volume; see \cite{Bombieri1982}. Using these density estimates together with the inequality $\mathrm{dist}(x_0,E_0)\geq  \frac{\beta}{2}$, we obtain
		\begin{equation}
			\vert E_{k_0} \Delta  E_0 \vert \geq c \beta^2,
		\end{equation}
		where $c$ depends on $\Lambda$.
		Now, using the inequality above together with \eqref{zizzagna}, \eqref{zizzagna2} and the triangular inequality, we derive
		\begin{equation}
			\begin{split}
				c \beta^2 &\leq \vert E_{k_0} \Delta  E_0 \vert \leq \sum_{j=1}^{k_0} \vert E_j \Delta E_{j-1} \vert = \sum_{j=1}^{k_0} \| \xi_j \|_{L^1(\pa E_{j-1})}\leq P(E_{j-1})^{\frac{1}{2}} \sum_{j=1}^{k_0} \| \xi_j \|_{L^2(\pa E_{j-1})}\\
				&\leq C_\varphi P_\varphi(E_{j-1})^{\frac{1}{2}} \sum_{j=1}^{k_0} \| \xi_j \|_{L^2(\pa E_{j-1})} \leq  C_\varphi (P_\varphi(E_0)+K_{el}^2\vert \Omega \vert)^{\frac{1}{2}} \sum_{j=1}^{k_0} \| \psi_j \|_{L^2(\pa E_{j-1})} \\
				& \leq C_\varphi (P_\varphi(E_0)+K_{el}^2\vert \Omega \vert)^{\frac{1}{2}} L_0 k_0 h = C_\varphi(P_\varphi(E_0)+K_{el}^2\vert \Omega \vert)^{\frac{1}{2}} L_0 T_0
			\end{split}
		\end{equation}
		where we have used that $P_\varphi(E_j) \leq K_{el}^2\vert \Omega \vert +P_\varphi (E_0)$, which follows from the minimizing movements scheme.\\
		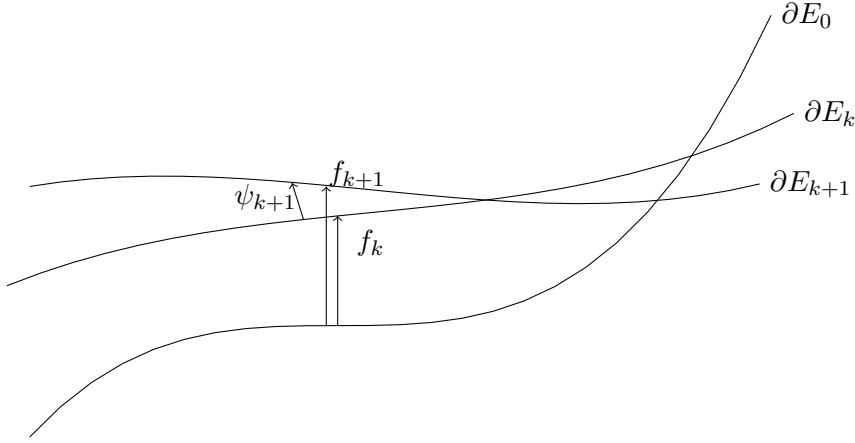
\begin{figure}[t!]
			\centering
			\begin{tikzpicture}[domain=-2:2]
				\draw[scale=1.5,domain=-3.2:3.7,color=black] plot (\x, 0.01*\x*\x*\x+0.1*\x+1) node[right] {$\partial E_{k}$};
				\draw[scale=1.5,domain=-3:3.4,color=black] plot (\x, 0.01*\x*\x*\x-0.1*\x+1.2) node[right] {$\partial E_{k+1}$};
				\draw[scale=1.5,domain=-3:3.5,color=black] plot (\x, 0.05*\x*\x*\x+0.05*0.3*0.3*0.3+3*0.05*0.3*\x*\x+0.05*3*0.3*0.3*\x) node[right] {$\partial E_0$};
				\draw[scale=1.5,->] (-0.3,0) -- (-0.3,0.96973) ;
				\node at ( -0.02,1.1) {$f_{k}$};
				\draw[scale=1.5,->] (-0.4,0) -- (-0.4,1.22973);
				\node at ( -0.2,2) {$f_{k+1}$};
				\draw[scale=1.5,->] (-0.6,0.94) -- (-0.7,1.26) ;
				\node at (-1.4,1.7) {$\psi_{k+1}$};
			\end{tikzpicture}
			\caption{Boundary of $ E_k, E_{k+1}$ and functions $f_k,f_{k+1},\psi_{k+1}$}
		\end{figure}
		\textit{Claim 3:} There exist constants $C_0,\sigma_1>0$ such that 
		\begin{equation}\label{12052025claimlego}
			E_j \in \mathfrak{H}^4_{C_0,\sigma_1}(E_0) \text{ for all } 0 \leq j \leq k_0 . 
		\end{equation}
		
		This claim follows by adapting the arguments from the proof of Theorem \ref{MainThm375}. We provide here a sketch of the proof. As in the previous claim, we may apply Lemma \ref{lalalambdamin},  which implies that each $E_j$ is a $\Lambda$-minimizer of the $\varphi$-perimeter for some $\Lambda$ independent of $h$. Then, using Lemma \ref{propdadim}, we deduce that each $E_j$ is a normal graph over $\pa E_0$, with
		\begin{equation}
			\mathrm{dist}(\pa E_j, \pa E_0) \leq \beta \leq \sigma_1 \text{ for all }j.
		\end{equation}
		If $\sigma_1$ is small enough, we can again apply Lemma \ref{propdadim} to conclude that each $E_j$ is a normal graph over $\pa E_0$. Therefore, there exist functions $f_j: \pa E_0 \rightarrow \R$ such that 
		\begin{equation}
			\pa E_j = \{ x+ f_j(x)\nu_{E_0}(x)\colon x \in \pa E_0  \}.
		\end{equation}
		Moreover, by Lemma \ref{propdadim}, we have that $ \| f_j \|_{C^{1,\gamma}}(\pa E_0) \leq C$ for some $C>0$.
		Using formula \eqref{12052025form1} (the bound $ \| \kappa_{E_j}^{\varphi} \|_{H^2(\pa E^j)}  \leq K_0$) we deduce that $\| f_j \|_{ H^4(\pa E_0)} \leq C_0$. \\
		\textit{Claim 4:} There exists a constant $L_{lip}>0$ such that for all $0 \leq i,k \leq k_0$,
		\begin{equation}\label{12052025claimlego2}
			\| f_i- f_k \|_{L^1(\pa E_0)} \leq L_{lip} h\vert -1+i-k \vert.
		\end{equation}
		
		Without loss of generality, suppose $i < k$.
		The claim follows from the following estimate:
		\begin{equation}
			\begin{split}
				\| f_i- f_k \|_{L^1(\pa E_0)} &\leq \vert E_{k} \Delta  E_i \vert \leq \sum_{j=i+1}^{k} \vert E_j \Delta E_{j-1} \vert = \sum_{j=i+1}^{k} \| \xi_j \|_{L^1(\pa E_{j-1})}\\
				&\leq P(E_{j-1})^{\frac{1}{2}} \sum_{j=i+1}^{k} \| \xi_j \|_{L^2(\pa E_{j-1})}\\
				&\leq \sqrt{2}(K_{el}^2\vert \Omega \vert +P(E_0))^{\frac{1}{2}} \sum_{j=i+1}^{k} \| \psi_j \|_{L^2(\pa E_{j-1})} \\
				& \leq \sqrt{2}(K_{el}^2\vert \Omega \vert +P(E_0))^{\frac{1}{2}}L_0 (k-i-1) h ,
			\end{split}
		\end{equation}
		where we have used \eqref{zizzagna2}.
		
		Hence, combining 
		\eqref{12052025claimlego} and \eqref{12052025claimlego2} and by a standard application of the Ascoli Arzelà Theorem, commonly used in the analysis of minimizing movements,
		we conclude that there exists a subsequence $\{h_m\}_{m \in \N} $ such that
		$f_{h_m}(t) \rightarrow f^\beta(t)$ in $L^1(\pa E_0)$ for a.e. $t \in [0,T_0]$ as $m \rightarrow + \infty$, where
		\begin{equation}\label{diocane-1}
			f^\beta \in \mathrm{Lip} ([0,T_0], L^1(\pa E_0)), \quad f^\beta \in L^\infty ([0,T_0], H^4(\pa E_0)).
		\end{equation}
		in what follows we omit the dependence on $m$ for this subsequence. Therefore by the Sobolev embedding, we get $f^\beta \in L^\infty([0,T_0], C^{3,\frac{1}{2}}(\pa E_0))$. We define the family $ \{ E^\beta_t \}_{t \in [0,T_0]} $ by
		\begin{equation}\label{27052025form2}
			E^\beta_t \Delta E_0 \subset \mathcal{I}_{\sigma_1}(\pa E_0) \text{ and } \pa E^\beta_t := \{  x+  f^\beta(t,x) \nu_{E_0}(x) \colon \, x \in \pa E_0\}.
		\end{equation}
		
		Recall that $u_{E_j}^{K_{el},h}$ is a minimizer of the problem \eqref{12052025probmin1}, with $ u_{E_j}^{K_{el},h} \in \mathfrak{C}_{K_{el}}^{3,\frac{1}{4}}(\Omega ,\R^2)$. 
		%For every $t \in [0,T_0]$, there exists a sequence $\{ j_h \}\subset  \N$ such that $f_{j_h} \rightarrow f^\beta (t) $ in $L^1(\pa E_0)$ as $h \rightarrow 0^+$. 
		Then, up to subsequence, by Lemma \ref{lemmafacilespero}, we obtain
		\begin{equation}
			u_{E_{j_h}}^{K_{el},h} \rightarrow u_{E^\beta_t}^{K_{el},0} \text{ in } \mathfrak{C}_{K_{el}}^{3,\frac{1}{4}}(\Omega ,\R^2) \text{ as } h \rightarrow 0^+,
		\end{equation}
		where $ u^{K_{el},0}_{E^\beta_t}$ is a minimizer of \eqref{12052025probmin2}. Thus, we obtain that 
		\begin{equation}\label{12052025serapgelato}
			u_{E^\beta_t}^{K_{el},0} \in \arg \!\min \left\{ \int_{\Omega \setminus E^\beta(t) } Q(E(u))\, dx \colon u \in \mathfrak{C}_{K_{el}}^{3,\frac{1}{4}}(\Omega ,\R^2) , \, u |_{\pa \Omega}=w_0    \right\}.
		\end{equation}
		\textit{Claim 5:} $E^\beta_t$ is a solution of equation \eqref{EQ1primasoluzione}.
		
		We define the discrete normal velocity on $\pa E_j$ as 
		\begin{equation}
			V_j : \pa E_j \rightarrow \R, \quad V_j := \frac{\psi_{j+1}}{h}.
		\end{equation}
		Let $\Psi_j : \pa E_0 \rightarrow \pa E_j$ be defined by $$\Psi_j(x):= x+ f_j (x)\nu_{E_0}(x) .$$ We recall that 
		$$J_{\tau} \Psi_j(x)= \sqrt{(1+ f_j(x) \kappa_{E_0}(x))^2+\vert \pa_{\pa E_0} f_j(x) \vert^2 } \quad x \in \pa E_0 .$$
		We also define $N_j : \pa E_0 \rightarrow \R^2$ by
		$$  N_j(x):= \frac{-\pa_{\pa E_0}  f_j(x)}{1+ \kappa_{E_0}(x) f_j(x)} \tau_{E_0}(x)+ \nu_{E_0}(x).$$
		We observe that 
		\begin{equation}
			\vert N_j \vert= \frac{ J_{\tau} \Psi_j}{1+ \kappa_{E_0} f_j}.
		\end{equation}
		\textit{Subclaim:} The following holds: 
		\begin{equation}\label{diocaneoverl2}
			\lim_{h \rightarrow 0^+} \|  V_j \circ \Psi_j - \frac{f_{j+1}-f_j}{\vert N_j \vert h} \|_{L^2(\pa E_0)} =0.
		\end{equation}
		
		Using the estimate in \eqref{zizzagna2} we obtain 
		$$\| \psi_{j+1} \circ \Psi_j \|_{C^1(\pa E_0)} \leq C h^{\frac{1}{2}} \text{ and } \| \psi_{j+1} \circ \Psi_j \|_{L^2(\pa E_0)} \leq C h. $$
		From the bounds $ \| f_j \|_{C^{1,\gamma}}(\pa E_0) \leq \varepsilon$ and the previous estimate, we deduce 
		$$ \vert f_{j+1}(x)-f_j(x) \vert \leq C \vert \psi_{j+1} \circ \Psi_j (x)\vert \, \, \forall x \in \pa E_0 \text{ and }  \|f_{j+1}-f_j \|_{C^1(\pa E_0)} \leq C h^{\frac{1}{2}}.$$
		Let $G : \pa  E_0 \rightarrow \R$ be a function such that $ \| G \|_{C^1(\pa E_0)} \leq C h^{\gamma}$ for some $\gamma$. We define 
		$$\Psi_t : \pa E_0 \rightarrow \R^2 , \,\, \Psi_t(x):=x+t \nu_{E_0}(x),  $$ and we recall that $J_{\tau} \Psi_t= 1+ t \kappa_{E_0}$.Applying the coarea formula, we get:
		\begin{equation}\label{13052025primacrossfit}
			\begin{split}
				&\int_{\R^2} G \circ \pi_{\pa E_0}(x) \big(   \chi_{E_{j+1}}(x)- \chi_{E_{j}}(x) \big)\, dx \\
				&= \int_{\pa E_0} G(x) \int_{-\sigma_1}^{\sigma_1} \big(  \chi_{E_{j+1}}(\Psi_t(x))- \chi_{E_{j}}(\Psi_t(x)) \big) (1+ t \kappa_{E_0}(x)\, dt \, d \mathcal{H}^1_x\\
				&= \int_{\pa E_0} G(x) \int_{f_j(x)}^{f_{j+1}(x)} (1+ t \kappa_{E_0}(x)) \, dt \, d \mathcal{H}^1_x\\
				&= \int_{\pa E_0} G(x) (f_{j+1}(x)-f_{j}(x)) (1+ f_j(x)\kappa_{E_0}(x))\, d \mathcal{H}^1_x+ o(h^2)\\
				&= \int_{\pa E_0} G(x) J_{\tau} \Psi_j(x) \frac{f_{j+1}(x)- f_j(x)}{\vert N_j (x) \vert} \, d \mathcal{H}^1_x+ o(h^2).
			\end{split}
		\end{equation}
		We define
		$$ \Phi_{j,t} : \pa E^j \rightarrow \R^2, \,\, \Phi_{j,t}(x):= x+ t \nu_{E_j}(x) , $$
		and we recall $ J_\tau \Phi_{j,t}= 1+ t \kappa_{E_j}$.
		We compute the integral $\int_{\R^2} G \circ \pi_{\pa E_0}(x) \big(   \chi_{E_{j+1}}(x)- \chi_{E_{j}}(x) \big)\, dx$ in a different way from \eqref{13052025primacrossfit}:
		\begin{equation}\label{diocaneover1}
			\begin{split}
				&\int_{\R^2} G \circ \pi_{\pa E_0}(x) \big(  \chi_{E_{j+1}}(x)- \chi_{E_j}(x) \big) \, d x\\
				&= \int_{\pa E_j} \int_{-\beta}^{\beta} G \circ \pi_{\pa E_0}(\Phi_{j,t}(x)) \big(  \chi_{E_{j+1}}(\Phi_{j,t}(x))- \chi_{E^j}(\Phi_{j,t}(x)) \big) (1+ t \kappa_{E^j}(x))\, dt \, d \mathcal{H}^1_x \\
				&= \int_{\pa E_j} \int_{0}^{\psi_{j+1}(x)} G \circ \pi_{\pa E_0} ( \Psi_{j,t}(x)) (1+ t \kappa_{E_j}(x)) \, d t \,d \mathcal{H}^1_x\\
				& = \int_{\pa E_j} \int_{0}^{\psi_{j+1}(x)} \big( G \circ \pi_{\pa E_0} ( \Psi_{j,t}(x)) -G \circ \pi_{\pa E_0}(x)+ G \circ \pi_{\pa E_0}(x)\big)(1+ t \kappa_{E_j}(x)) \, d t \,d \mathcal{H}^1_x \\
				& = \int_{\pa E_j} \psi_{j+1}(x) G \circ \pi_{\pa E_0}(x)  \,d \mathcal{H}^1_x + o(h^2)\\
				& = \int_{\pa E_0} \psi_{j+1} \circ \Psi_j (x) G (x) J_{\tau} \Psi_j(x)\, d\mathcal{H}_x^1+ o(h^2).
			\end{split}
		\end{equation}
		Comparing  \eqref{13052025primacrossfit} and \eqref{diocaneover1} we find that for all $G: \pa E_0 \rightarrow \R$ with $ \|G \|_{C^1(\pa E_0)} \leq Ch^{\gamma}$, it holds true
		\begin{equation}\label{diocaneoverl3}
			\int_{\pa E_0}  G(x) J_{\tau} \Psi_{j}(x) \bigg[ \psi_{j+1}\circ \Psi_j(x)-\frac{f_{j+1}(x)- f_j(x)}{\vert N_j (x) \vert}  \bigg]\, d \mathcal{H}^1_x=              o(h^2).
		\end{equation}
		We define
		\begin{equation}\label{diocaneoverl4}
			G(x):= \frac{1}{J_{\tau} \Psi_{j}(x)}\bigg[ \psi_{j+1}\circ \Psi_j(x)-\frac{f_{j+1}(x)- f_j(x)}{\vert N_j (x) \vert}  \bigg].
		\end{equation}  
		A straightforward computation yields  $ \| G \|_{C^1(\pa E_0) } \leq C h^\gamma$ for some $\gamma \in (0,1)$. Plugging \eqref{diocaneoverl4} into \eqref{diocaneoverl3} gives \eqref{diocaneoverl2}. 
		
		We now return to the main claim: "$E^\beta(t)$ is a solution of \eqref{EQ1primasoluzione}". Up to now, we have established:
		\begin{equation}\label{27052025form1}
			\|f_j \|_{H^4(\pa E_0)} \leq C_0, \, \| f_j \|_{L^\infty(\pa E_0)} \leq \sigma_1 ,\,  \bigg\| \frac{f_{j+1}-f_j}{\vert N_j \vert h} \bigg\|_{L^2(\pa E_0)} \leq C \, \text{ for all } j h \leq T_0 .
		\end{equation}
		Therefore, using \eqref{27052025form1}, along with \eqref{diocaneoverl2} and \eqref{diocane-1} we conclude:
		\begin{equation}\label{19052025fonzi}
			\exists \, L^2(\pa E_0)-\lim_{h \rightarrow 0^+} V_j \circ \Psi_j(\cdot)= \frac{\pa_t f^\beta(t,\cdot)}{\vert N(t,\cdot)\vert },\,\,  \text{ for } t \in [0,T_0],
		\end{equation}
		where $\vert N(t,x) \vert= \frac{J_\tau \Psi_t(x)}{1+ f^\beta(t,x) \kappa_{E_0}(x)}$ and $ \Psi_t(x):= x+ f^\beta(t,x)\nu_{E_0}(x)$ for $x \in \pa E_0$. Let $l \in C^2_c(\R^2)$. Multiplying the Euler–Lagrange equation \eqref{LEEULEROCOMB} by $l$ and integrating by parts yields:
		\begin{equation}\label{diomerda}
			\begin{split}
				\int_{\pa E_j} \frac{\psi_{j+1}(x)}{h} & l(x) \, d\mathcal{H}^1_x+ \int_{\pa E_j} 
				\frac{\psi_{j+1}^2(x)}{2 h}\kappa_{E_j}(x) l(x)  \, d \mathcal{H}^1_x\\
				&= \int_{\pa E_j}  -g(\nu_{E_j}(x))\pa_{\pa E_j}^2 \psi(x) \pa_{\pa E_j}^2 l(x) \, d \mathcal{H}^1_x \\
				&\qquad- \int_{\pa E_j}  Q(E(u_{E_{j+1}}^{K_{el}}))(x+ \psi_{j+1}(x)\nu_{E_j}(x)) \pa_{\pa E_j}^2 l(x) \, d \mathcal{H}^1_x
				\\ 
				&\qquad+ \int_{\pa E_j} R(x) \pa_{\pa E_j}^2 l(x)\, d \mathcal{H}^1_x.
			\end{split}
		\end{equation}
		We observe that
		\begin{equation}\label{diomerda1}
			\begin{split}
				&\lim_{h \rightarrow 0^+} \int_{\pa E_j} 
				\frac{\psi_{j+1}^2(x)}{2 h}\kappa_{E_j}(x) l(x)  \, d \mathcal{H}^1_x \\
				& \qquad\leq \lim_{h \rightarrow 0^+} \| l \|_{L^\infty(\R^2)} \frac{ \| \psi_{j+1} \|_{L^2(\pa E_j)} }{2h} \| \psi_{j+1} \|_{L^\infty(\pa E_j)} \| \kappa_{E_j}\|_{L^2(\pa E_j)} =0, 
			\end{split}
		\end{equation}
		where we have used \eqref{zizzagna2}. Thank the result of the previous claim we also have that $$ \| \psi_{j+1} \|_{H^2(\pa E_j)} \leq C h^\gamma.$$ Recalling the definition of
		$R$ (see formula \eqref{ILRESTO}), we have:
		\begin{equation}\label{diomerda2}
			\| R \|_{L^2(\pa E^j)} \leq C h^\gamma.
		\end{equation}
		Therefore, we can pass to the limit as $h \rightarrow 0^+$, in the equation \eqref{diomerda}, and using \eqref{diomerda1} and \eqref{diomerda2}, we conclude that $E^\beta(t) $ is a distributional solution of \eqref{EQ1primasoluzione} in $[0,T_0]$. Moreover from \eqref{27052025form1} and \eqref{27052025form2}. we get $f^\beta \in \mathrm{Lip}([0,T_0], L^2(\pa E_0))$. \\
		\textit{Claim 6:} It is holds true \eqref{l'interfinale}.
		
		Using formula \eqref{diocane-1} and Proposition \ref{PROPINTER}, we get
		\begin{equation}
			\begin{split}
				\|   \pa_{\pa E_0} f^\beta(t,\cdot) \|_{L^2(\pa E_0)} &\leq  C \| f^\beta(t,\cdot)\|^{\frac{2}{3}}_{L^1(\pa E_0)} \sup_{t \in [0,T_0]}\| f^\beta(t,\cdot)\|^\frac{1}{3}_{H^4(\pa E_0)} \\
				&\leq C t^\frac{2}{3}.
			\end{split}
		\end{equation}
		Using the estimate above, together with \eqref{diocane-1}, \eqref{interHOLDER} and the Sobolev embedding, we deduce
		\begin{equation}
			\begin{split}
				\|  f^\beta(t,\cdot) \|_{C^{3,\frac{1}{4}}(\pa E_0)} & \leq C \big(\sup_{t \in [0,T_0]}\| f^\beta(t,\cdot)\|^\frac{13}{14}_{H^4(\pa E_0)}\big) \| f^\beta(t,\cdot)\|_{C^{0,\frac{1}{2}}(\pa E_0)}^\frac{1}{14}\\
				& \leq C t^\frac{1}{21}.
			\end{split}
		\end{equation}
	\end{proof}
	
	We now recall the statement of Lemma 3.2 from \cite{FJM2018}, which we will use in the next theorem.
	\begin{lemma}
		Let $0 < \alpha < \beta \leq 1$, $M>0$ and $k \in \N$. Then there exists $C>0$ such that for any $ F, \tilde{F} \in \mathfrak{C}_M^{k,\beta}(E_0)$,  the following estimate holds:
		\begin{equation}\label{lemmafjm2018}
			\| u_F (\cdot+ \varphi_F(\cdot) \nu_{E_0}(\cdot))- u_{\tilde{F}} (\cdot+ \varphi_{\tilde{F}}(\cdot) \nu_{E_0}(\cdot)) \|_{C^{k,\alpha}(\pa E_0)} \leq C \| \varphi_F - \varphi_{\tilde{F}} \|_{C^{k,\alpha}(\pa E_0)}
		\end{equation}
		where
		\begin{equation}
			\pa F= \{  x+ \varphi_F(x)\nu_{E_0}(x): \, x \in \pa E_0\}, \, \pa \tilde F= \{  x+ \varphi_{\tilde F}(x)\nu_{E_0}(x): \, x \in \pa E_0\}.
		\end{equation}
	\end{lemma}
	We are now in a position to prove that the minimizing movement scheme converges to the classical solution of \eqref{MAINEQsol}, provided that the initial data $E$ is sufficiently smooth and $K_{el}$ is sufficiently large.
	\begin{theorem}\label{locexisitsol}
		There exist constants $K_{el}$ and $ T_s$ such that any family $\{E^\beta_t\}_{t \in [0,T_0]}$ obtained in Theorem \ref{thmausilmain} is a solution of the problem \eqref{MAINEQsol} in $[0,T_s]$.
	\end{theorem}
	\begin{proof}
		Thanks the regularity of $\pa E_0$ and using classical elliptic regularity theory, see \cite{AgmonDouglisNiremberg1964}, \cite[Proposition 8.9]{FM2012}, we know that the function $u_{E_0}$, which minimizes problem \eqref{minelast} for $F = E_0$, is of class $C^{3,\frac{1}{4}}(\Omega \setminus E_0 )$ and satisfies equation \eqref{eqelliel}. We define the function 
		\begin{equation}\label{estensioneu_E}
			\tilde{u}_{E_0}(x)=\left\{
			\begin{aligned}
				& u_{E_0}(x) \quad \text{ if } x \in \Omega \setminus E_0 , \\
				&  \eta\bigg( \frac{ d_{E_0}(x)}{\sigma_{E_0}}  \bigg)u_{E_0} \big( \pi_{\pa E_0}(x)   \big) \quad \text{ if } x \in E_{0} \cap \mathcal{I}_{\sigma_{E_0}}(\pa E_0),\\
				& 0 \quad \text{ if } x \in E_0 \setminus \mathcal{I}_{\sigma_{E_0}}(\pa E_0),
			\end{aligned}
			\right.
		\end{equation}  
		where $ \eta \in C^{\infty}_c (-2,2)$, and $ \eta \geq 0$, $ \eta =1$ in $ (-1,1)$. 
		We observe that
		\begin{equation}\label{sisisinonono1}
			\| \tilde{u}_{E_0} \|_{C^{3,\frac{1}{4}}(\Omega)} \leq C_1 \| u_{E_0} \|_{C^{3,\frac{1}{4}}(\Omega \setminus E_0)}, 
		\end{equation}
		where $C_1=C(\| \kappa_{E_0} \|_{C^2(\pa E_0)})$.
		Moreover, since $u_{E_0}$ solves the equation \eqref{eqelliel}, we get
		\begin{equation}\label{sisisinonono2}
			\| u_{E_0} \|_{C^{3,\frac{1}{4}}(\Omega \setminus E_0)} \leq L (\| w_0 \|_{C^{3,\frac{1}{4}}(\pa \Omega)}+ \| u_{E_0}  \|_{C^{3,\frac{1}{4}}(\pa E_0)}),
		\end{equation}
		where $L $ is a universal constant.
		We define the constant $K_{el}$ as 
		\begin{equation}\label{sisisinonono2.5}
			K_{el}:= 2 \max\{1,C_1 L (\| w_0 \|_{C^{3,\frac{1}{4}}(\pa \Omega)}+ \| u_{E_0} \|_{C^{3,\frac{1}{4}}(\pa E_0)})\}.
		\end{equation} 
		Therefore, $\tilde{u}_{E_0}$ is a minimizer of the problem \eqref{12052025probmin2} with this choice of $K_{el}$, and satisfies
		$$ \| \tilde{u}_{E_0} \|_{C^{3,\frac{1}{4}}(\Omega)} < K_{el}.$$ Let $f^\beta$ be the function obtained in Theorem \ref{thmausilmain}, and let $T_0$ be the corresponding time from the same theorem. Recall that $f^\beta(0,\cdot)= 0$. By combining formulas \eqref{l'interfinale} and \eqref{lemmafjm2018} (with $F= E_0$ and $\tilde{F}= E^\beta_t$), we get, for all $t \in [0,T_0]$,
		\begin{equation}\label{sisisinonono3}
			\| u_{E_0} - u_{E^\beta_t} (\cdot+ f^\beta(t,\cdot) \nu_{E_0}(\cdot)) \|_{C^{3,\frac{1}{4}}(\pa E_0)} \leq C t^\frac{1}{21},
		\end{equation}
		where $u_{E^\beta_t}$ is a minimizer of \eqref{minelast} for $F= E^\beta_t$.\\
		\textit{Claim:} There exits $T_s$ such that for all $ 0 \leq t \leq T_s$, the function \begin{equation}\label{estensioneu_E^beta}
			\tilde{u}_{E^\beta_t}(x)=\left\{
			\begin{aligned}
				& u_{E^\beta_t}(x) \text{ if } x \in \Omega \setminus {E^\beta_t} , \\
				&  \eta\bigg( \frac{ d_{E_0}(x)}{\sigma_{E_0}}  \bigg)u_{E^\beta_t} \bigg( \pi_{\pa E_0}(x)+ f^\beta(t,x)\nu_{E_0}(\pi_{E_0}(x))   \bigg) \text{ if } x \in E^{E^\beta(t)}
			\end{aligned}
			\right.
		\end{equation}
		is a minimizer of \eqref{12052025probmin2} for $F= E^\beta_t$ and satisfies $ \| \tilde{u}_{E^\beta_t} \|_{C^{3,\frac{1}{4}}(\Omega)} < K_{el}$.
		
		As in formulas \eqref{sisisinonono1} and \eqref{sisisinonono2}, we obtain
		\begin{equation}
			\| \tilde{u}_{E^\beta_t} \|_{C^{3,\frac{1}{4}}(\Omega)} \leq C_1  L (\| w_0 \|_{C^{3,\frac{1}{4}}(\pa \Omega)}+ \| u_{E^\beta_t} (\cdot+ f^\beta(t,\cdot)\nu_{E_0}(\cdot)) \|_{C^{3,\frac{1}{4}}(\pa E_0)}).
		\end{equation}
		Using this and \eqref{sisisinonono3}, we conclude that
		\begin{equation}
			\begin{split}
				\| \tilde{u}_{E^\beta_t} \|_{C^{3,\frac{1}{4}}(\Omega)} &\leq C_1  L (\| w_0 \|_{C^{3,\frac{1}{4}}(\pa \Omega)}+ \| u_{E_0}  \|_{C^{3,\frac{1}{4}}(\pa E_0)})+C t^{\frac{1}{21}}\\
				&< K_{el},
			\end{split}
		\end{equation}
		for $t$ sufficiently small. By the minimality of $u_{E^\beta_t}$ in \eqref{minelast}, we get that $\tilde{u}_{E^\beta_t}$ is a minimizer of \eqref{12052025probmin2}.
		
		Therefore, there exist constants $K_{el},T_s$ such that the family $\{E^\beta_t\}_{t \in [0,T_s]}$ satisfies \eqref{EQ1primasoluzione}. Moreover, the minimizer $\tilde{u}_{E^\beta_t}$ of \eqref{12052025probmin2} satisfies $ \|\tilde{u}_{E^\beta_t} \|_{C^{3,\frac{1}{4}}(\Omega)}< K_{el} $ and 
		$$\tilde{u}_{E^\beta_t} (x) = {u}_{E^\beta_t}(x) \text{ for all } x \in \Omega \setminus E^\beta_t,$$
		where $u_{E^\beta_t}$ is a minimizer of \eqref{minelast}. Hence, we conclude that the family $\{E^\beta_t\}_{t \in [0,T_s]}$, parametrized by the diffeomorphisms $\Phi_t(x)= x+f^\beta(t,x)\nu_{E_0}$, constitutes a strong solution to the anisotropic surface diffusion equation with elasticity. More precisely, using the expansion of the curvature from \eqref{L'ESPANSOIONE} and the expansion of the Laplace–Beltrami operator as in \cite{FJM2018}, we find that the function $f^\beta : [0,T_s] \times \pa E_0 \rightarrow \R$ is a strong solution to the equation (see formulas $(3.6),(3.32)$, and $(3.38)$ in \cite{FJM2018})
		\begin{equation}\label{EQFJM2018}
			\left\{
			\begin{aligned}
				& \pa_t f^\beta= \frac{1}{1+f^\beta \kappa_{E_0}} \pa_\tau \bigg( \frac{\pa_\tau \big( \big( g(\nu_{E^\beta_t}) \kappa_{E^\beta_t} - Q(E(u_{E^\beta_t}     ) ) \big)\circ \pi^{-1}_{\pa E^\beta_t}  \big)}{\sqrt{(1+f^\beta \kappa_{E_0})^2+\vert \pa_\tau f^\beta \vert^2}} \bigg), \text{ on } \pa E_0
				\\
				& f^\beta(0,\cdot)=0 \text{ on } \pa E_0.
			\end{aligned}
			\right.
		\end{equation} 
		By a strong solution, we mean that $f^\beta \in \mathrm{Lip}([0,T_s], L^2(\pa E_0)) \cap L^\infty ([0,T_s], H^4(\pa E_0))$ and that it satisfies equation \eqref{EQFJM2018} almost everywhere.
		Using Gr\"onwall's lemma, one can deduce that the strong solution to \eqref{EQFJM2018} with zero initial data is unique. This implies that the limiting flat flow coincides with the classical solution of \eqref{MAINEQsol} on the interval $[0,T_s]$.
	\end{proof}
	We recall the definition of uniform ball condition.
	\begin{definition}
		We say that a set $E\subset \R^2$ satisfies the uniform ball condition (UBC) with a given
		radius $r>0$ if for every $x \in \pa E$ there are balls $B_r(x_{+})$ and $B_r(x_{-})$ such that
		\begin{equation}
			B_r(x_{+}) \subset \R^2 \setminus E, \quad B_r(x_{-}) \subset E \text{ and } x \in \pa B_r(x_{+}) \cap \pa B_r(x_{-}).
		\end{equation}
	\end{definition}

	\begin{remark}
		\label{rem:thm2}
		We may quantify the statement of Theorem  \ref{locexisitsol} as follows 
		: Let $E_0\Subset \Omega$ be a open connected set of class $C^5$ and that satisfies the UBC with radius $2r_0$,
		and the heightfunction $ \psi_1$, see Theorem \ref{thmausilmain}, satisfies 
		\[
		\|\psi_1 \|_{L^2(\pa E_0)} \leq L_0 h \quad \text{and} \quad \|\Delta_{\pa E_0}\psi_1 \|_{L^2(\Sigma_0)} \leq  L_0\sqrt{h}.
		\]
		Then there is  $K_0 = K_0(r_0,L_0)$ and $\beta_0=\beta_0(r_0,K_0)$ such that if
		\[
		\|\kappa_{ E_0}^\varphi  \|_{H^2(\pa E_0)} \leq K_0  \quad \text{and} \quad \| \nabla_{\pa E_0} \Delta_{\pa E_0}  \kappa_{\pa E_0}^\varphi \|_{L^2(\pa E_0)} \leq  K_0 h^{-\frac14}
		\]
		then  the discrete constrained flat flow $\{E^{h,\beta}_t\}_{t\geq 0}$, where $\beta \leq \beta_0$, also satisfies the UBC with radius $r_0$, and 
		\[
		\|\kappa_{E^{h,\beta}_t }^\varphi \|_{H^2(\pa E^{h,\beta}_t)} \leq  K_0  \quad \text{and} \quad \| \nabla_{\pa E^{h,\beta}_t}\Delta_{\pa E^{h,\beta}_t} \kappa^\varphi_{E^{h,\beta}_t} \|_{L^2(\pa E^{h,\beta}_t)} \leq K_0 h^{-\frac14}
		\]
		for all $t \in [0,T_s]$, where $T_s = T_s(r_0,K_0)$.  
	\end{remark}

	\begin{remark}
		\label{rem:uniqueness-strong}
		The arguments in the proofs of Theorems \ref{locexisitsol} and \ref{thmausilmain} imply that if a constrained discrete flat flow $\{ E^{h, \beta}_t   \}_{t \geq 0}$, starting from $E_0$,  satisfies 
		\[
		\|\kappa_{E^{h,\beta}_t }^\varphi \|_{H^2(\pa E^{h,\beta}_t)} \leq  K_0  \quad \text{and} \quad \| \nabla_{\pa E^{h,\beta}_t}\Delta_{\pa E^{h,\beta}_t} \kappa^\varphi_{E^{h,\beta}_t} \|_{L^2(\pa E^{h,\beta}_t)} \leq K_0 h^{-\frac14}
		\]
		for all $t \in [0,T_s]$, then the limiting flat flow coincides with the classical solution on the time interval $[0,T_s]$.  
	\end{remark}
	
	\section{Convergence to the global solution}\label{CONVGLOBSOL}
	We recall that the classical solution of \eqref{MAINEQsol} with initial datum $E_0$ exists on the interval $[0,T_e)$, where $T_e$denotes the maximal existence time. In this subsection, we prove that for every $T< T_e$, there exist $\beta(T)$ and $K_{el}(T)$ such that the constrained discrete flat flow with initial datum $E_0$ converges to the classical solution of \eqref{MAINEQsol} on $[0,T]$ as $h \rightarrow 0^+$. The proof of the next theorem is similar to the one presented in \cite{CFJKsd}[Theorem 1.2], but we include it here for the reader’s convenience.
	\begin{theorem}\label{03thmfinaleGLOB}
		Let $\{E_t\}_{t \in [0,T_e)}$ be a classical solution of \eqref{MAINEQsol} with initial datum $E_0$. Then for every $T < T_e$ there exist $\beta(T)$ and $K_{el}(T)$ such that for all $\beta \in (0,\beta(T)]$ the constrained flat flow  $E^\beta_t$, starting from $E_0$, coincide with $E_t$ in $[0,T]$.    
	\end{theorem}
	\begin{proof}
		Let $\{E_t\}_{t \in [0,T_e)}$ be the classical solution of \eqref{MAINEQsol}, and let $ T < T_e$ be fixed. Since the classical solution is regular on $[0,T]$, there exist constants $ K_2,\sigma_2$, and $\tilde K_{el}$ such that 
		\begin{equation}\label{lamerdainculo}
			E_t \in \mathfrak{H}^4_{K_2,\sigma_2}(E_0),\, \| \tilde{u}_{E_t}\|_{C^3(\Omega)} < \tilde K_{el}\,\, \text{ for all } t \in [0,T],
		\end{equation}
		where $\tilde{u}_{E_t}$ is the function defined by \eqref{12052025probmin2}, where we replaced $E^\beta_t$ with $E_t$. It is easy to check that the condition $E_t \in \mathfrak{H}^4_{K_2,\sigma_2}(E_0) $ implies that there exists $r_0>0$ such that $E_t $ satisfies the UBC with $r_0$.
		Let $\beta_0,T_s$ be the constants obtained in Theorems \ref{thmausilmain}, \ref{locexisitsol} and Remark \ref{rem:thm2} for 
		$$K_{el}=4 \tilde{K}_{el}. $$
		We fix $\beta\leq  \beta_0$. 
		Let $k_0 \in \N$ be such that $T_0 \in [hk_0, h(k_0+1))$, and let $(E^{h,\beta}_{hk})_{k \in \N}$ be a discrete constrained flat flow starting from $E_0$. As in Theorem \ref{thmausilmain}, we have
		\begin{equation}
			\begin{split}
				&\pa E^{h,\beta}_{h}= \{  x +\psi_1(x)\nu_{E^0}(x)\colon x \in \pa E^0 \},\\
				&\| \psi_1 \|_{L^2(\pa E^0)} \leq L_0 h, \quad \| \psi_1 \|_{H^4(\pa E^0)} \leq L_0, 
				\\
				&  \|  \kappa_{E_h^{h,\beta}}^{\varphi} \|_{H^2(\pa E_h^{h,\beta})} \leq K_0, \quad  \| \pa_{\pa E^{h,\beta}}^3 \kappa_{E_h^{h,\beta}}^\varphi \|_{L^2(\pa E_h^{h,\beta})} \leq \frac{K_0}{h^{\frac{1}{4}}},
			\end{split} 
		\end{equation}
		where $L_0$ and $K_0$ are as in \eqref{12052025form-1}, with $K_0 > K_{el}$. 
		We adopt the following notation: $E^h_t:= E^{h,\beta}(t)$, $E_k:= E^{h,\beta}_{hk}$, and recall \eqref{agligderigligb} for the definition of $ E^{h,\beta}_t$.
		The conclusion of the theorem follows from the next claim, together with Remarks \ref{rem:thm2} and \ref{rem:uniqueness-strong}. 
		\\
		\textit{Claim:} For evert $t \in [0,T]$
		\begin{equation}\label{19052025ancora1}
			\|\kappa_{E^h_t}^{\varphi} \|_{H^2(\pa E^h_t )} \leq K_0, \quad  \| \pa_{\pa E^h_t}^3 \kappa_{E^h_t}^\varphi \|_{L^2(\pa E^h_t )} \leq \frac{K_0}{h^{\frac{1}{4}}}.
		\end{equation}
		
		By Theorem \ref{thmausilmain}, estimate \eqref{19052025ancora1} holds for all $t \in [0,T_s]$. We define
		\begin{equation}\label{19052025ttilde}
			\tilde{t}: = \inf \{   t \in [T_s, T]: \text{ formula \eqref{19052025ancora1} is true for all }t \in[0,\tilde{t}]  \} .
		\end{equation}
		We will show that \eqref{19052025ancora1} continues to hold for all $t \leq \tilde{t}+ \frac{T_s}{2}$, which implies the claim. To this end, let $\tilde{k} \in \N$ be such that $ \tilde{t}-\frac{T_s}{2} \in [h \tilde{k}, (\tilde{k}+1)h  )$ satisfies \eqref{19052025ancora1}, we apply Theorems \ref{thmausilmain} and \ref{locexisitsol} with $E_0= E_{\tilde{k}}$ to obtain there exist $k_1 \in \N$ and $c>0$ (we recall that $c=c(K_{el}, K_0)$) such that $0<c \leq   h k_1=T_1 \leq  T_s$, and for all $ k \in \{ \tilde{k}, \dots, \tilde{k}+k_1 \}$
		\begin{equation}
			\pa E_k= \{  x+ \psi_k(x)\nu_{E_{k-1}}(x)\colon x \in \pa E_{k-1}\}.
		\end{equation}
		Using formula \eqref{zizzagna2}, we obtain
		\begin{equation}
			\| \psi_k \|^2_{L^2(\pa E_{k-1})}+ h \sum_{j=\tilde{k}}^{\tilde{k}+k_1} \| \Delta_{\pa E_{k-1}} \psi_k \|^2_{L^2(\pa E_{k-1})} \leq C h^2,  
		\end{equation}
		for some constant $C$. Since  $0 <c \leq hk_1=T_1$, there exists $\hat{k} \in \{  \tilde{k}, \dots, \tilde{k}+ \lfloor \frac{k_1}{4} \rfloor \}$ such that
		\begin{equation}\label{19052025pom1}
			\| \psi_{\hat{k}} \|_{L^2(\pa E_{\hat{k}-1})}+  \|   \Delta_{\pa E_{\hat{k}-1} } \psi_{\hat{k}} \|_{L^2(\pa E_{\hat{k}-1})} \leq C h.
		\end{equation}
		From the above and using the very definition of $\tilde{t}$, see\eqref{19052025ttilde}, we get
		$$h\hat{k} \leq h (\tilde{k}+ \lfloor \frac{k_1}{4} \rfloor ) \leq \tilde{t}.$$
		Recall that in the minimizing movement scheme, each set $E_j$  is of class $C^5$, since it solves the Euler–Lagrange equation \eqref{LEEULERO}. 
		Moreover, $E_{\hat{k}}$ is uniformly $C^{3,\frac{1}{2}}$-regular.  Let $t_h= \hat{k}h$ and we set
		\begin{equation}
			v^h(t_h,x)= \frac{\psi_{\hat{k}}(x) }{h}.
		\end{equation}
		By \eqref{19052025pom1} and the Sobolev embedding theorem, we obtain
		\begin{equation}\label{19052025pom2}
			\| v^h(t_h,\cdot)\|_{C^{1,\frac{1}{2}}(\pa E_{\hat{k}-1})} \leq C.
		\end{equation}
		Since $t_h = \hat{k}h \in [\tilde{t}- \frac{T_s}{2}, \tilde{t}]$, by passing to a subsequence if necessary, we can assume \begin{equation}\label{19052025pom3}
			\exists \lim_{h \rightarrow 0^+} t_h= \bar{t} .
		\end{equation}
		From \eqref{19052025pom2} and \eqref{19052025pom3}, we conclude
		\begin{equation}
			v^h(t_h,\cdot) \rightarrow v(\bar t, \cdot) \, \text{in } C^{1,\frac{1}{2}} \text{ as } h \rightarrow 0^+, \, \| v(\bar t, \cdot )\|_{C^{1,\frac{1}{2}}(\pa E^\beta(\bar t))} \leq C.
		\end{equation}
		Hence,
		\begin{equation}
			\lim_{h \rightarrow 0^+} \| v^h(t_h,\cdot)\|_{L^2(\pa E_{\hat{k}-1})}= \| v(\bar t, \cdot )\|_{L^2(\pa E^\beta_{\bar t}   )}.
		\end{equation}
		Since we assumed that \eqref{19052025ancora1} holds for all $t \leq \tilde{t}$,
		and $\bar t \leq \tilde{t}$, Remark \ref{rem:uniqueness-strong} implies that the flat flow agrees with the classical solution up to time $\tilde{t}$.
		Using \eqref{19052025fonzi} and \eqref{diocaneoverl2} with $E(\tilde{t}- \frac{T_0}{2})$ in place of $E_0$, we find that $v(\bar t, \cdot)$ coincides with the normal velocity $V_{\bar t}$ of the classical solution $\{ E_t\}_{t\geq 0}$, and
		\begin{equation}
			\| V_{\bar t} \|_{L^2(\pa E_{\bar t})} = \| v(\bar t,\cdot) \|_{L^2(\pa E^\beta_{\bar t} )} . 
		\end{equation}
		By the definition of $K_{el}$ ($K_{el}= 4 \max \{K_2, \tilde K_{el}  \}$) and using equation \eqref{MAINEQsol}, we get
		\begin{equation}
			\| V_{\hat t} \|_{L^2(\pa E_{ \hat t})} = \| \Delta_{\pa E_t } \big(   \kappa^\varphi_{E_{\hat t}}-Q (E(u_{E_{\hat t}}))\big) \|_{L^2(\pa E_{\hat t}  )} \leq K_2+ \tilde{K}_{el} \leq \frac{K_{el}}{2} .
		\end{equation}
		Hence, we conclude \begin{equation}
			\| \psi_{k}\|_{L^2(\pa E_{k-1})} \leq \frac{K_{el}}{2} h.
		\end{equation}
		Using \eqref{19052025pom1}, we also have
		\begin{equation}
			\| \Delta_{\pa E_{k-1}}\psi_{k}\|_{L^2(\pa E_{k-1})} \leq  C h \leq K_{el} \sqrt{h}
		\end{equation}
		for $h$ small enough. Finally, since $E_{\bar t}= E^\beta_{\bar t}$ is uniformly $C^{3,\frac{1}{2}}$ regular with bound $C$, the same holds for $E_{\hat{k}}$ is uniformly $C^{3,\frac{1}{2}}$,  with a bound $2C$. Then, applying Remarks \ref{rem:thm2} and \ref{rem:uniqueness-strong} with $E_{\hat{k}}$ instead of $E_0$, we deduce that \eqref{19052025ancora1} holds on $[\hat{k}h,\hat{k}h+T_s],$. Since $ \hat{k}h \in [\tilde{t}-\frac{T_s}{2}, \tilde{t}]$, this implies that  $E^h_t$ satisfies \eqref{19052025ancora1} on $[0,\tilde{t}+ \frac{T_s}{2}]$. Repeating this argument a finite number of times yields the claim.
	\end{proof}

	\section*{Aknowledgments} 
	I was supported by the Academy of Finland grant 314227.

\end{document}